\documentclass[a4paper,10pt,reqno]{amsart}
\usepackage{thmtools,enumerate,graphicx,xcolor}
\usepackage{tcolorbox}
\usepackage{etex,float}
\usepackage{ucs}
\usepackage{amsmath, amssymb, amscd, tabu, diagbox}
\usepackage{tikz}
\usetikzlibrary{arrows,chains,matrix,positioning,scopes}
\usepackage{mathtools}
\usepackage{sseq}
\usepackage{stmaryrd}
\usepackage[latin1]{inputenc}
\usepackage{amsmath}
\usepackage{amssymb}
\usepackage{epsfig}
\usepackage{psfrag}

\usepackage{tikz-cd}
\makeatletter
\def\@tocline#1#2#3#4#5#6#7{\relax
  \ifnum #1>\c@tocdepth % then omit
  \else
    \par \addpenalty\@secpenalty\addvspace{#2}%
    \begingroup \hyphenpenalty\@M
    \@ifempty{#4}{%
      \@tempdima\csname r@tocindent\number#1\endcsname\relax
    }{%
      \@tempdima#4\relax
    }%
    \parindent\z@ \leftskip#3\relax \advance\leftskip\@tempdima\relax
    \rightskip\@pnumwidth plus4em \parfillskip-\@pnumwidth
    #5\leavevmode\hskip-\@tempdima
      \ifcase #1
       \or\or \hskip 2em \or \hskip 2em \else \hskip 3em \fi%
      #6\nobreak\relax
    \hfill\hbox to\@pnumwidth{\@tocpagenum{#7}}\par
    \nobreak
    \endgroup
  \fi}
\makeatother

\makeatletter
\newcommand{\DrawLine}{%
	\begin{tikzpicture}
	\path[use as bounding box] (0,0) -- (\linewidth,0);
	\draw[%color=red!75!black,
	dashed,dash phase=2pt]
	(0-\kvtcb@leftlower-\kvtcb@boxsep,0)--
	(\linewidth+\kvtcb@rightlower+\kvtcb@boxsep,0);
	\end{tikzpicture}%
}
\makeatother

\usepackage[all]{xy}

\usepackage{graphicx}
\usepackage{tabularx}
\usepackage{overpic}
\usepackage{tikz}
\usetikzlibrary{arrows}
\usepackage{rotating}
\usepackage{framed}
\usepackage{xcolor}
\usepackage{placeins}
\usepackage{booktabs}
\usepackage{pinlabel}
 \usepackage{picins}\usepackage{wrapfig}\usepackage{booktabs}\usepackage[font=footnotesize,labelfont=bf]{caption}
\usepackage{placeins}

\usepackage{color}
\usepackage{multirow}
\usepackage{overpic}
\usepackage{colortbl}
\usepackage{geometry}

\makeatletter
\newcommand{\LeftEqNo}{\let\veqno\@@leqno}
\makeatother
\newlength{\lowerhalftmp}

\usepackage{microtype}

\newtheorem{theorem}{Theorem}[section]
\newtheorem{definition}[theorem]{Definition}
\newtheorem{lemma}[theorem]{Lemma}

\newtheorem{proposition}[theorem]{Proposition}

\theoremstyle{definition}
\newtheorem{remark}[theorem]{Remark}
\newtheorem{question}[theorem]{Question}
\renewcommand{\epsilon}{\varepsilon}

\DeclareMathAlphabet{\mathpzc}{OT1}{pzc}{m}{it}
\usepackage{mathrsfs}

\newcommand{\dtot}{\partial}

\newcommand{\Z}{\mathbb{Z}}

\newcommand{\bF}{\mathbb{F}}

\newcommand{\R}{\mathbb{R}}

\renewcommand{\qed}{$\hfill \square$ \smallskip \\ }

\newcommand{\id}{\text{id}}

\newcommand{\tqft}{\mathfrak{T}}

\newcommand{\G}{\mathcal{G}}

\newcommand{\F}{\mathcal{F}}
\newcommand{\LL}{\mathcal{L}}
\newcommand{\Kf}{\mathcal{K}}

\newcommand{\Kh}{\operatorname{Kh}}
\newcommand{\Khred}{\widetilde{\operatorname{Kh}}{}}
\newcommand{\CKh}{\operatorname{CKh}}
\newcommand{\CKhred}{\widetilde{\operatorname{CKh}}}
\newcommand{\CKhtred}{\widetilde{\operatorname{CKh}}{}_\tau}
\newcommand{\Kht}{\operatorname{Kh}_\tau}
\newcommand{\Khredt}{\widetilde{\operatorname{Kh}}{}_\tau}
\newcommand{\Htredj}{\widetilde{\bf H}_\tau^j}
\newcommand{\CKht}{\operatorname{CKh}_\tau}
\newcommand{\Ht}{\mathbf{H}_\tau}

\def\co{\colon\thinspace}

\newcommand{\cobIIOOIOIOIIOO}
	{\raisebox{-2pt}{\includegraphics[scale=0.1]{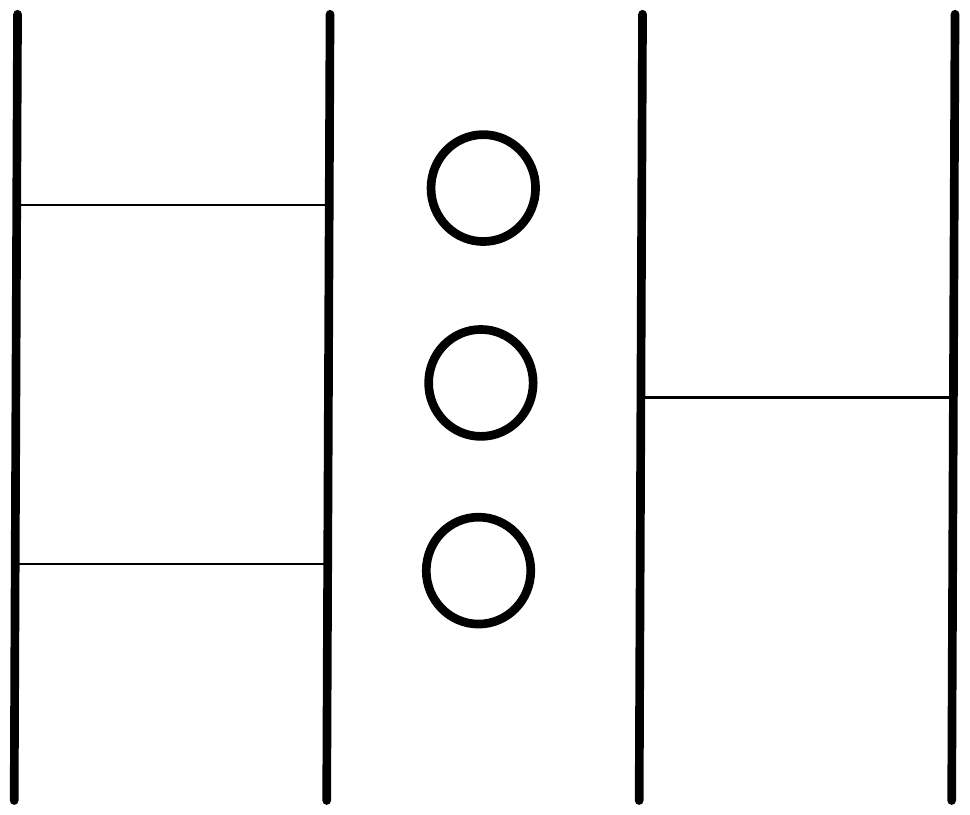}}}

\newcommand{\cobOIOOIIOOIIOI}
{\raisebox{-2pt}{\includegraphics[scale=0.1]{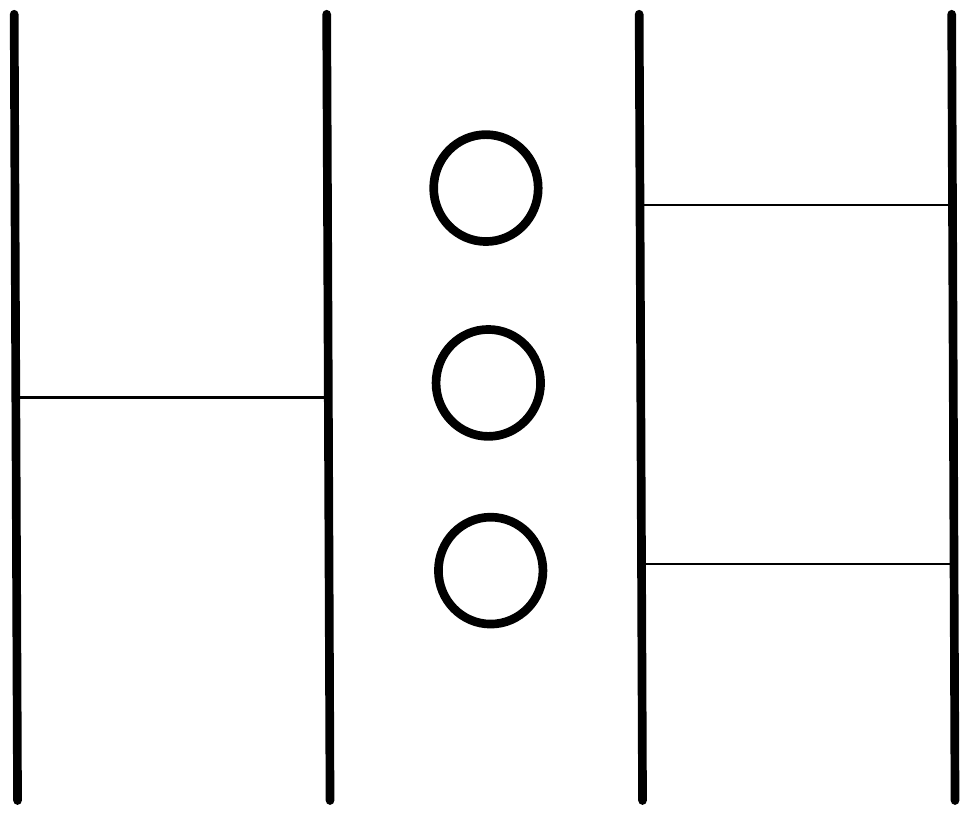}}}

\newcommand{\cobOOIIOOIIOOII}
{\raisebox{-2pt}{\includegraphics[scale=0.1]{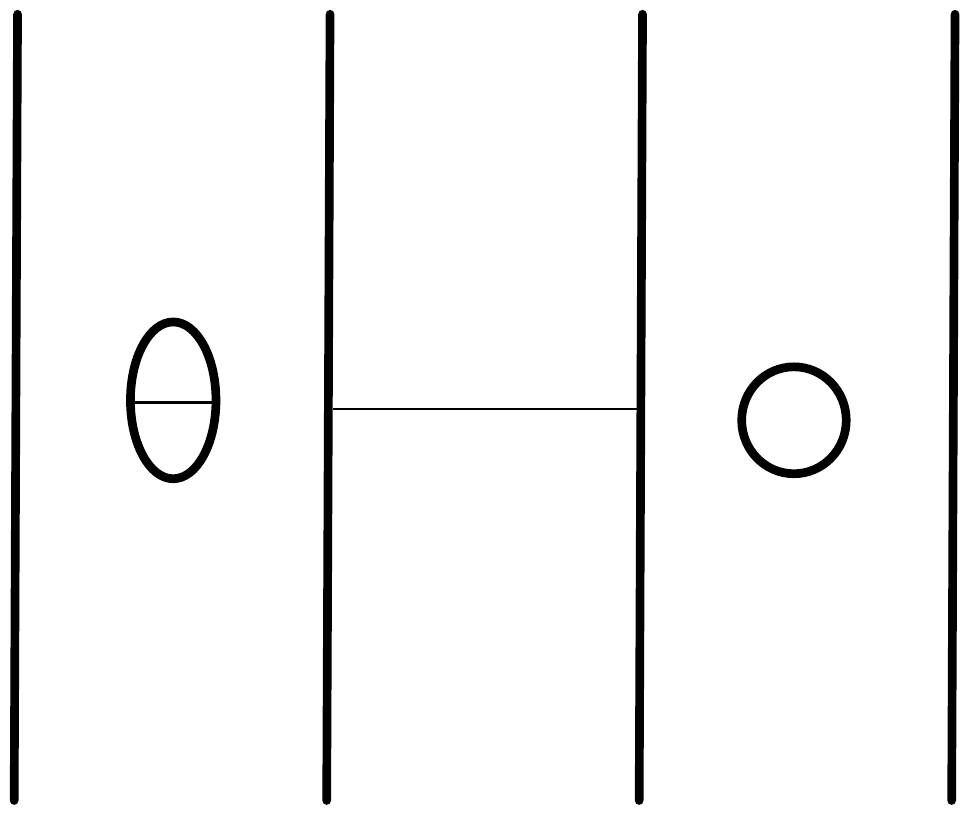}}}

\newcommand{\cobIIOOIIOIOOIO}
{\raisebox{-2pt}{\includegraphics[scale=0.1]{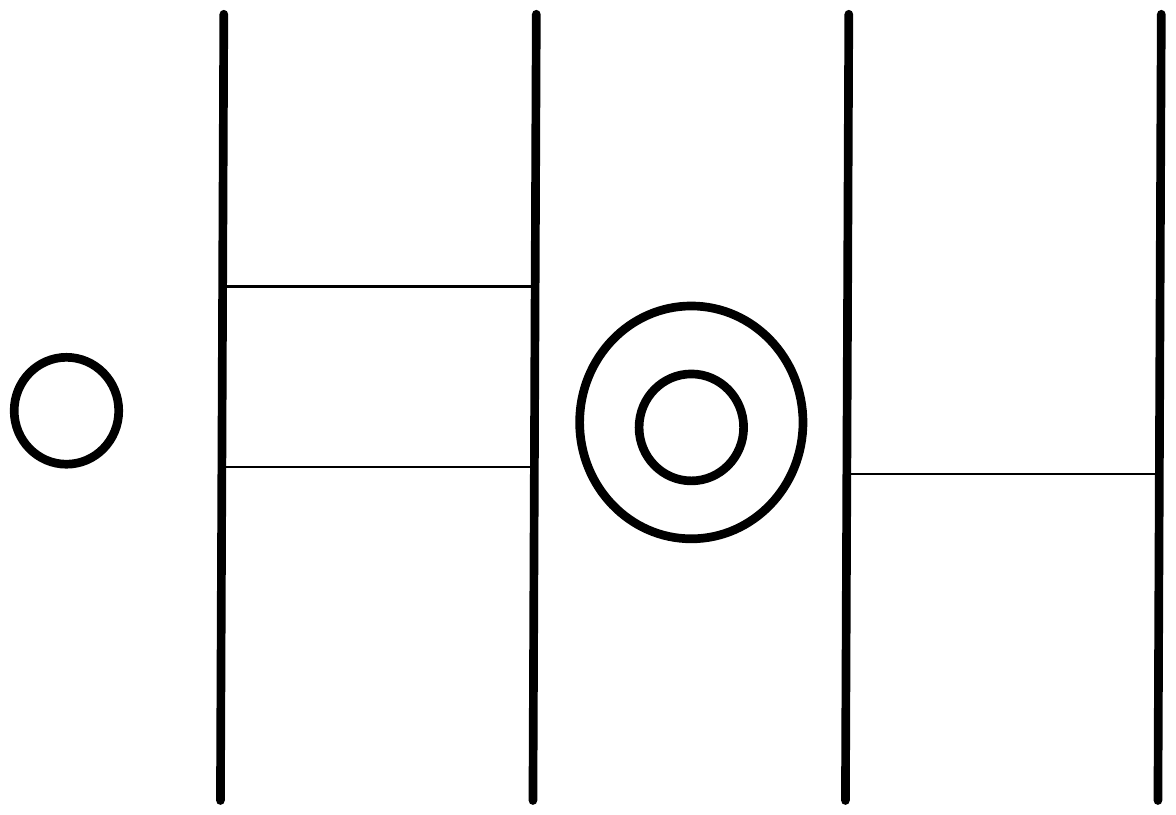}}}

\newcommand{\cobIIOOIIOIOIOO}
{\raisebox{-2pt}{\includegraphics[scale=0.1]{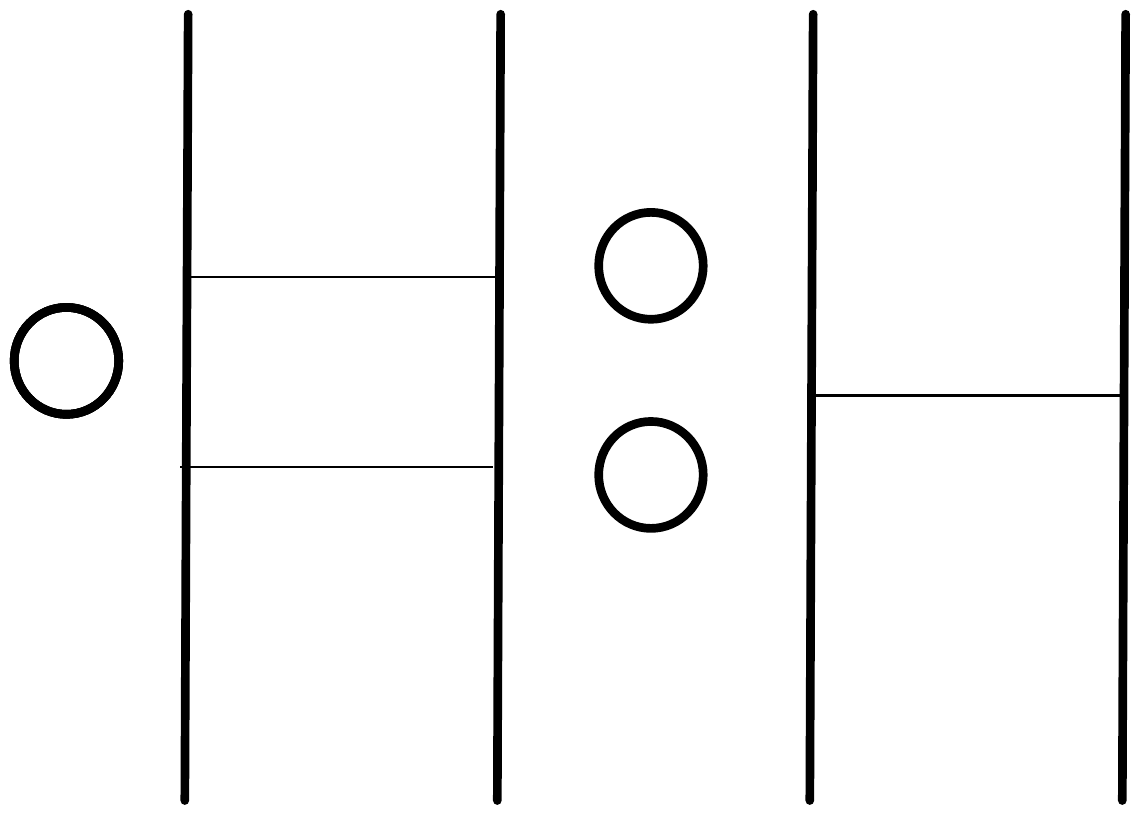}}}

\newcommand{\cobIOIIOIOOIIOO}
{\raisebox{-2pt}{\includegraphics[scale=0.1]{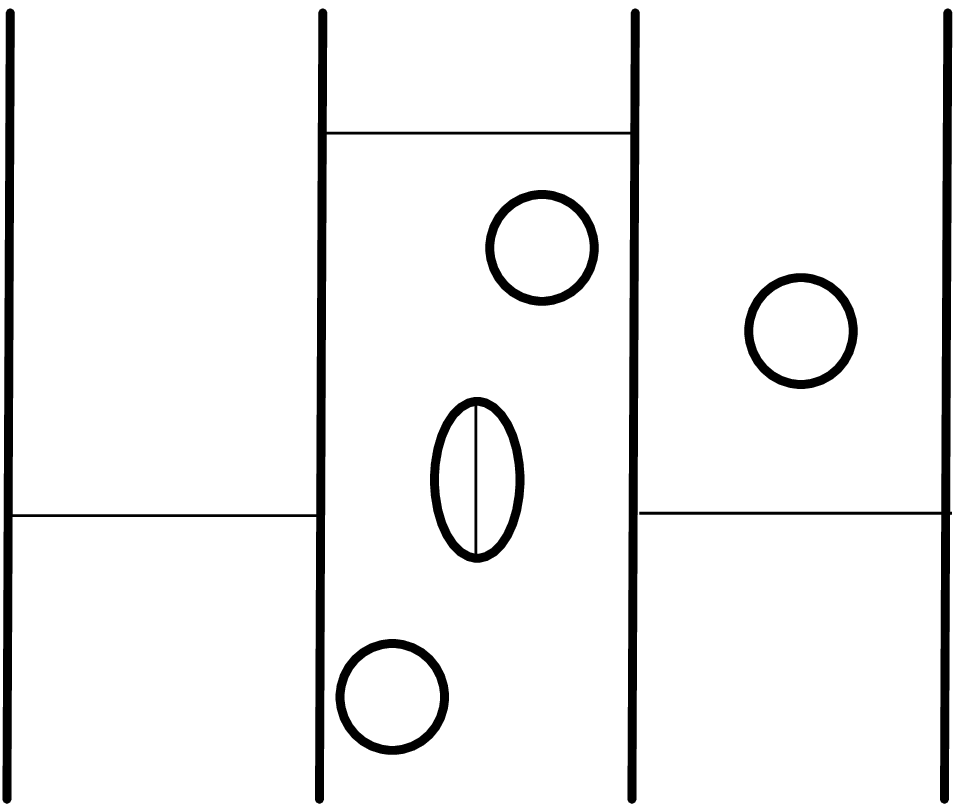}}}

\newcommand{\cobIOOIIIOOIIOO}
{\raisebox{-2pt}{\includegraphics[scale=0.1]{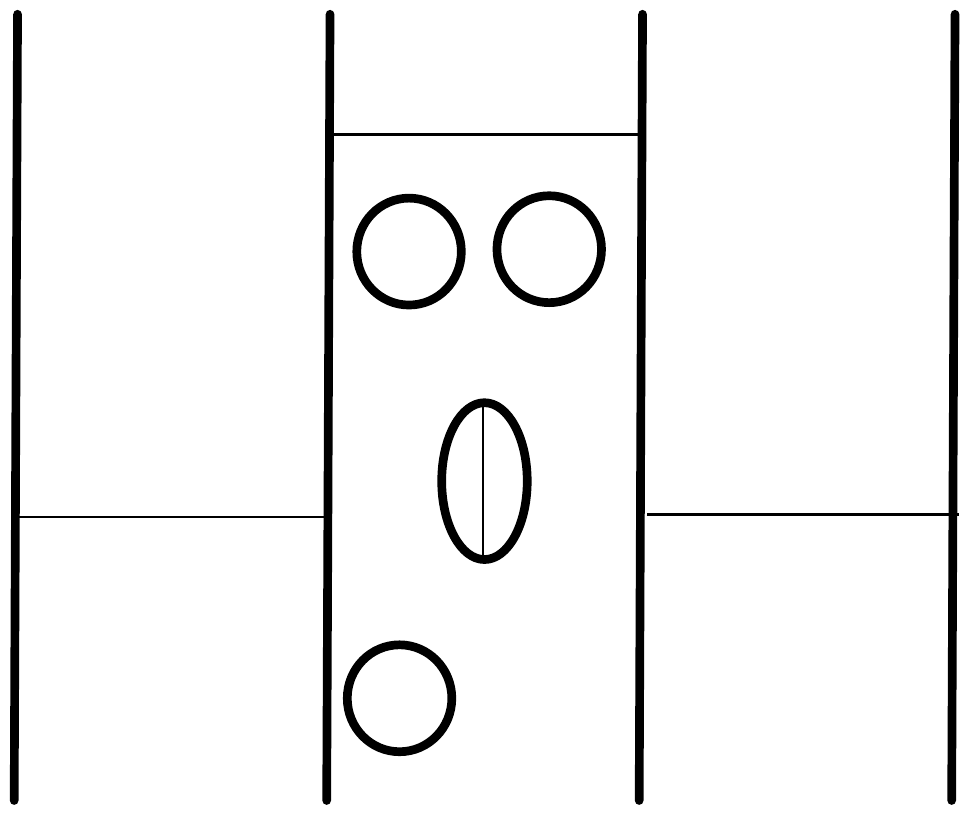}}}

\newcommand{\cobIIOOIIOOIOIO}
{\raisebox{-2pt}{\includegraphics[scale=0.1]{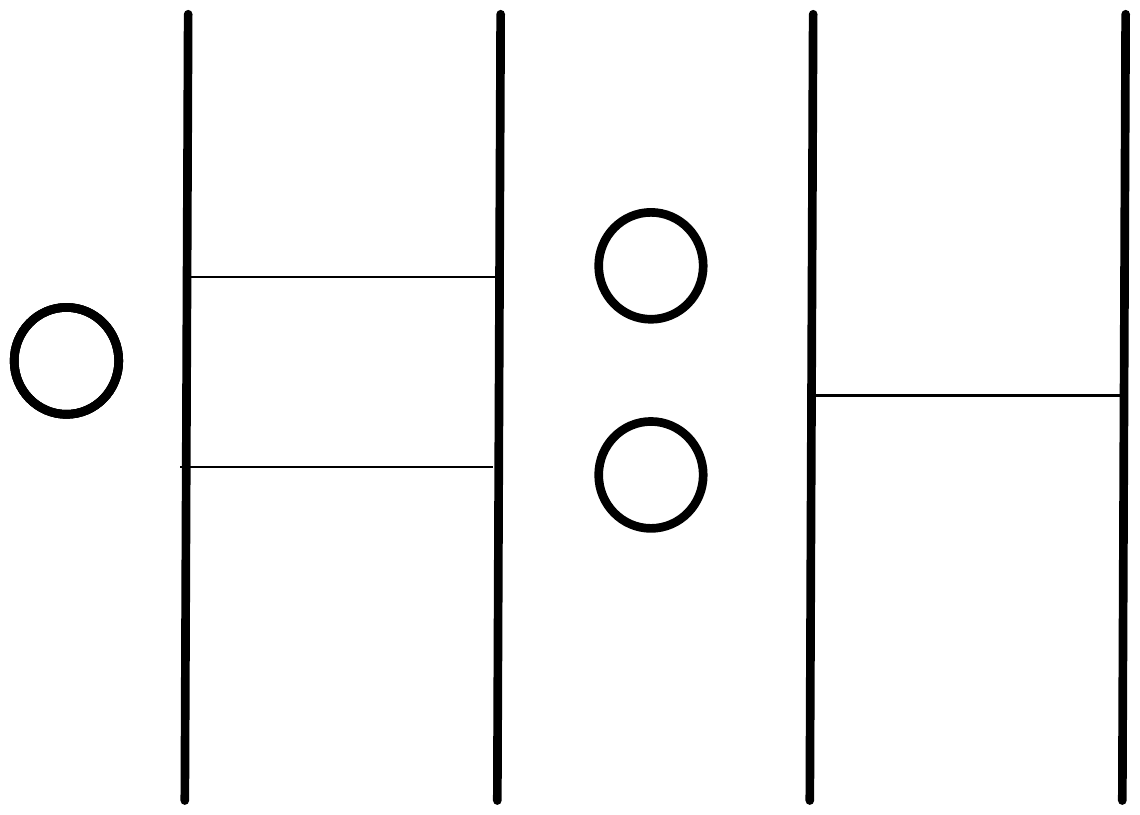}}}

\newcommand{\cobIIIOOIOOIIOO}
{\raisebox{-2pt}{\includegraphics[scale=0.1]{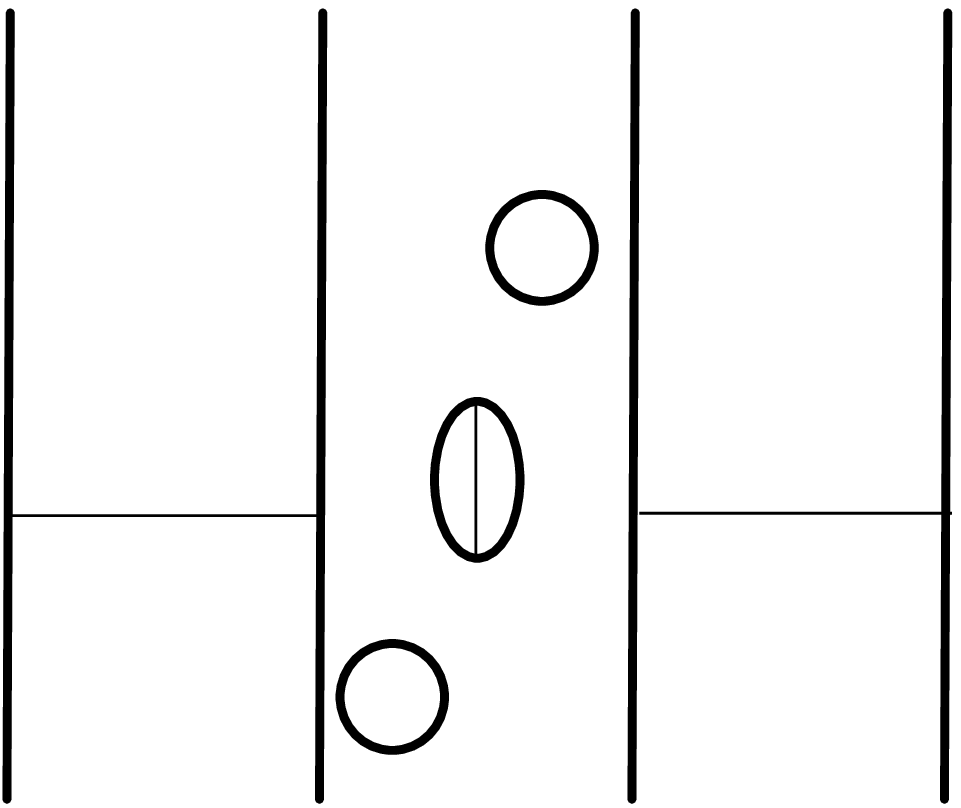}}}

\newcommand{\cobIIOOIIOOIIOO}
{\raisebox{-2pt}{\includegraphics[scale=0.1]{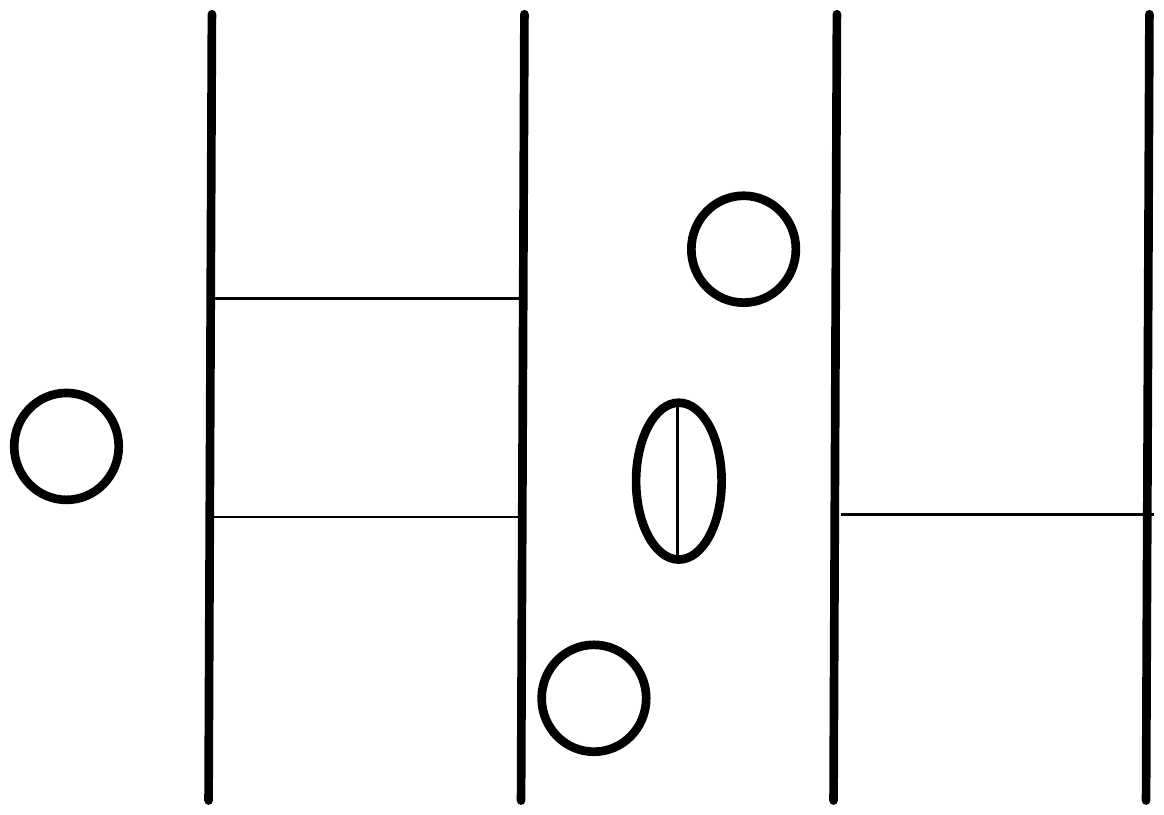}}}

\newcommand{\cobIIOOIIIOOIOO}
{\raisebox{-2pt}{\includegraphics[scale=0.1]{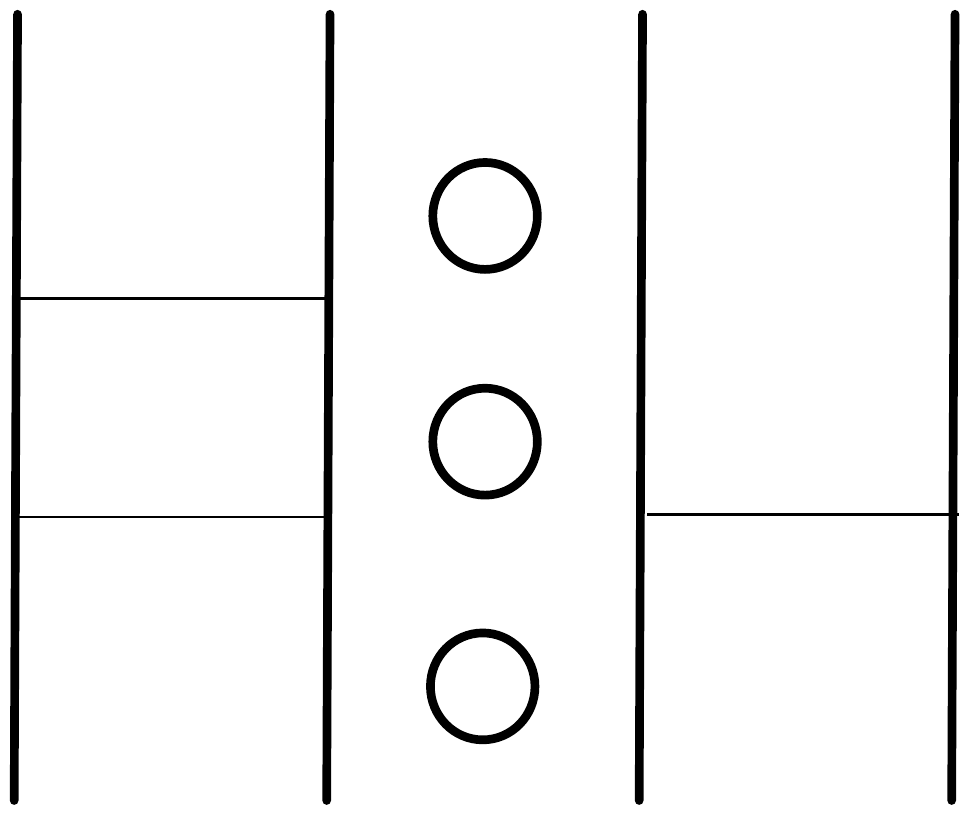}}}

\newcommand{\cobIIOOIIIOIOOO}
{\raisebox{-2pt}{\includegraphics[scale=0.1]{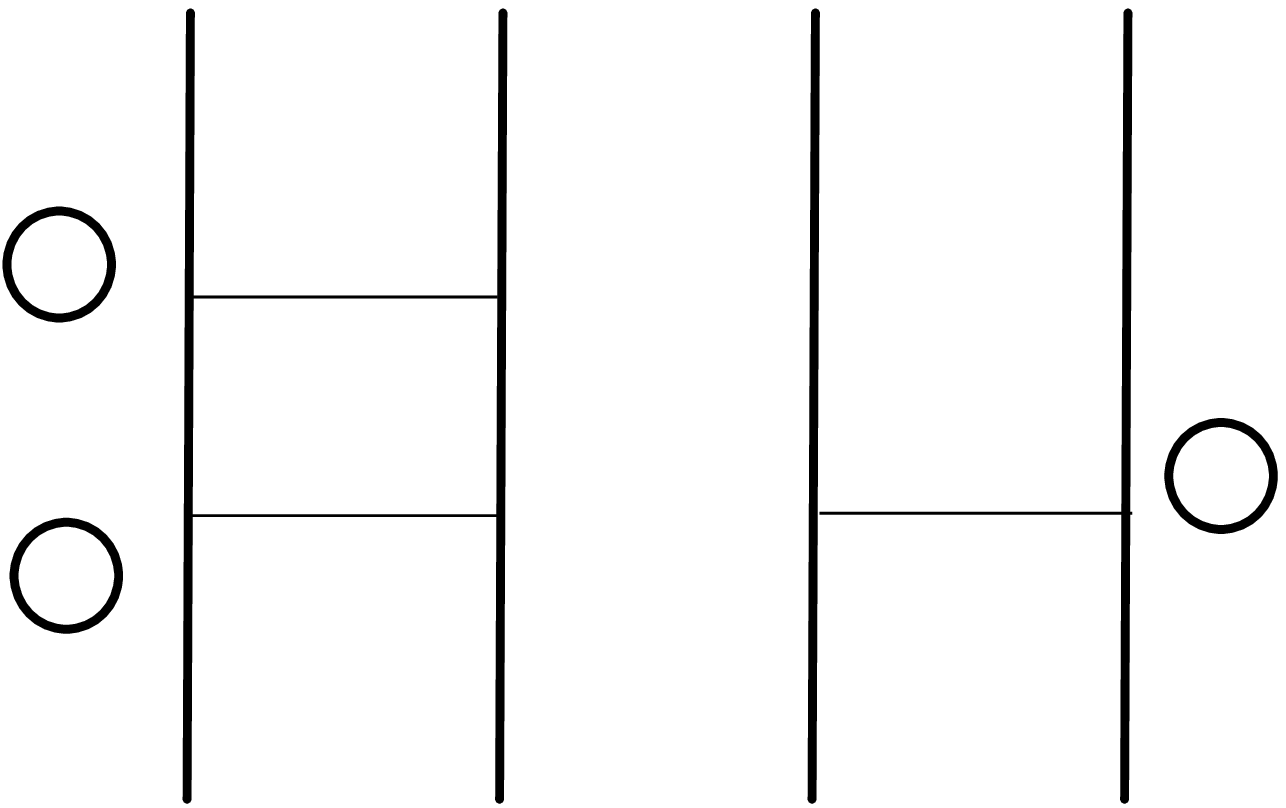}}}

\newcommand{\cobIOOOIIOOIIOI}
{\raisebox{-2pt}{\includegraphics[scale=0.1]{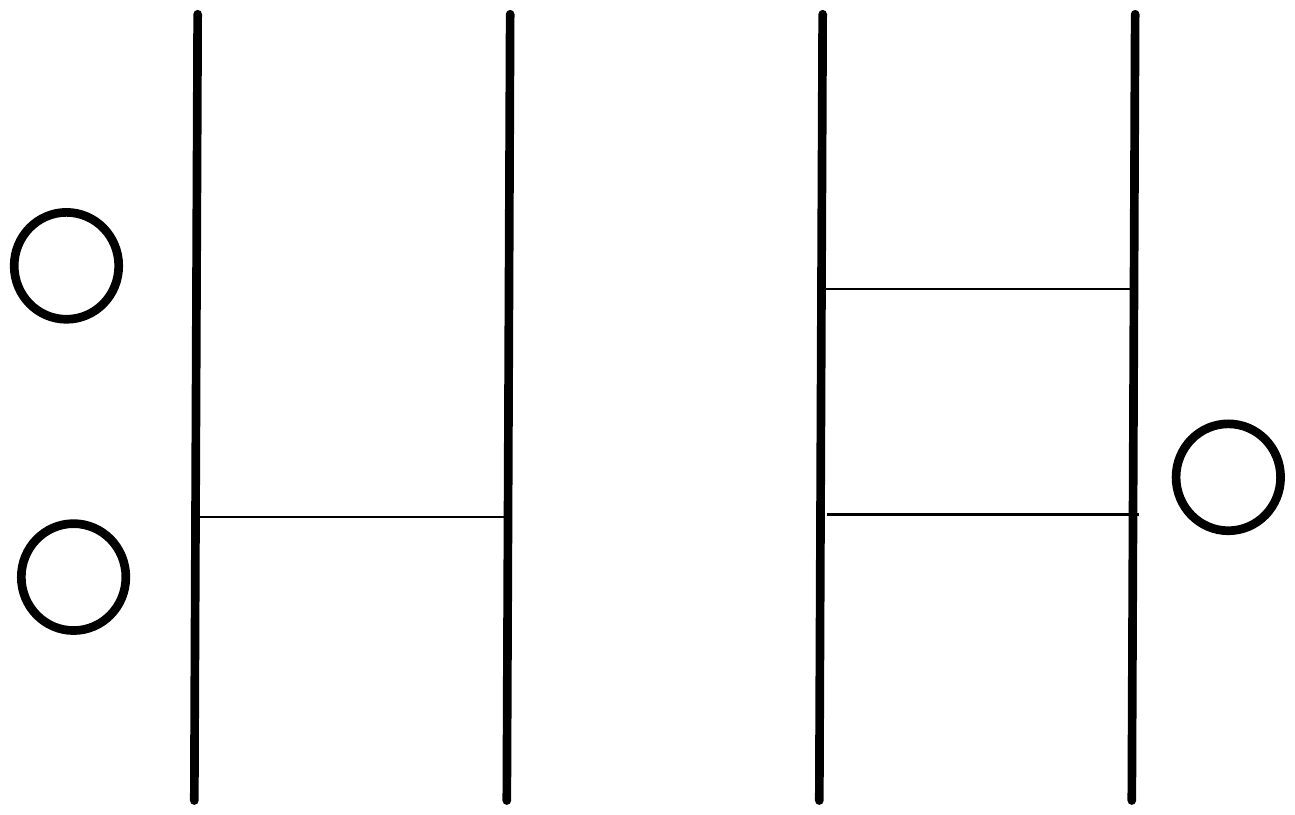}}}

\newcommand{\cobIIOOOIOOIIOI}
{\raisebox{-2pt}{\includegraphics[scale=0.1]{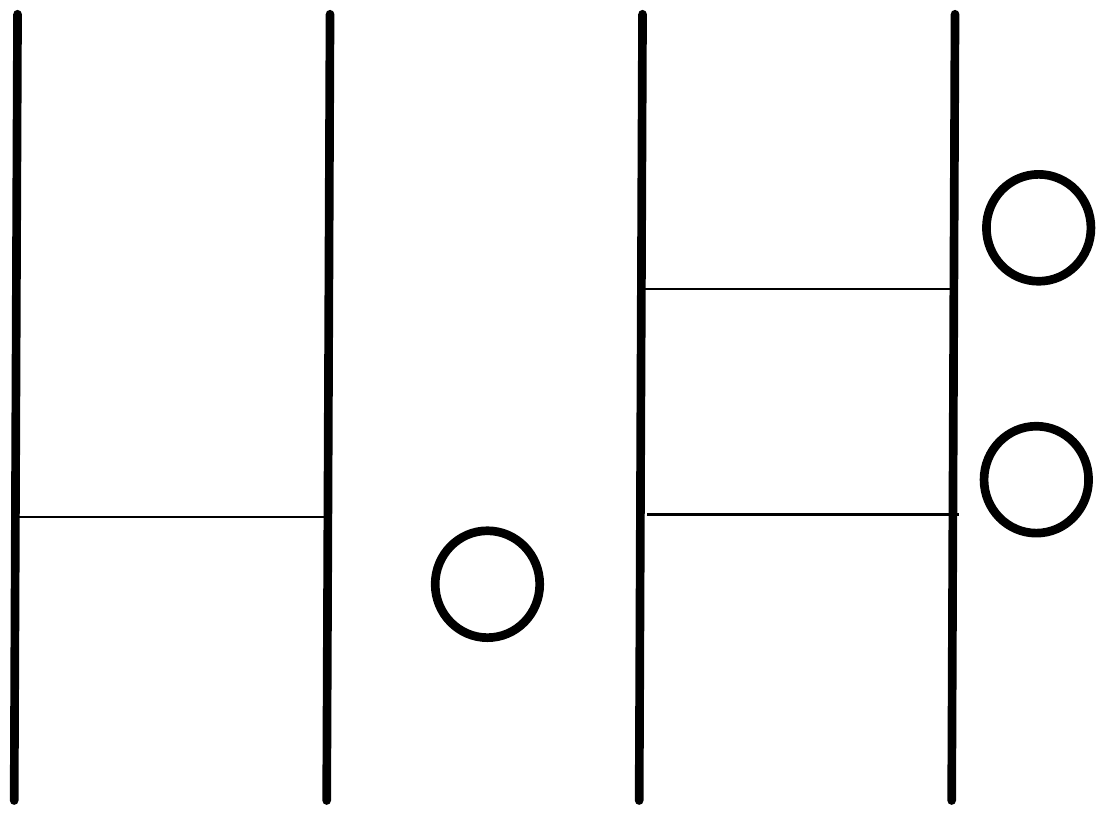}}}

\newcommand{\cobIIOOOIOOIIOO}
{\raisebox{-2pt}{\includegraphics[scale=0.1]{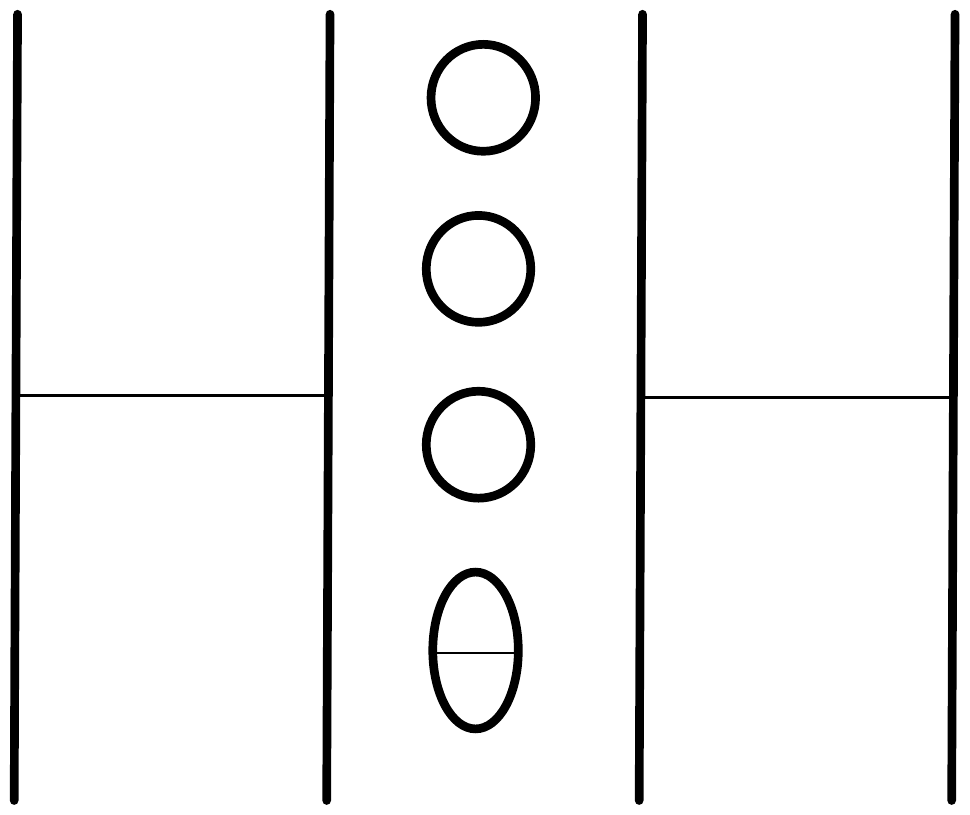}}}

\newcommand{\cobA}
{\raisebox{-2pt}{\includegraphics[scale=0.1]{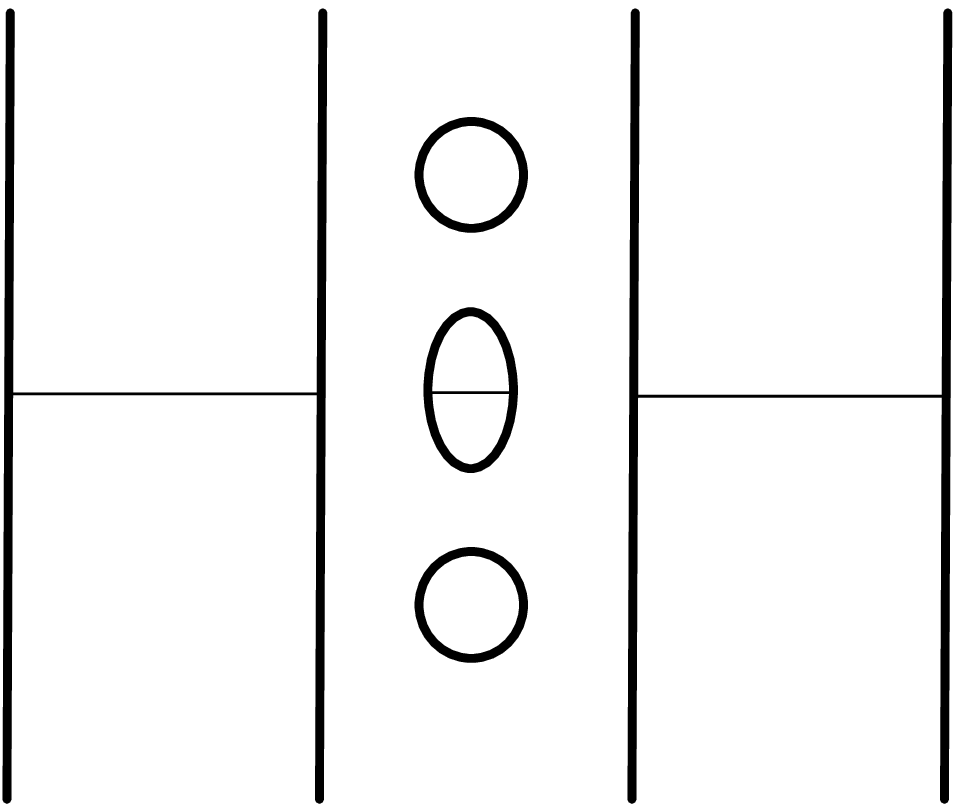}}}

\newcommand{\cobB}
{\raisebox{-2pt}{\includegraphics[scale=0.1]{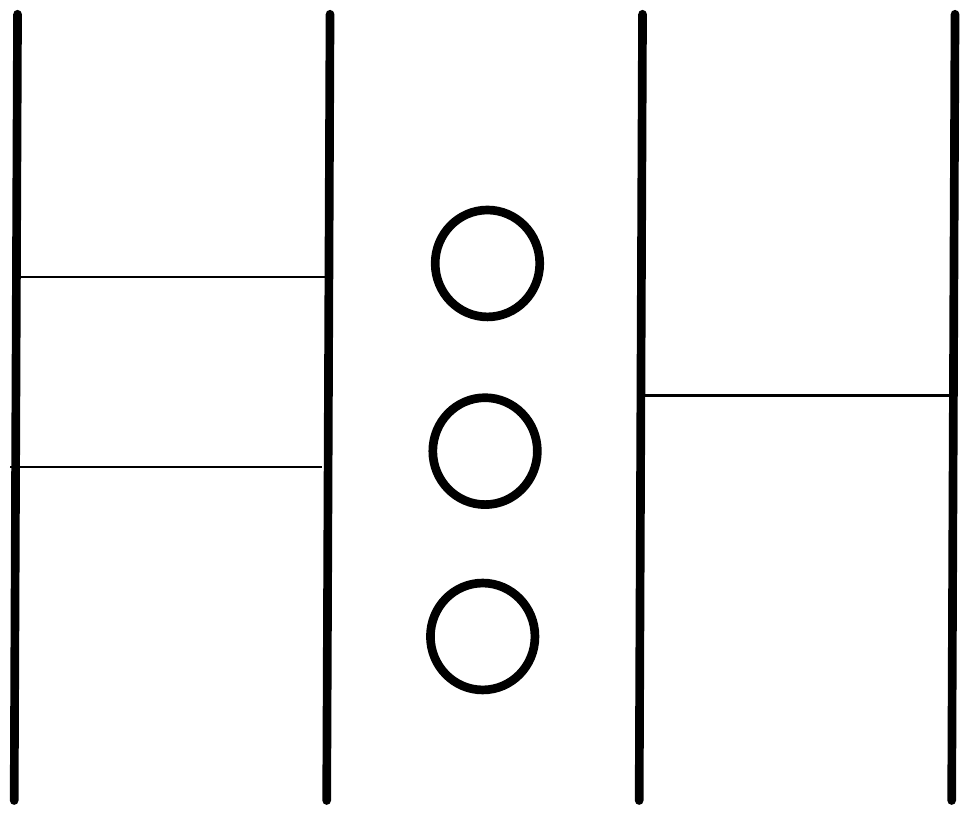}}}

\newcommand{\cobC}
{\raisebox{-2pt}{\includegraphics[scale=0.1]{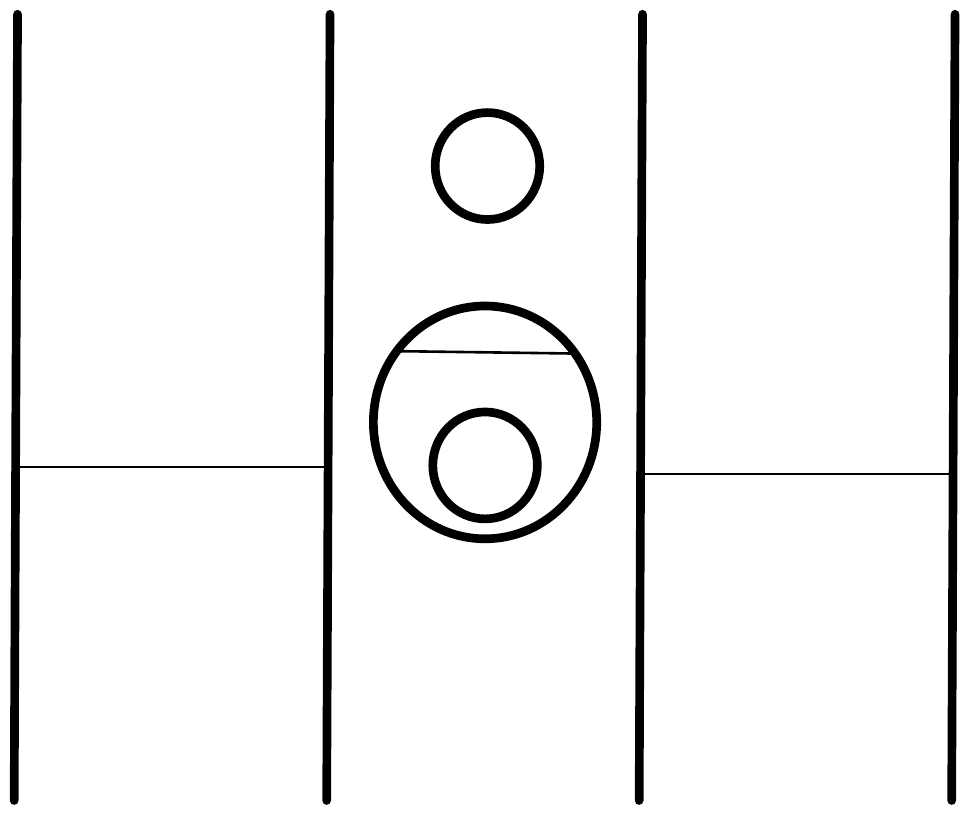}}}

\newcommand{\cobD}
{\raisebox{-2pt}{\includegraphics[scale=0.1]{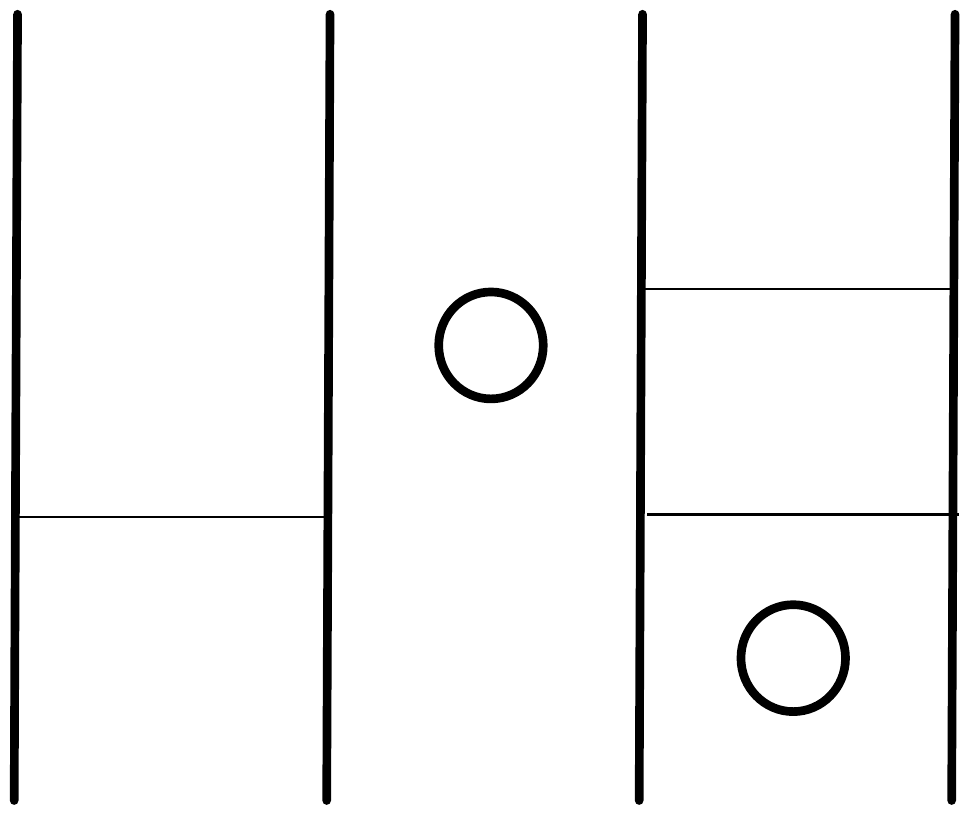}}}

\def\co{\colon\thinspace}

\begin{document}
\thispagestyle{empty}
\title[A refinement of Khovanov homology]{A refinement of Khovanov homology}
\author{Andrew Lobb} 
\address{Mathematical Sciences,
Durham University,
UK}
\email{andrew.lobb@durham.ac.uk}
\urladdr{http://www.maths.dur.ac.uk/users/andrew.lobb/}
\author{Liam Watson} 
\thanks{AL was partially supported by the EPSRC grant
	EP/K00591X/1 and by the Leverhulme Research Programme Grant SPOCK.
	LW was partially supported by a Marie Curie career integration grant, by a Canada Research Chair, and by an NSERC discovery/accelerator grant; AL and LW were organizers and participants of a program at the Isaac Newton Institute while part of this work was completed and acknowledge partial support from EPSRC grant EP/K032208/1; additionally, LW was partially supported by a grant from the Simons Foundation while at the Isaac Newton Institute}
\address{Mathematics Department, University of British Columbia, Canada}
\email{liam@math.ubc.ca}
\urladdr{http://www.math.ubc.ca/~liam}

\begin{abstract}
We refine Khovanov homology in the presence of an involution on the link.  This refinement takes the form of a triply-graded theory, arising from a pair of filtrations. We focus primarily on strongly invertible knots and show, for instance, that this refinement is able to detect mutation.
\end{abstract}

\maketitle

\noindent Khovanov's categorification \cite{kh1} of the Jones polynomial has been well-studied and developed in various %fruitful
directions. Despite this, the original cohomology theory introduced by Khovanov is still not well-understood geometrically. On the other hand, improving on the Jones polynomial, this invariant is known to contain important geometric information. This is perhaps best illustrated by Rasmussen's $s$-invariant \cite{ras3}, which gave new lower bounds on the smooth slice genus. The $s$-invariant has recently been put to use in Piccirillo's elegant proof that the Conway knot is not slice \cite{Piccirillo}, for instance. In order to obtain this invariant, one perturbs the differential and so trades the quantum grading for a quantum filtration.  Indeed, a key strength of Khovanov's categorification lies in its amenability to deformation by perturbation of the differential. Our goal is to apply this apparent strength by perturbing Khovanov's differential in the presence of an involution on the link.  Although the perturbation itself will be unsurprising to those who have thought about equivariant cohomology, the construction gives rise to a surprising filtration that yields a triply graded invariant.
We shall first give an overview of the invariant, together with some properties and applications, postponing a rigorous definition and the proof of invariance.   %Although the deformation itself of the differential will be unsurprising to those who have thought about equivariant cohomology, we shall find that we obtain extra information which takes the form of a non-trivial extra grading.

%We shall postpone rigorous discussion until after the next section.

%\tableofcontents

%\section{Introduction}
%\label{sec:introduction}
\section{Overview and summary}
\label{sec:intro}
%\section{An executive summary for the busy mathematician about town.}

 \parpic[r]{
 \begin{minipage}{57mm}
 \centering
 \includegraphics[scale=0.85]{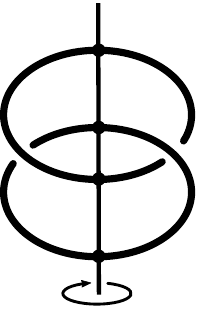}
 \captionof{figure}{The Hopf link is involutive.}
 \label{fig:hopf}
  \end{minipage}%
  }
 \subsection*{Involutions on links} 
A link in $S^3$ is \emph{involutive} if it is fixed, as a set, by some orientation-preserving involution of $S^3$.  Equivalently (the equivalence is discussed later), we can fix some standard involution $\tau\co S^3\to S^3$ and consider links that are equivariant relative to it---this will usually be our point of view in this paper.   The resolution of the Smith Conjecture showed that the fixed point set of such an involution $\tau$ is necessarily an unknotted closed curve \cite{SmithConjecture,Waldhausen1969}.  We take involutive links to be equivalent if they are isotopic through involutive links; see Definition \ref{def:involutive}. In this way, a given link type may admit different involutions (see for example Figure \ref{fig:2-fig-8s}).

%We will adhere to the convention that an involutive link is a link together with a fixed choice of involution on it, that is, a given link admitting two non-isotopic involutions gives rise to two distinct  involutive links. This will be our preferred point of view, since our invariant of choice is defined diagrammatically. 

We work with the Khovanov cohomology $\Kh(L)$ associated with an oriented link $L$ with coefficient ring being the two-element field $\bF$. Recall that this is a finite dimensional  bigraded vector space $\Kh(L)\cong\oplus_{i,j\in\Z}\Kh^{i,j}(L)$ from which the Jones polynomial arises as 
\[{V}_L(q) = \textstyle\sum_{i,j\in\Z}(-1)^iq^j \dim\big(\Kh^{i,j}(L)\big)\]
normalized so that ${V}_U(q)=q^{-1}+q$ where $U$ is the unknot. Khovanov cohomology is the cohomology of a cochain complex defined from a diagram of the link.  If $D$ is a link diagram, let $\CKh(D)=\oplus_{i,j \in \Z}\CKh^{i,j}(D)$ denote the associated cochain complex with differential $d \co \CKh^{i,j}(D)\to \CKh^{i+1,j}(D)$.

When $D$ is an involutive link diagram (that is, a link diagram with a vertical axis of symmetry as in Figure \ref{fig:hopf}, for example) we equip $\CKh(D)$ with the perturbed differential $\partial = d + d_\tau$.  By deliberate abuse of notation, we write $\tau$ for an involution on $\CKh(D)$ induced by the symmetry of the diagram, and $d_\tau=\id+\tau$ for a new differential $d_\tau\co \CKh(D)\to \CKh(D)$ commuting with the Khovanov differential $d$.

The deformed complex with differential $\partial$, which we will henceforth denote by $\CKht(D)$, decomposes according to the quantum grading $\CKht(D) = \oplus_{j\in\Z}\CKht^j(D)$ and so its cohomology $\Ht(D)$ inherits the quantum grading $\Ht(D) = \oplus_{j \in \Z} \Ht^j(D)$ too.  In this paper we shall show that $\CKht(D)$ may be further equipped with two filtrations, which we will now briefly describe.

\labellist
  \pinlabel {$\F^0$} at 12 63  \pinlabel {$\F^1$} at 36 40  \pinlabel {$\F^2$} at 63 15
   \pinlabel {$\G^0$} at 393 15  \pinlabel {$\G{^\frac{1}{2}}$} at 415 40  \pinlabel {$\G^1$} at 440 63
       \pinlabel {$\tau$} at 103 110  \pinlabel {$\tau$} at 347 110
       \pinlabel {$\tau$} at 206 180  \pinlabel {$\tau$} at 246 180	
         \endlabellist
\begin{figure}[t]
\includegraphics[scale=0.75]{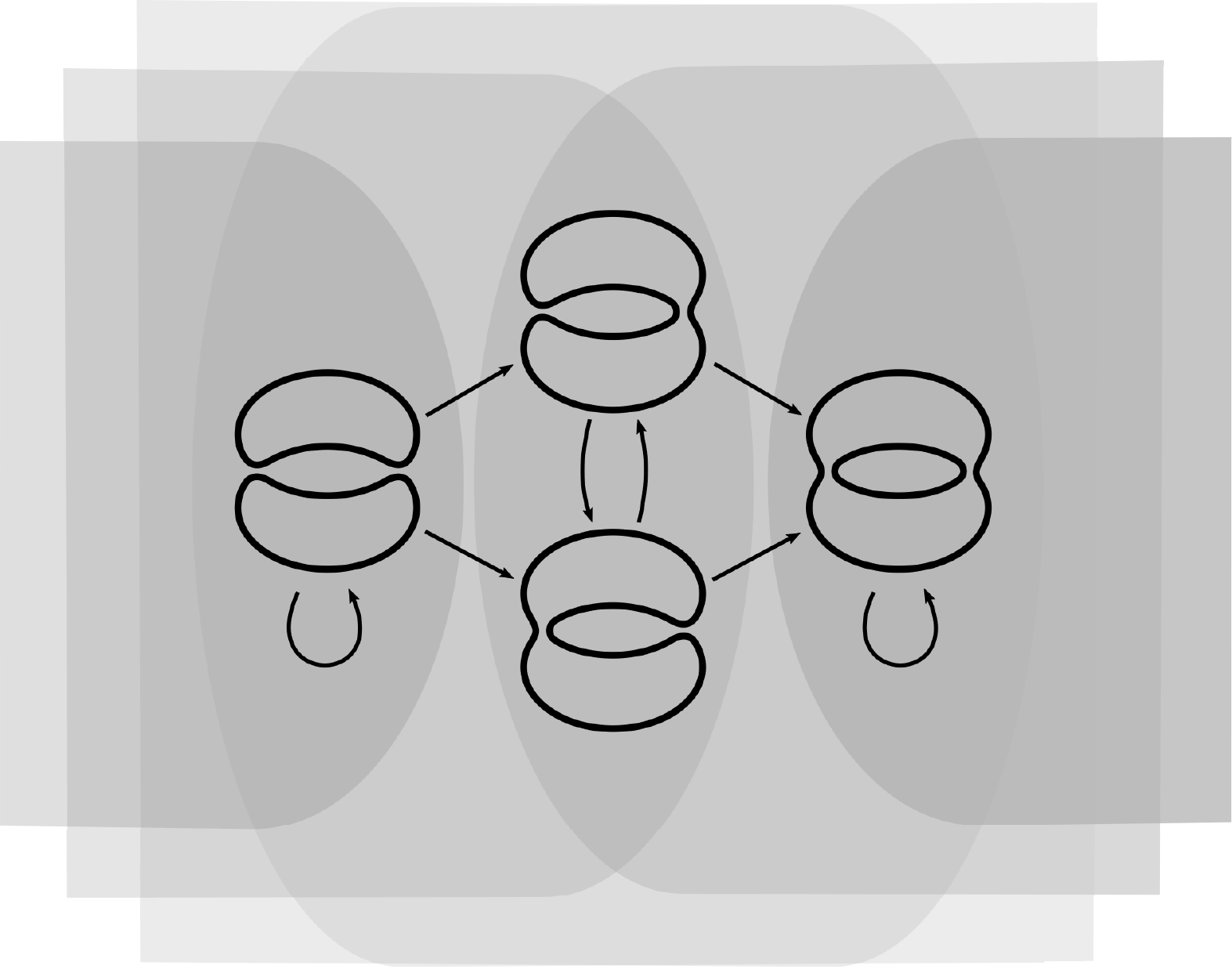}
\caption{Two filtrations on the cube of resolutions for an involutive diagram of the Hopf link (shown in Figure \ref{fig:hopf}) oriented so that the linking number between the components is 1. New edges added to the cube by the involution are labelled by $\tau$, while unlabelled edges correspond to the usual cobordisms appearing in Khovanov's original construction.}\label{fig:hopf-cube-of-resolutions}
\end{figure}

The first of these is inherited from the cohomological grading on $\CKh(D)$, which must necessarily now be relaxed to a filtration.  We write this as $\{\F^i\CKht(D)\}$ so that $\F^{i-1}\CKht(D)\subseteq\F^i\CKht(D)$ and $\partial\F^{i-1}\CKht(D)\subset \F^i\CKht(D)$.  This filtration gives rise to a spectral sequence with the $E^2$ page being Khovanov cohomology $E^2_\F(D) = \Kh(D)$.  We shall show that every page after the first page is an invariant of the involutive link.

%Note that taking cohomology with respect to the Khovanov differential $d$ gives rise to an induced differential $d_\tau^*$ on Khovanov cohomology. We establish that $\{\F^i(L)\}$ is an invariant of the involutive link $L$, and consequently the map $d_\tau^*$ is an invariant of the symmetry. 

The second filtration is less obvious: we consider a half-integer valued filtration by subcomplexes denoted by $\{\G^k\CKht(D)\}$.  With a suitable choice of page indexing convention, we establish that each page after the second page of the induced spectral sequence is an involutive link invariant.  For the purposes of this introduction, this filtration is best illustrated in an example;  see Figure \ref{fig:hopf-cube-of-resolutions} for the pair of filtrations illustrated on the cube of resolutions for the  involutive Hopf link diagram of Figure \ref{fig:hopf}.

%We show that each of these filtrations is an invariant of the link-with-involution, and establish a third grading on the $E_\infty$-page of the associated spectral sequence. Namely, we have:

Together, these two filtrations induce a bifiltration on the cohomology $\Ht^j(D)$.  Taking the associated graded pieces of the bifiltration we obtain a triply graded link invariant. To summarize, our main result is:

\begin{theorem}
	\label{thm:triply-graded-invariant}
Following the construction as outlined above, there is a triply graded invariant $$\Kht(L)\cong\oplus_{i,j\in\Z, \, k \in \frac{1}{2}\Z}\Kht^{i,j,k}(L)$$ of an involutive link $L$  obtained as the associated graded space to the singly graded and bifiltered cohomology $\Ht(L)$. Furthermore, the construction outlined above gives two spectral sequences, each with bigraded pages.  Every page including and after $E^2_\F(L)=\Kh(L)$ and $E^3_\G(L)$ is an involutive link invariant.
\end{theorem}

The construction will be set up more precisely in Section \ref{sec:hom_alg}.

In order to describe the pair of filtrations giving rise to this new grading in Khovanov cohomology, we continue with our explicit example: A cochain complex for the involutive Hopf link is shown in Figure \ref{fig:hopf-complex} with a detailed explanation in the caption.

\labellist
\pinlabel {$\partial$} at 158 152  \pinlabel {$d_\tau$} at 142 138 \pinlabel {$d^\ast$} at 172 138
\small
  \pinlabel {$2$} at -7 17 \pinlabel {$4$} at -7 47  \pinlabel {$6$} at -7 78 \pinlabel {$8$} at -7 108
  \pinlabel {$0$} at 17 -7 \pinlabel {$1$} at 46 -7 \pinlabel {$2$} at  78 -7
   \pinlabel {$0$} at 125 -7 \pinlabel {$1$} at 156 -7 \pinlabel {$2$} at  187 -7
    \pinlabel {$0$} at 234 -7 \pinlabel {$1$} at 265 -7 \pinlabel {$2$} at  296 -7
    % \pinlabel {$0$} at 432 -7 \pinlabel {$1$} at 470 -7 \pinlabel {$2$} at  508 -7
    \pinlabel {$0$} at 356 52 \pinlabel {$1$} at 387 52 \pinlabel {$2$} at  418 52
     \pinlabel {$0$} at 331 78  \pinlabel {$1$} at 331 108
     \pinlabel {$i$} at 437 60  \pinlabel {$k$} at 337.5 130
     \tiny
 %     \pinlabel {$2$} at 437 7  \pinlabel {$4$} at 437 24
   %    \pinlabel {$6$} at 514 7  \pinlabel {$8$} at 514 24
    %    \pinlabel {$2$} at 437 83  \pinlabel {$4$} at 437 100
      % \pinlabel {$6$} at 513 120  \pinlabel {$8$} at 513 137
                \endlabellist
\begin{figure}[t]
\vspace*{5pt}\includegraphics[scale=0.9]{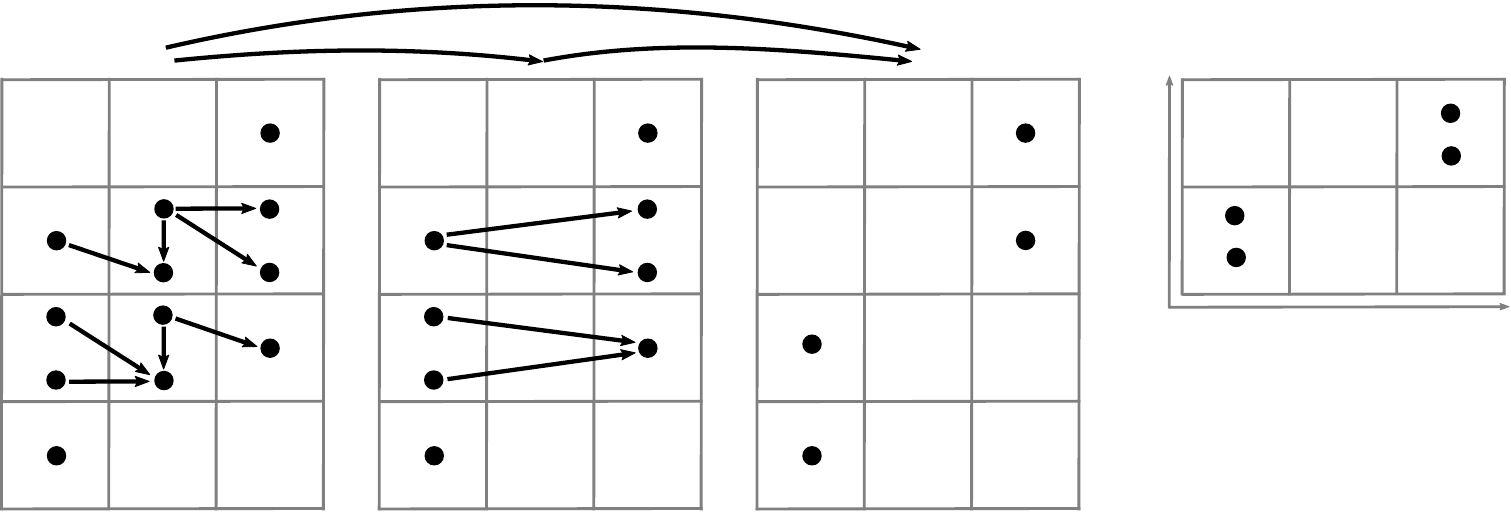}\vspace*{5pt}%room for labels
\caption{Computing the invariant for the involutive Hopf link $L$ from Figure \ref{fig:hopf}, where each $\bullet$ denotes a copy of the field $\bF$. The chain complex $\CKht(L)$ (far left) is written with respect to a non-standard choice of basis for clarity; the $i$-grading runs horizontally and the $j$-grading runs vertically.
	% is a mild perturbation of the usual Khovanov complex $\CKh(L)$ having imposed the change of basis $\langle v_\pm^{01},v_\pm^{10}\rangle\cong\langle v_\pm^{01},v_\pm^{01}+v_\pm^{10}\rangle$ in $\CKh^{i=1}(L)$
	%We write $i$ horizontally and $j$ vertically.
 Note that the differential $d_\tau^\ast$ on $E^2_\F(L) = \Kh(L)$ vanishes.  In fact, ignoring the new gradings, the invariant $\Ht(L)$ agrees with $\Kh(L)$. On the other hand, taking cohomology with respect to $d_\tau$ first yields $E^2_\G(L)$ with an induced differential $d^\ast\ne 0$ (the induced differential has been indicated). Finally, as will be our usual convention, at the top right of the figure we have given the invariant $\Kht^{i,j,k}(L)$ in the plane with the $j$-grading collapsed.}\label{fig:hopf-complex}
\end{figure}

The proof of invariance, which occupies Section \ref{sec:inv}, deserves a few words.  The reader can consult Theorem \ref{thm:R} and Figure \ref{fig:moves} for a list of local moves giving rise to the involutive Reidemeister theorem.  With this in hand, the initial strategy is straightforward: one wishes to construct bifiltered chain homotopy equivalences associated with each of these moves.  Surprisingly, it turns out that this is not possible; instead we find $\F$-filtered cochain maps and modify these by a suitable homotopy to give the $\G$-filtered cochain maps.  Moreover, we are forced to contend with a local move consisting of six crossings (see M3 in Figure \ref{fig:moves}); the desired $\F$-filtered map in this case is a perturbation of a composite of 10 maps induced by standard (non-involutive) Reidemeister moves.

\subsection*{Strongly invertible knots} An oriented knot $K$ in the three-sphere $S^3$ is called \emph{invertible} if it is equivalent, through ambient isotopies, to its orientation reverse. Examples are provided by low-crossing number knots. However, as first established by Trotter, non-invertible knots exist \cite{Trotter1963}; the smallest (in terms of crossing number) non-invertible knot is $8_{17}$ in the Rolfsen knot table \cite{Rolfsen1976}. In fact, there are only 3 non-invertible examples among knots with 9 or fewer crossings ($8_{17}$, $9_{32}$, $9_{33}$), and each of the inversions present in the remaining knots is realized by an involution of $S^3$. The existence of this apparently stronger form of inversion promotes $K$ to a \emph{strongly invertible} knot.  Strongly invertible knots are examples of involutive links.  More generally, an involutive link for which the involution reverses any choice of orientation on each link component is called a \emph{strongly invertible link}.  This condition is equivalent to requiring that the fixed point set of the involution meets every component in exactly two points.

While strong invertibility is known to be a strictly stronger condition than invertibility by work of Hartley \cite{Hartley1980}, it turns out that $K$ admitting an inversion and $K$
admitting a strong inversion are equivalent whenever $K$ is a non-satellite knot. In the hyperbolic case this follows from work of Riley \cite{Riley1979} and Thurston \cite{Thurston1982} (see \cite[Lemma 3.3]{Sakuma1986}), while the case of torus knots is due to Schreier \cite{Schreier1924}. In particular, these involutions may all be viewed as order 2 elements in a subgroup of some dihedral group, which coincides with the symmetry group of $K$ (or the mapping class group of the knot complement). For a more comprehensive study of strongly invertible knots, including a tabulation recording their symmetries, see Sakuma \cite{Sakuma1986}. 

From this point of view, invariants of strongly invertible knots can be viewed as invariants of conjugacy classes of elements in the symmetry group of a knot. Sakuma constructs such an invariant using a construction due to Kojima and Yamasaki \cite{KY1979} applied to a 2-component link obtained through a quotient of the symmetry in question. Further invariants of strong inversions have been obtained from Khovanov cohomology by Couture \cite{Couture2009}, by Snape \cite{Snape}, and by the second author \cite{Watson2017}. These tend to be stronger than Sakuma's invariant, if harder to compute.   But they are similar to Sakuma's work in that, in each instance, these Khovanov-type invariants also arise by first passing to some auxiliary object via a quotient. By contrast, our construction embeds the involution, when present, directly in the Khovanov complex of the knot in question. In so doing, when restricted to strongly invertible knots, Couture's invariant is recovered as a specific page in the spectral sequence corresponding to the $\G$-filtration; see Section \ref{sec:E2G}, particularly Remark \ref{rem:couture}.

\subsection*{Separating mutant pairs} By work of Bloom \cite{Bloom2010} and Wehrli \cite{Wehrli2010},  Khovanov cohomology with coefficients in $\bF$ does not detect mutation. In certain cases, the presence of a strong inversion allows us to separate mutant pairs using the triply graded refinement. 

\parpic[r]{
 \begin{minipage}{60mm}
 \centering
 \includegraphics[scale=0.6]{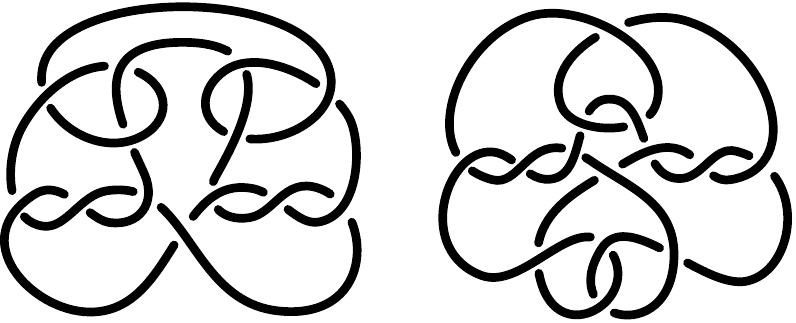}
 \captionof{figure}{A mutant pair that is separated by the refined invariant.}
 \label{fig:mutants-small}
  \end{minipage}%
  }
As a particular example, consider the strongly invertible knots shown in Figure \ref{fig:mutants-small}. The left-hand diagram is alternating, so it follows that this pair is indistinguishable by knot Floer homology \cite{OSz2004-knot,ras}---for alternating knots the Alexander polynomial and the signature, both preserved under mutation, determine the knot Floer homology. This example can be leveraged to produce infinite families, and seems worth recording:  

\begin{theorem}\label{thm:mutation}
There exist pairs of knots related by mutation, with identical Khovanov cohomology and identical knot Floer homology, which can be distinguished from one another by the triply graded refinement of Khovanov cohomology by appealing to an appropriate symmetry. 
\end{theorem}

A detailed account of the construction is given in Section \ref{sec:examples}, but the key point in the proof is this: The left-hand diagram has non-trivial invariant in degree $k=-3$ while the right-hand diagram has invariant supported only in degrees $k>-3$. As a result, just as knot Floer homology can detect mutation in cases where the genus of the relevant knots changes \cite{OSz2004-genus,OSz2004-mutation}, Theorem \ref{thm:mutation} establishes that a similar fact is true for Khovanov cohomology, suitably refined, in the presence of an appropriate involution. There is another, closely related, viewpoint on this result: since mutant knots have homeomorphic two-fold branched covers, Theorem \ref{thm:mutation} allows us to distinguish involutions on this common covering space using Khovanov cohomology. In general, certifying that involutions on a given manifold are distinct is a difficult task. %\footnote{Boileau points out: The 2-fold branched covers of these mutant knots are homeomorphic, but as a result of this Theorem we have separated two distinct involutions on this 3-manifold. This is generally very hard to do, apparently. On the other hand, It seems like the invariant for the two involutions we check for the pre-mutant have realtively close Khovanov cohomology. Maybe the support is the same? I wonder if we know how to separate these. More generally, I don;t know if we have an example of a knot admitting a pair of involutions that we can't separate. This seems like a possible candidate.}

\parpic[r]{
 \begin{minipage}{65mm}
 \centering
 \includegraphics[scale=0.50]{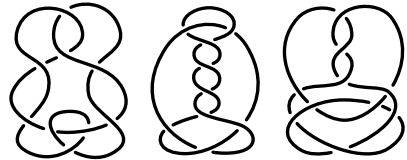}
 \captionof{figure}{The knots $8_8$ (left and centre) and $10_{129}$ (right) share the same Khovanov cohomology.}
 \label{fig:8_8-vs-10_129-small}
  \end{minipage}%
   }
The triply-graded invariant separates other pairs of knots with identical Khovanov cohomology. The knots $8_8$ and $10_{129}$ are shown in Figure \ref{fig:8_8-vs-10_129-small}. These knots have identical Khovanov cohomology but are known not to be related by mutation; see \cite{Watson2007}. Note that, in this example, we appeal to the fact that $8_8$ is a (non-torus) 2-bridge knot, hence it admits a pair of strong inversions and two involutive symmetric diagrams. One of these gives rise to an invariant which is non-trivial in grading $k=4$, while the other gives rise to an invariant supported only in gradings $k\ge-2$. On the other hand, the invariant for $10_{129}$ with the symmetry shown is trivial in gradings $k> 3$ but non-trivial in grading $k=-3$. The interesting point here is that, as with the mutant examples, it is often possible to gain new information from the $k$-grading without having to compute the entire refined invariant. 

\subsection*{Properties and further remarks} Certain general properties are already apparent in simple examples. 

\begin{proposition}\label{prp:vanishing}
When $L$ is an involutive link and  $\dim\Kh^{i,j}(L)\le 1$ then the induced differential $d_\tau^*$ vanishes in grading $(i,j)$.  
\end{proposition} 

This observation follows immediately from the definition of the $\F$-filtration; see Subsection \ref{subsec:filtration_definition}. By contrast, it is relatively straightforward to construct examples for which $d_\tau^*\ne0$. Recall that, for any knot $K$ one obtains an obvious strong inversion---switching the summands---on $K\# K$, and finds that $\dim\Kh(K) =\dim E^3_\G(K\# K)\ge \dim\Kht(K\# K)$ in this setting. But when $K$ is non-trivial $\dim\Kh(K\# K)>\dim\Kh(K)$, and hence $(\Kh(K\# K), d_\tau^*)$ is an interesting chain complex. Perhaps more strikingly, there are examples of 7-crossing knots, each admitting a pair of strong inversions, for which one strong inversion has $d_\tau^*=0$ and the other has $d_\tau^*\ne0$. These are the knots $7_4$ and $7_7$; see Section \ref{sec:examples}, in particular, Figure \ref{fig:table}. It seems surprising that the vanishing of $d_\tau^*$ alone is sufficient to separate strong inversions on a given knot, particularly on small knots where the behaviour of the undeformed Khovanov cohomology is known to be rather tame.

\begin{proposition}\label{prp:mirror}
When $L$ is an involutive link $\Kht^{i,j,k}(L^*) \cong \Kht^{-i,-j,-k}(L)$ where $L^*$ is the involutive mirror of $L$. 
\end{proposition}

\parpic[r]{
 \begin{minipage}{35mm}
 \centering
 \includegraphics[scale=0.7]{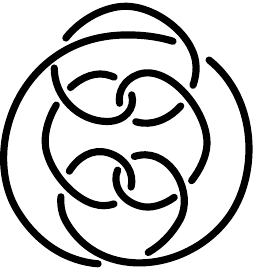}
 \captionof{figure}{The knot $10_{71}$ is not amphicheiral.}
 \label{fig:10_71-small}
  \end{minipage}%
   }
A crucial point here is that this mirror is taken in the involutive setting; the involution $\tau\co S^3\to S^3$ induces an involution $\tau_*$ on reversing the orientation on $S^3$.  Diagramatically, one obtains the involutive mirror $D^*$ of a involutive diagram $D$ by changing all the crossings of $D$ from over- to under-crossings.

If a knot $K$ is amphicheiral and admits a strong inversion, the mirror of that strong inversion may be a \emph{distinct} strong inversion on $K$.
%Consequently, amphichieral knots admitting distinct strong inversions give rise to pairs of distinct strongly invertible knots.
Consider $4_1$ and $6_3$ as particular examples (see Figure \ref{fig:table}), and observe that the triply graded invariant separates the two strong inversions. On the other hand, Sakuma has observed that a property like this can be exploited to certify that a knot is non-amphicheiral; see \cite{Sakuma1986,Watson2017}. We return to this point, and give the proof of Proposition \ref{prp:mirror}, in Subsection \ref{sub:amph}. For example, the knot $10_{71}$ admits a unique strong inversion (shown in Figure \ref{fig:10_71-small}) however $\Kht^{i,j,k}(K)\ncong \Kht^{-i,-j,-k}(K)$ when $K$ is this knot. It follows that $K$ is not isotopic to $K^*$ by the uniqueness of the involution. We remark it is the case, however, that $\Kh^{i,j}(K)\cong\Kh^{-i,-j}(K)$ for all gradings $i,j\in\Z$ so that the undeformed Khovanov cohomology does not obstruct $K$ from being amphicheiral.
%viewed through Khovanov homological-glasses.\footnote{Still don't have a pair of these amazing specs? Get yours today!}

In the case that a component of the involutive link meets the fixed point set, one may define a \emph{reduced} version of the theory, written with a tilde.  (For the \emph{cognoscenti}, the intersection with the fixed point set provides a suitable place for a basepoint.)  Since we work over $\bF$ there is a splitting of the unreduced invariant.  The reduced invariant is set up more precisely and the following proposition is proved in Subsection \ref{subsec:reduced_invariant}.

\begin{proposition}\label{prop:unreduced_splits}
When $L$ is an involutive link with fixed point set meeting at least one component of $L$, there is a splitting $\Kht^{i,j,k}(L) \cong \Khredt^{i,j+1,k}(L)\oplus\Khredt^{i,j-1,k}(L)$. In particular, such a splitting always exists for strongly invertible links. 
\end{proposition}

For strongly invertible knots and links we shall work exclusively with the reduced invariant $ \Khredt^{i,j,k}(K)$. Just as with Khovanov cohomology, this invariant comes with skein exact sequences. We use this to calculate the invariant for all strong inversions on hyperbolic knots through 7 crossings; see Figure \ref{fig:table}.  Note that for strongly invertible links the invariant is trivial in non-integral $k$-gradings.

%\labellist
%	\tiny
%	\pinlabel {$k$} at 1 64 \pinlabel {$k$} at 111 64
%	\pinlabel {$i$} at 77 2 \pinlabel {$i$} at 186 2

%	\endlabellist
%\parpic[r]{
 %\begin{minipage}{75mm}
 %\centering
 %\includegraphics[scale=0.75]{figures/intro-trefoils}
 %\captionof{figure}{Projecting the unique strong inversion on the trefoil.}
 %\label{fig:intro-trefoils}
  %\end{minipage}%
 % }
 \subsection*{Choice of involutive diagrams}We make a final remark regarding the choice of involutive diagrams used in this work. In order to incorporate an involution into the Khovanov complex such a choice of diagram was required but other choices are possible.

On the right of Figure \ref{fig:intro-trefoils} another possibility is illustrated in which the axis extends vertically from the plane of the diagram, narrowly missing puncturing the eye of the careless reader.  Although less generally applicable, this possibility is available when restricting to strongly invertible knots, or to links of which no more than one component intersects the fixed point set of $\tau$.  Furthermore, the involutive Reidemeister calculus is less involved.  Certain grading adjustments to the construction given below are required (in particular when a component intersects the axis then one seems to need to shift according to a certain winding number) although we give no details here.  Having made the appropriate adjustments in the construction, a direct calculation shows that the invariants obtained by this construction are different to those that we consider in this paper. This means that, for strongly invertible knots there are (at least) two different triply graded theories available. 

It is also possible to use techniques similar to those used in this work to endow annular Khovanov invariants with refined grading structures. This requires a slightly different setup, but seems of sufficient interest that we shall instead pursue it in a separate paper.

\labellist
\tiny
\pinlabel {$k$} at 69 72 \pinlabel {$k$} at 257 72
\pinlabel {$i$} at 147 9 \pinlabel {$i$} at 334.5 9
                \endlabellist
\begin{figure}[t]
\includegraphics[scale=0.75]{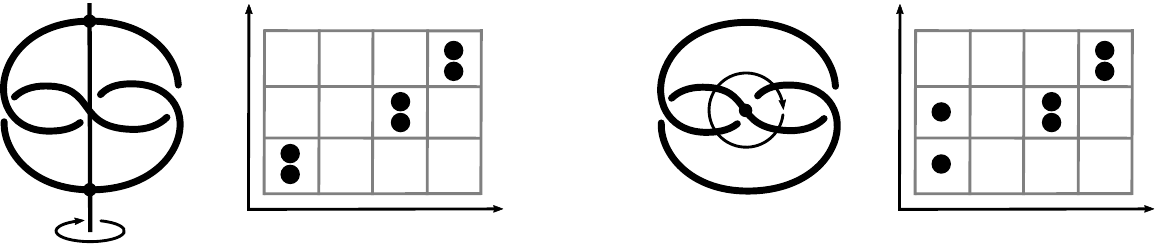}
\caption{Projecting the unique strong inversion on the trefoil in two different ways. For illustration we have only shown the relative $(i,k)$-grading, which is already enough to exhibit the difference between the constructions. The example on the left is given more attention in Section \ref{sec:examples}; the calculation on the right is left for the reader.}\label{fig:intro-trefoils}
\end{figure}

\subsection*{Acknowledgements}
This work was started at the Isaac Newton Institute while the authors were participants in the program {\em Homology Theories in Low Dimensions}, and completed during visits by the first author to the Centre Interuniversitaire de Recherches en G\'eom\'etrie et Topologie (CIRGET) in Montr\'eal and the Pacific Institute of Mathematics (PIMS) in Vancouver. The authors thank all three institutes for their support. We also thank Luisa Paoluzzi for some motivating discussions concerning mutants.

%\subsection{Involutions on diagrams}
%\label{subsec:diagraminvolutions}
%\input{diagram_involutions.tex}

\section{Involutions on links}

Our focus will be on links that are equivariant with respect to a fixed involution on the 3-sphere.  We give first a diagrammatic model for links of this type and we then relate this model to Sakuma's notion of strongly invertible knots.

\subsection{Diagrammatics of involutive links.}
\label{subsec:diagrammatics}
Recall that, due to the positive resolution of the Smith conjecture, any non-trivial orientation-preserving involution of the $3$-sphere has a connected, one dimensional fixed point set that is unknotted. For our purposes, it will be convenient to fix a standard model for such an involution: We start with a great circle inside the $2$-sphere ${\bf a} \subset S^2$, and let $R \co S^2 \rightarrow S^2$ be reflection in ${\bf a}$.  Then we get an orientation-preserving involution of $S^2 \times (-\infty,+\infty)$ by $(p,q) \mapsto (R(p), -q)$.  Compactifying by adding two points $+\infty$ and $-\infty$ we obtain our model for the involution $\tau \co S^3 \rightarrow S^3$.  

%Identifying the 3-sphere with $\R^3$, compactified by adding a point at infinity, we consider the involution that fixes the vertical axis in $\R^3$. Fix the notation $(S^3,\tau)$ for $S^3$ together with this particular involution that rotates around the vertical axis, and write $\mathbf{a}=\operatorname{fix(\tau)}$.

  \begin{definition}\label{def:involutive}
  A link $L$ is called \emph{involutive} if it is fixed, as a subset of $S^3$, by the involution $\tau$ described above and ${\bf a} \not\subset L$.  Two involutive links are equivalent if they are isotopic through involutive links.
  \end{definition}

%More colloquially, involutive links are those links in $S^3$ that can be placed in general position with respect to $(S^3,\tau)$.  I don't think that this is a common usage of the phrase "general position".

This class contains some familiar examples as special cases.

\begin{definition}
	\label{def:strongly_invertible}
	A link $L$ is called \emph{strongly invertible} if it is involutive and, additionally, $\tau$ reverses (any choice of) orientation on $L$.
\end{definition}
\noindent Equivalently, every component of a strongly invertible link meets ${\bf a}$ in exactly two points. In the case of a single component, we recover the class of strongly invertible knots. For a comprehensive study of strongly invertible knots, including a tabulation recording their symmetries, see Sakuma~\cite{Sakuma1986}.

\parpic[r]{
 \begin{minipage}{60mm}
 \centering
 \includegraphics[scale=0.75]{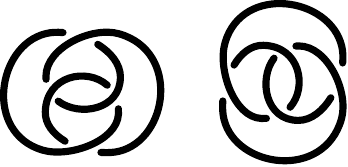}
 \captionof{figure}{Two distinct strong inversions on the figure eight knot.}
 \label{fig:2-fig-8s}
  \end{minipage}%
  }
Our definition of involutive links is well-suited to diagrammatic study.  We obtain diagrams of generic (so missing $\{ + \infty, - \infty \}$) involutive links by puncturing $S^3$ at $+\infty$ and $-\infty$ and projecting to $S^2$ (recording over-crossing and under-crossing information).  Our diagrams throughout this paper are on $S^2$ punctured at a point on ${\bf a}$.  The axis ${\bf a}$ is usually omitted in these drawings, but lies vertically---hence all our diagrams will have a vertical symmetry (as seen, for example, in Figure \ref{fig:2-fig-8s} as well as Figures  \ref{fig:mutants-small} through \ref{fig:10_71-small} of the previous section).

%We first pick a point $p \in S^3$ with $\tau(p) \not= p$.  Then we write $S^3 \setminus \{ p, \tau(p) \} = S^2 \times (-1,1)$, where 

%Given an involutive link we obtain an involutive link diagram $D$ by projecting to a plane in $\R^3$ that contains $\mathbf{a}$. It will be convenient to identify the image of $\mathbf{a}$ with the vertcial axis in $\R^2$. Such a projection should be chosen that $D$ is a link diagram in the usual sense, and the union $D\cup\operatorname{im}(\mathbf{a})$ is a link diagram away from a finite number of crossings that form triple-points with the image of the fixed point set. By setting this convention, we can provide symmetric diagrams that omit the fixed point set as in Figure \ref{fig:2-fig-8s} -- our involutive diagrams will always have a vertical symmetry, and this symmetry on the projection will also be denoted $\tau$.
There is a subtle property worth noting here: even when a given knot is amphicheiral the knot in question and its mirror need not be equivalent as strongly invertible knots. The figure eight knot ($4_1$ in Rolfsen's table) provides an example: this knot admits a pair of distinct strong inversions, shown in Figure \ref{fig:2-fig-8s}, that are exchanged on taking the mirror image.  In particular there are two equivalence classes of strongly invertible knot which are isotopic to the figure eight knot.

In order to define invariants of involutive links using involutive diagrams, we will appeal to an enhancement of the Reidemeister theorem. 
\begin{theorem}\label{thm:R}
Let $D_1$ and $D_2$ be involutive diagrams for involutive links $L_1$ and $L_2$, respectively. Then $L_1$ and $L_2$ are equivalent if and only if $D_1$ and $D_2$ are related by a sequence of moves from the list given in Figure \ref{fig:moves} (together with equivariant planar isotopy). 
\end{theorem}

\labellist
%\small
\pinlabel {IR1} at -40 388 \pinlabel {$\sim$} at 115 388
\pinlabel {IR2} at -40 340 \pinlabel {$\sim$} at 115 340
\pinlabel {IR3} at -40 272 \pinlabel {$\sim$} at 115 272
\pinlabel {R1} at -40 205 \pinlabel {$\sim$} at 115 205
\pinlabel {R2} at -40 160 \pinlabel {$\sim$} at 115 160
\pinlabel {M1} at -40 110 \pinlabel {$\sim$} at 115 110
\pinlabel {M2} at -40 65 \pinlabel {$\sim$} at 115 65
\pinlabel {M3} at -40 18 \pinlabel {$\sim$} at 115 18
	\endlabellist
\begin{figure}[ht]
\includegraphics[scale=0.75]{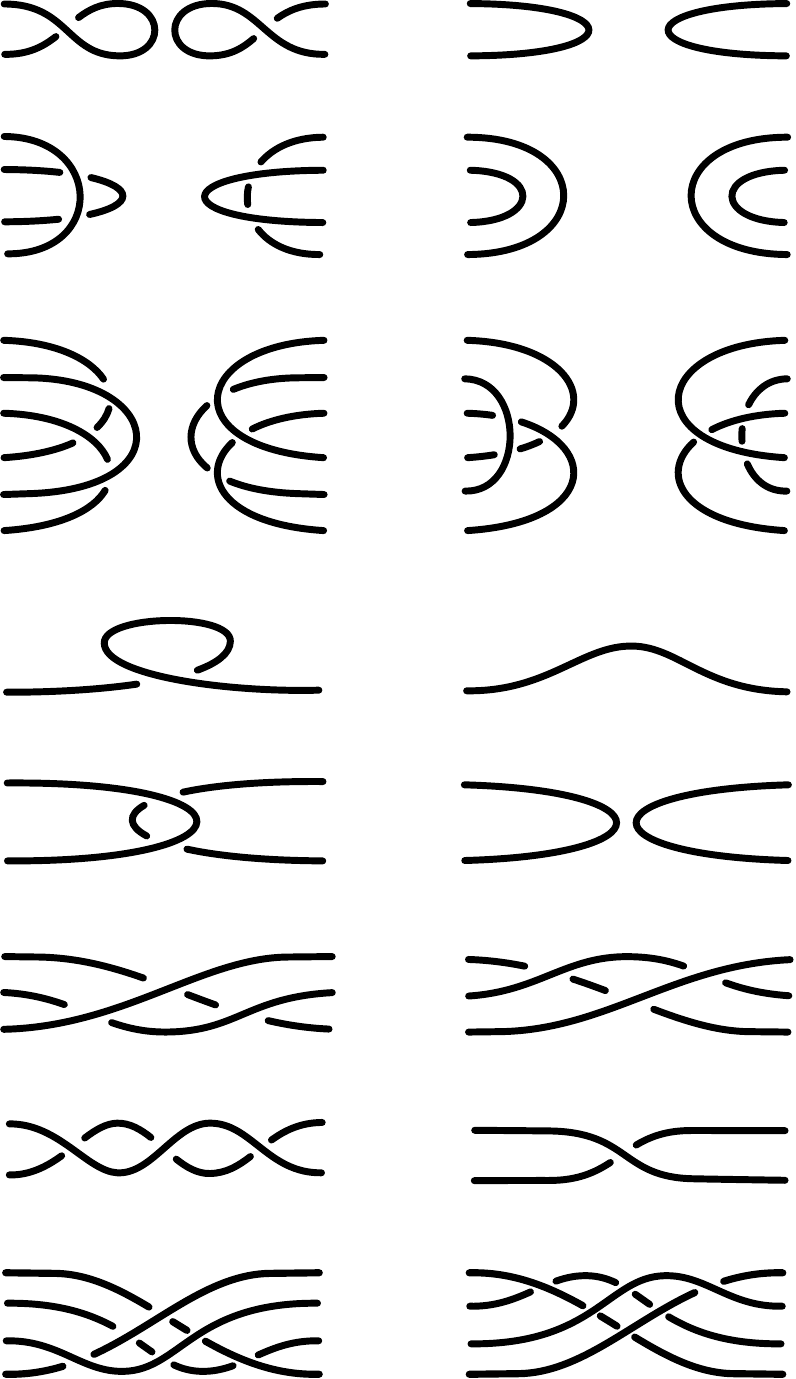}
\caption{The involutive Reidemeister moves.%, with mirrors omitted for simplicity.
}\label{fig:moves}
\end{figure}

\parpic[r]{
 \begin{minipage}{60mm}
 \centering
 \includegraphics[scale=0.75]{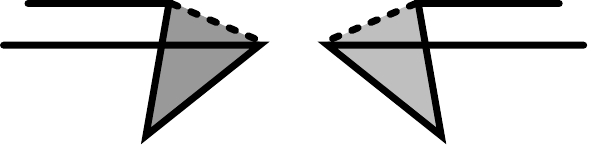}
 \captionof{figure}{The involutive Reidemeister move IR1 obtained by adding or deleting a pair of equivariant triangles.}
 \label{fig:tri-IR1}
  \end{minipage}%
  }
  We should remark here that this does not seem to appear anywhere in print, though something rather similar appears in work of Collari and Lisca \cite{CL} for a different class of symmetric diagrams. The approach is clear enough; we give a sketch.
\begin{proof}[Sketch of the proof of Theorem \ref{thm:R}] We consider a PL approximation of the diagram $D_1$, and follow Reidemeister's original proof for non-involutive diagrams \cite{Reidemeister1932}. In particular, an involutive isotopy between $L_1$ and $L_2$ may be approximated by a series of PL moves amounting to the addition or deletion of equivariant triangles in a diagram; see Figure \ref{fig:tri-IR1}, for example. The key point in Reidemeister's proof is to track the possibly interesting changes, and list the behaviour. In the case at hand, modulo the requirement that we study equivariant pairs of triangles, the proof proceeds as before and gives rise to IR1, IR2, and IR3 for triangles off the axis. For triangles in the axis, the reader can check that R1 and R2 occur (see Figure \ref{fig:moves}), but R3 does not. The interesting point is the appearance of new {\it mixed} moves involving on- and off-axis crossings. We show these in Figure \ref{fig:interesting-cases}. Once one has a list of all the possible interesting triangles, one can reduce the set of moves further by removing redundancies.  For example, changing the top crossing on the lefthand diagram and the bottom crossing on the righthand diagram in move M3 gives one such redundant involutive Reidemeister move (this one is a consequence of M2, M3, IR2, and IR3). \end{proof}

\begin{figure}[ht!]
\includegraphics[scale=0.75]{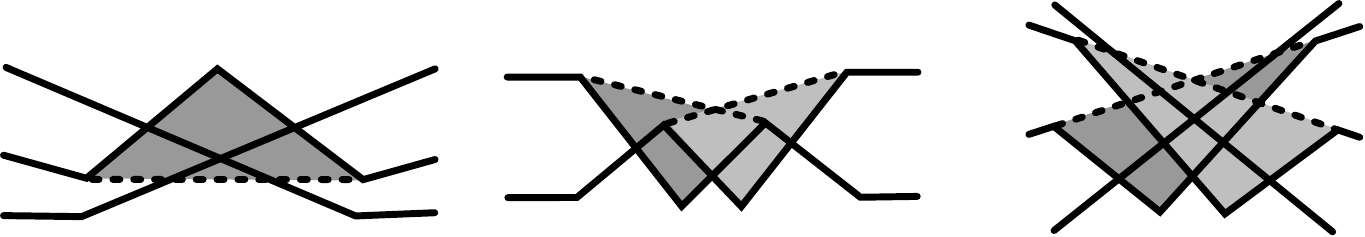}
\caption{Configurations of equivariant triangles giving rise to the mixed involutive Reidemeister moves M1, M2, and M3 (left to right).}\label{fig:interesting-cases}
\end{figure}

\subsection{Strongly invertible knots.}
\label{subsec:stronglyinvertible}
	Sakuma defines a strongly invertible knot to be a pair $(K,h)$ consisting of a knot $K \subset S^3$ and an orientation preserving involution $h\co S^3 \rightarrow S^3$ with $h(K) = K$, but with orientation reversed \cite{Sakuma1986}. Two pairs $(K,h)$ and $(K',h')$ are Sakuma-equivalent if there is an orientation preserving homeomorphism $f\co S^3\to S^3$ for which $f(K)=K'$ and $f\circ h\circ f^{-1}=h'$.  From this point of view, invariants of strongly invertible knots can be viewed as invariants of conjugacy classes of elements in the symmetry group of a knot (the mapping class group of the knot complement). Let us see how far the notion that we work with in this paper (in which the involution is fixed and we consider equivariant links up to equivariant isotopy) is equivalent to Sakuma's.
	\begin{proposition}
		\label{prop:sakumaequivalenceclasses}
		Suppose that $D$ and $D'$ are two involutive diagrams of strongly invertible knots that are Sakuma-equivalent.  Then $D$ and $D'$ are related by a finite sequence of involutive Reidemeister moves together with a global move illustrated in Figure \ref{fig:global}.
	\end{proposition}
	
	\begin{figure}[ht]
		\includegraphics[scale=0.9]{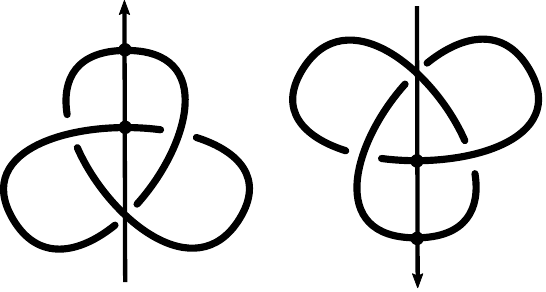}
		\caption{A global move on involutive diagrams, illustrated here for the strong inversion on the right-handed trefoil.  The move consists of a half-rotation of the 2-sphere on which the diagram is drawn.  The fixed points of the half-rotation are two antipodal points each lying on the great circle which is the image of the fixed point set ${\bf a}$ of the involution $\tau$.}\label{fig:global}
	\end{figure}
	
	Note that the invariants constructed in this paper are of course also invariant under the global move of Figure \ref{fig:global}.
	
	\begin{proof}
		First note that any two involutions of the $3$-sphere are related by conjugation by a diffeomorphism (in fact, we may assume isotopy \cite{Waldhausen1969}).  This implies that each Sakuma-equivalence class contains a pair $(K,\tau)$ where $\tau$ is the standard involution.  Then $(K,\tau)$ is Sakuma-equivalent to $(K', \tau)$ if there exists an orientation-preserving diffeomorphism of the 3-sphere $f \co S^3 \rightarrow S^3$ with $f(K) = K'$ such that $f$ commutes with $\tau$.  Let us therefore write ${\rm Diff}_\tau(S^3)$ for the orientation-preserving diffeomorphisms of $S^3$ that commute with $\tau$.  This space splits as the disjoint union of two homeomorphic spaces
		\[ {\rm Diff}_\tau(S^3) = {\rm Diff}^+_\tau(S^3) \sqcup {\rm Diff}^-_\tau(S^3) \]
		where the superscript $+$ (respectively $-$) indicates those $\tau$-equivarient orientation-preserving homeomorphisms of $S^3$ that preserve (respectively reverse) the orientation of the fixed point set~${\bf a}$.
		
		The global move on involutive diagrams is realized by a $\tau$-equivariant diffeomorphism that interchanges these two homeomorphic spaces.  Therefore we shall be done if we can see that ${\rm Diff}^+_\tau(S^3)$ is path-connected.  So suppose that $\phi \in {\rm Diff}^+_\tau(S^3)$---we begin by first showing that there is a path from $\phi$ to an element of ${\rm Diff}^+_\tau(S^3)$ which is the identity in a small neighbourhood of ${\bf a}$.  In the non-equivariant setting, this kind of construction is fairly standard and used in showing that handle attachment depends only on the isotopy class of the attaching map and the framing of the normal bundle.  However, since it is important for us that the construction works $\tau$-equivariantly, we outline the stages below.
			
		Let us pick a model for $\tau$ in a neighbourhood of ${\bf a}$.  We write the neighbourhood as ${\bf a} \times D^2$ where $D^2$ is the unit disc and $\tau$ is just rotation by $\pi$ of each $D^2$ fiber.  The framing of ${\bf a}$ given by pulling back a tangent vector to $0 \in D^2$ should corresponding the trivial framing of ${\bf a} \in S^3$.
			
		The map $\phi \vert_{\bf a} \co {\bf a} \rightarrow {\bf a}$ is an orientation preserving diffeomorphism and so is isotopic to the identity map on ${\bf a}$.  Using the radial coordinate of $D^2$ as a time coordinate along this isotopy, we can construct a path from $\phi$ to $\phi_1 \in {\rm Diff}^+_\tau(S^3)$ which is the identity when restricted to ${\bf a}$.
		
		If $p \in {\bf a}$ and $D^2_p$ is the $D^2$ fiber over $p$, then the differential
		\[ d\phi_1 \co T_p(S^3) \rightarrow T_p(S^3) \]
		is the identity when restricted to tangent vectors along ${\bf a}$ and, since $\phi_1$ is $\tau$-equivariant, is equivalent to its restriction $T_p(D^2_p) \rightarrow T_p(D^2_p)$.  Therefore we shall think of $d\phi_1 \vert_{\bf a}$ as a smooth map ${\bf a} \rightarrow {\rm GL}_2(\R)$.  We note that $d\phi_1 \vert_{\bf a}$ is isotopic to the constant loop at the identity element of ${\rm GL}_2(\R)$.  We wish to convince ourselves that we can construct some such isotopy induced by an isotopy of $\phi_1 \in {\rm Diff}^+_\tau(S^3)$.
		
		Given a loop $\gamma \co {\bf a} \rightarrow {\rm GL}_2(\R)$ we shall say that $\gamma$ is \emph{realizable} if there exists a path in ${\rm Diff}^+_\tau(S^3)$, consisting of diffeomorphisms which are the identity on ${\bf a}$ and the identity outside a neighbourhood of ${\bf a}$, between the identity and $\psi \in {\rm Diff}^+_\tau(S^3)$ with $d \psi = \gamma$.
		
		If $r,s \co S^1 \rightarrow \R$ are smooth with $r > 0$, and ${\rm U}(1)$ is rotations then we claim that the following $\gamma$ are realizable:
	\[ \gamma = \left( \begin{array}{cc} r & 0 \\ 0 & 1 \end{array} \right), \gamma = \left( \begin{array}{cc} 1 & 0 \\ 0 & r \end{array} \right), \gamma = \left( \begin{array}{cc} 1 & s \\ 0 & 1 \end{array} \right), \gamma \co {\bf a} \rightarrow {\rm U}(1) \]
		where the final $\gamma$ has winding number $0$.  In fact it is possible to realize all such $\gamma$ by paths in ${\rm Diff}^+_\tau(S^3)$ through diffeomorphisms which are fiber-preserving and are linear on each $D^2$ fibre in a neighbourhood of ${\bf a}$.  We do not present explicit constructions.
		
		By composing $\phi_1$  with a sequence of such paths in ${\rm Diff}^+_\tau(S^3)$, we obtain a path from $\phi_1$ to $\phi_2$ which satisfies that $d \phi_2$ is the constant identity path.  Finally, we linearize $\phi_2$ in a neighbourhood of ${\bf a}$ to give a path from $\phi_2$ to $\phi_3$ which is the identity on a neighbourhood of ${\bf a}$.
		
		To complete our proof, we would like to see that there is a path in ${\rm Diff}^+_\tau(S^3)$ from $\phi_3$ to the identity map, and we shall see this by passing to the quotient.  Let us write $\overline{\bf a} \subset S^3$ for the image of the fixed point set under $\tau$.  %${\rm Diff}^+_\tau(S^3 \,\, {\rm rel} \,\, n(\bf a))$ for the space of $\tau$-equivariant orientation-preserving diffeomorphisms which are the identity on some neighbourhood of ${\bf a}$, and ${\rm Diff}(S^3 \,\, {\rm rel} \,\, n(\overline{\bf a}))$ for the space of diffeomorphisms of the quotient $S^3$ fixing some neighbourhood of ${\bf a}$.
		There is an obvious homeomorphism between the space of $\tau$-equivariant orientation-preserving diffeomorphisms which are the identity on some neighbourhood of ${\bf a}$, and the space of diffeomorphisms of the quotient $S^3$ which are the identity on some neighbourhood of $\overline{\bf a}$.
		
		Finally note that there is a path from any point in the latter space to the identity map on $S^3$.  This follows from the fact that the space of diffeomorphisms of the solid torus which restrict to the identity on the boundary is known to be contractible.  Indeed, this is equivalent to the Smale conjecture, (see Hatcher \cite{Hatcher1983}, in particular statement 9 in the Appendix).
		\end{proof}

	\begin{remark}
		\label{rem:globalsakumamovenecessary}
	The authors do not know if there exists a pair of involutive diagrams $D$ and $D'$ of Sakuma-equivalent knots which are not related by some finite sequence of involutive Reidemeister moves (forgetting the global move).
	\end{remark}
	
	%We also know that ${\rm UEmb}(S^1,S^3)$ is simply connected.  Well, I hope that's true and it's not like $\pi_1 = \Z/2$ or something.  Well, reckon Ryan Budney knows it's homotopy equivalent to $S^2 \times S^2$ perhaps...  And then we're basically done by the long exact sequence of homotopy groups.
	%Firstly note that the space of self-homeomorphisms of the $3$-sphere is path-connected\footnote{reference for this perhaps?...}.  Given $(K,h)$, $(K',h')$ Sakuma-equivalent via the homeomorphism $f: S^3 \rightarrow S^3$, take a path of homeomorphism $f_t$ with ${\rm id}=f_0$ and $f = f_1$.  Then we see that $K$ is isotopic via the isotopy through $h$-equivariant links to.... um...
	
	%In other word $h, h' : S^3 \rightarrow S^3$ there is an isotopy through involutions $h_t : S^3 \rightarrow$.
	%take 
	%It is easy to see that the the two notions of strong inversions on knots are equivalent; given $(K,h)$ it suffices to find an ambient isotopy taking $\operatorname{fix}(h)$ to $\mathbf{a}$. Composing appropriately with such an isotopy through shows that the two notions of equivalence agree.

%\section{Definition of the invariant (Maps on Khovanov Complex).}
%\label{sec:maps_on_Khov_complex}
%\input{sections/maps_on_Khov_complex.tex}

\section{Definition of the invariant and proof strategy}
\label{sec:hom_alg}
In this section we give the formal definition of our invariant and collect some homological algebra necessary to the proof of invariance.

\subsection{Structure of the main and ancilliary invariants}
The main invariant takes the form of (the isomorphism class of) a finite dimensional triply-graded vector space over $\bF$
\[ \Kht(L)\cong\bigoplus_{i,j\in\Z; k \in \frac{1}{2}\Z}\Kht^{i,j,k}(L) \]
as in the statement of Theorem \ref{thm:triply-graded-invariant}.
The reader will note that for general involutive links, the grading $k$ runs over the half-integers.  However for strongly invertible links (our main object of study) we shall see that the invariant vanishes for non-integral values of $k$; see Remark \ref{rem:k_grading_integral}.

This invariant is the associated graded vector space to a singly-graded and bifiltered vector space.  It will be seen, in fact, that our whole homological construction splits along the quantum grading (denoted $j$).  Hence the reader will readily see how the $j$-grading arises.

On the other hand, how the bifiltration arises and how it is seen to be an invariant are both a little unusual.  We take some time in the rest of this section in laying out the formal algebraic structure which underlies the invariant and our proof of invariance.

\subsection{Bifiltrations and associated graded spaces}
%\begin{definition}
%	\label{defn:filtration}
%	Suppose that $V$ is a finite dimensional vector space over $\bF$.  An increasing \emph{filtration} on $V$ is a nesting of subspaces of $V$
%	\[ \cdots \subseteq \Kf^{l-1} V \subseteq \Kf^l V \subseteq \Kf^{l+1} V \subseteq \cdots \]
%	such that
%	\[ \bigcap_{l \in \Z} \Kf^l V = 0 \,\,\, {\rm and} \,\,\, \bigcup_{l \in \Z} \Kf^l V = V  {\rm .} \]
%\end{definition}

%We note that one often also considers filtrations in which the nesting goes in the other direction with increasing index.  From the point of view of the homological algebra to follow, this is merely a cosmetic change.
In this paper we often consider vector spaces which carry two filtrations simultaneously.  We have chosen the nesting conventions of such \emph{bifiltrations} to match with the natural indices and filtration directions of our diagrammatic constructions.

\begin{definition}
	\label{defn:bifiltration}
Suppose that $V$ is a finite dimensional vector space over $\bF$.  A \emph{bifiltration} on $V$ is two nestings of subspaces of $V$
\[ \cdots \subseteq \F^{i-1} V \subseteq \F^i V \subseteq \F^{i+1} V \subseteq \cdots \]
and
\[ \cdots \subseteq \G^{k+\frac{1}{2}} V \subseteq \G^k V \subseteq \G^{k-\frac{1}{2}} V \subseteq \cdots \]
such that
\[ \bigcap_{i \in \Z} \F^i V = 0 = \bigcap_{k \in \frac{1}{2}\Z} \G^k V \,\,\, {\rm and} \,\,\, \bigcup_{i \in \Z} \F^i V = V = \bigcup_{k \in \frac{1}{2}\Z} \G^k V {\rm .} \]
\end{definition}

The reader will notice that the $\F$ filtration and the $\G$ filtration go in different directions with increasing index, and she will also notice that the $\G$ filtration is indexed by half integers.

\begin{definition}
	\label{defn:bifiltered_maps}
	%A linear map $\phi : V \rightarrow \widetilde{V}$ between two filtered vector spaces is said to be \emph{filtered} if it satisfies
	%\[ \phi (\Kf^l V) \subseteq \Kf^l \widetilde{V}  {\rm .} \]
	%A filtered map is called a \emph{filtered isomorphism} if it has a filtered inverse.
	
	A linear map between bifiltered vector spaces is called \emph{bifiltered} if it is filtered with respect to both filtrations, and is called a \emph{bifiltered isomorphism} if it has a bifiltered inverse.
\end{definition}

\begin{definition}
	\label{defn:associated_graded}
	%Suppose that $V$ is a filtered vector space as in Definition \ref{defn:filtration}.  We define the \emph{associated graded} vector space to the filtration to be the quotient space
	%\[ V^{l} = \Kf^l V / \Kf^l V {\rm ,} \]
%	for $l \in \Z$.
	Suppose that $V$ is a bifiltered vector space as in Definition \ref{defn:bifiltration}.  We define the \emph{associated graded} vector space to the bifiltration to be the quotient vector space
	\[ V^{i,k} = \frac{\F^i V \cap \G^k V}{ (\F^i V \cap \G^{k + \frac{1}{2}} V) + (\F^{i-1} V \cap \G^k V) } \]
	for $i \in \Z$ and $k \in \frac{1}{2} \Z$.
\end{definition}

Note %that
%\[ \dim(V) = \sum_{i,k} \dim(V^{i,k}) \]
%and
that bifiltered vector spaces are classified up to bifiltered isomorphism by the dimensions of the graded pieces of their associated graded vector spaces.

\subsection{Spectral sequences and morphisms of spectral sequences}
\label{subsec:maps_of_spec_sequences}
We turn now to spectral sequences.  These arise in homological algebra when a cochain complex carries a filtration.  We summarize the results that we need.  Perhaps the most unusual aspect of the set-up is that there is no cohomological grading that is respected by the differentials, otherwise the results are standard.  We work here with increasing filtrations indexed by the integers, but the same results apply \emph{mutatis mutandis} to decreasing filtrations indexed by the half-integers.

\label{subsec:spec_sequences}
Suppose that $C$ is a finite dimensional vector space carrying a filtration
\[ \cdots \subseteq \Kf^{l-1} C \subseteq \Kf^l C \subseteq \Kf^{l+1} C \subseteq \cdots, \,\,\, \bigcup_l \Kf^l C = C,\,\,\, \bigcap_l \Kf^l C = 0 {\rm .} \]
Further suppose that there is a differential $\partial : C \rightarrow C$ (so satisfying $\partial^2 = 0$), and let $L$ be an integer satisfying
\[ \partial(\Kf^l C) \subseteq \Kf^{l + L} C \]
(the reader will note that finite dimensionality ensures that \emph{some} such $L$ exists).

In such a situation one obtains a spectral sequence abutting to the cohomology of $(C,\partial)$.  The $E^1$ page of this spectral sequence consists of the associated graded groups to the filtration $E^{1,l}_\Kf = \Kf^l C / \Kf^{l-1} C$ together with the differential $\partial^1_\Kf \co E^{1,l}_\Kf \rightarrow E^{1,l+L}_\Kf$ induced by $\partial$.  Then for $n \geq 2$, the $E^n$ page of the spectral sequence is obtained as the cohomology of the $E^{n-1}$ page (so that $E_\Kf^{n,l}$ is a subquotient of $E_\Kf^{n-1,l}$).

%a construction of which we now describe explicitly, albeit briefly, for the benefit of the reader.

%Firstly, for each $l \in \Z$ choose a finite subset $B_l \subset \Kf^l C$ so that the image of $B_l$ under projection to $\Kf^l C / \Kf^{l-1} C$ has the same cardinality as $B_l$ and is a basis for $\Kf^l C / \Kf^{l-1} C$.  We call such a collection of subsets a \emph{homogeneous basis} for $C$.

%Then $d$ is expressible as a sum
%\[ d = d_L + d_{<L} \]
%where, for all $l$,
%\[ d_L(\langle B_l \rangle) \subseteq \langle B_{l+L} \rangle \, \, \, {\rm and} \, \, \, d_{<L}(\langle B_l \rangle) \subseteq \langle \bigcup_{K<L} B_{l + K} \rangle {\rm .} \]

%Then the composition of ${d_L^2}\vert_{\langle B_l \rangle}$ with projection to $\langle B_{l+2L} \rangle$ agrees with the  $d_L^2 = 0$.

Since $C$ was assumed to be finite dimensional, there exists some $N\gg0$ such that $E_\Kf^{n+1,l} \equiv E_\Kf^{n,l}$ for all $l$ and all $n \geq N$.  We define $E_\Kf^{\infty,l} \equiv E_\Kf^{N,l}$ and we have
\[ E_\Kf^{\infty,l} \equiv \Kf^l H / \Kf^{l-1} H \]
where $H = \ker (\partial) / {\rm im} (\partial)$ and we write $\Kf$ also for the induced filtration on $H$.

Suppose that $C$ and $\overline{C}$ are two filtered vector spaces carrying filtered differentials
%as in Subsection \ref{subsec:spec_sequences},
and suppose that $\phi \co C \rightarrow \overline{C}$ is a filtered cochain map (so $\phi(\Kf^l C) \subseteq \Kf^l \overline{C}$ for all $l$).

Then $\phi$ induces cochain maps $\phi^n \co E_\Kf^{n,l} \rightarrow \overline{E}_\Kf^{n,l}$ at each page of the spectral sequences and $\phi^{n+1}$ is the map induced on cohomology by $\phi^n$ for each $n \geq 1$. These three statements are equivalent:
\begin{itemize}
	\item The cochain map $\phi$ induces a filtered map $H \rightarrow \overline{H}$ which has a filtered inverse.
	\item The cochain map $\phi$ induces isomorphisms between the associated graded vector spaces to the cohomology
	\[  H^l =  \Kf^l H / \Kf^{l-1} H \rightarrow \Kf^l \overline{H} / \Kf^{l-1} \overline{H} = \overline{H}^l {\rm .} \]
	\item There exists some $n$ such that $\phi$ induces isomorphisms $\phi^n \co E_\Kf^{n,l} \rightarrow \overline{E}_\Kf^{n,l}$.
\end{itemize}

\subsection{Involutive diagrams and involutions on the Khovanov complex}
\label{subsec:involutive_diagrams}
Now we turn to the specific bifiltered cochain complexes that we consider in this paper.
%Let $h : S^3 \rightarrow S^3$ be an orientation preserving involution on the 3-sphere.  Up to self-diffeomorphism of $S^3$, such an involution is unique.  The fixed point set (or \emph{axis}) of $h$ is then an unknotted circle.  In this paper we consider oriented links $L \subset S^3$ fixed as sets under such an involution $h(L) = L$, and so that the fixed point set of $h$ does not form a component of $L$.
Let $L$ be an involutive link, and let $D$ be an involutive diagram for $L$.

Khovanov cohomology $\Kh^{i,j}(L)$ of the link $L$ (for our purposes) is (the isomorphism class of) an integer bigraded vector space over $\mathbb{F} = \Z/2\Z$.  This may be computed as the cohomology of a bigraded cochain complex associated to the diagram $D$ of $L$:
\[ d \colon \CKh^{i,j}(D) \rightarrow \CKh^{i+1,j}(D) {\rm .} \]

There is an involution, which we also refer to as $\tau$, on the set of smoothings of such an involutive diagram $D$, given by reflecting any smoothing in the axis ${\bf a}$.  Giving a smoothing $s$ of $D$, $\tau$ induces a bijection between the components of $s$ and the components of $\tau(s)$.
%Note that this bijection need not be the identity map in the case that $s = \tau(s)$.

This involution $\tau$ gives rise to a map, which we also refer to as $\tau\co \CKh^{i,j}(D) \rightarrow \CKh^{i,j}(D)$.  This is because $\CKh(D)$ splits as a direct sum in which each summand corresponds to a smoothing.  The summands corresponding to $\tau$ and to $\tau(s)$ can then be identified by the permutation of tensor factors according to the bijection of components induced by $\tau$.
Since $\tau^2 = {\rm id}$, we have that $d_\tau := \tau + {\rm id}$ is a differential.

We define $\CKh^j(D) = \oplus_{i} \CKh^{i,j}(D)$.  Since $\tau$ commutes with $d$ we have that
\[ \partial := d + d_\tau \colon \CKh^j(D) \rightarrow \CKh^j(D) \]
is a differential.

\begin{definition}
	\label{defn:perturbed_cochain_complex}
	We write $\CKht^j(D)$ for the cochain complex with the differential $\partial$.
\end{definition}

Pending further refinements, we now give the first invariant, which takes the form of a singly-graded finite dimensional vector space.

\begin{definition}
	\label{def:total_cohomology}
We define the $\bF$ vector space $\Ht^j(D)$ to be the cohomology of $\CKht^j(D)$.
\end{definition}

In fact,% as we shall see in the next subsection, 
this vector space comes equipped with a bifiltration, giving rise to a triply-graded refinement, as well as to two invariant spectral sequences.

\subsection{Two filtrations on the cochain groups and the invariant.}
\label{subsec:filtration_definition}
We now give two filtrations on the cochain group $\CKht^j(D)$.  The first one is just a filtration induced by the usual cohomological grading.

\begin{definition}
	\label{defn:F-filtration}
	For $i \in \Z$ we define
	\[ \F^i \CKht^j(D) = \bigoplus_{l \leq i} \CKh^{l,j}(D) \subseteq \CKht^j(D) {\rm .} \]
\end{definition}

We shall work a little harder to define our second filtration.
	Firstly, we can write
	\[ \CKht^j(D) = \bigoplus_{s} V(s) \]
	where the sum is taken over all smoothings $s$ of $D$ and $V(s)$ is the quantum degree $j$ part of the chain group summand corresponding to $s$.
	
\begin{definition}
	\label{defn:kasagrading}
	We associate a half-integer weight $w(s) \in \frac{1}{2} \Z$ to each smoothing $s$ by summing up local contributions at each crossing of $D$.  The local contributions fall into three types corresponding to the type of crossing being smoothed.  We write $o(c)$ for the oriented smoothing of a crossing and $u(c)$ for the anti-oriented smoothing of a crossing.
	
	\begin{itemize}
		\item[(1)] Suppose the crossing $c$ is off the axis of symmetry.  Then the local contribution of $o(c)$ is $0$.  If $c$ is positive (respectively negative) then $u(c)$ contributes $\frac{1}{2}$ (respectively $-\frac{1}{2}$).
		
		\item[(2)] Suppose the crossing $c$ is on axis and the involution reverses the orientations of $c$.  Then the local contribution of $o(c)$ is $0$.  If $c$ is positive (respectively negative) then $u(c)$ contributes $1$ (respectively $-1$).
		
		\item[(3)] Finally, suppose the crossing $c$ is on axis and the involution preserves the orientations of $c$.   If $c$ is positive then the local contribution of $o(c)$ is $-\frac{1}{2}$ and that of $u(c)$ is $\frac{1}{2}$.  If $c$ is negative then the local contribution of $o(c)$ is $\frac{1}{2}$ and that of $u(c)$ is $-\frac{1}{2}$.
	\end{itemize}

This gives the $k$\emph{-grading} on $\CKht^j(D)$, where we say that $v \in V(s)$ is of $k$-grading $w(s)$.
\end{definition}

\begin{remark}
\label{rem:k_grading_integral} The reader will note two special features when considering a diagram of a strong inversion on a link. First, case (3)  of the definition above does not arise, and second, the oriented resolution always has weight 0.
\end{remark}

We now take a filtration associated to the $k$-grading defined above.
	
\begin{definition}
		\label{defn:G-filtration}
	For $k \in \frac{1}{2} \Z$ we define
	\[ \G^k \CKht^j(D) = \bigoplus_{w(s) \geq k} V(s) \subseteq \CKht^j(D) {\rm .} \]
\end{definition}

With these two filtrations on the cochain complex, we can now define the triply-graded version of our invariant.

\begin{definition}
The filtrations $\F$ and $\G$ induce a bifiltration on $\Ht^j(D)$, and $\Kht^{i,j,k}(D)$ is defined as the associated graded space to this bifiltration.
\end{definition}

Furthermore, associated with each filtration of the cochain complex is a spectral sequence, and each of these is an invariant (after a certain page).

\begin{remark}
	\label{rem:kisreallyanintegergrading}
	It turns out that although the $\G$ filtration is indexed by half-integers, all non-trivial differentials in the associated spectral sequence are integer graded.  Thus we write $E^2_\G$ for the page resulting from differentials of grading $0$, and then $E^3_\G$ for the page resulting from the differentials of grading $-1$ and so on.
\end{remark}

%In the next subsection we outline the general strategy for our proof of invariance of the isomorphism class of $\Kht$.

\subsection{Proof of invariance of filtrations: the strategy}
\label{subsec:outline_invariance_proof}
%We wish to show invariance under the involutive Reidemeister moves given in Figure \ref{fig:moves}.
Suppose then that $D$ and $D'$ are involutive diagrams differing by a single involutive Reidemeister move; compare Figure \ref {fig:moves} and, in particular, Theorem \ref{thm:R}.  We shall first give a cochain map 
\[ f \co \CKht^j(D_1) \rightarrow \CKht^j(D_2) \]
with respect to the total differential $\partial$ that preserves the $\F$ filtration
\[ f(\F^i \CKht^j(D_1)) \subseteq \F^i \CKht^j(D_2) {\rm .} \]

The map $f$ is constructed so that it induces an isomorphism $f^2_\F \co E^2_\F(D) \rightarrow E^2_\F(D')$ on the second pages of the spectral sequences coming from the $\F$ filtration.  In fact, $f^2_\F$ will be a composition of Khovanov isomorphisms between the Khovanov cohomologies of link diagrams differing by a sequence of Reidemeister moves.

This tells us that $f$ induces an $\F$-filtered isomorphism $f^* \co \Ht^j(D) \rightarrow \Ht^j(D')$.

However, we shall find that the cochain map $f$ does not necessarily preserve the $\G^k$ filtrations.  We then give a map
\[ h \co \CKht^j(D) \rightarrow \CKht^j(D') \]
and define
\[ g = f + \partial h + h \partial {\rm .} \]

It follows that $g$ is necessarily a cochain map, and induces the same map on cohomology as that induced by $f$
\[ g^* = f^* \colon \Ht^j(D) \rightarrow \Ht^j(D') {\rm .}\]
Furthermore, by choosing $h$ carefully, we are able to achieve that $g$ preserves the $\G^k$ filtration
\[ g(\G^k \CKht^j(D)) \subseteq \G^k \CKht^j(D') {\rm .} \]
We then check that $g$ induces an isomorphism $g^3_\G \co E^3_\G(D) \rightarrow E^3_\G(D')$, telling us that $g$ induces a $\G$-filtered isomorphism $g^* \co \Ht^j(D) \rightarrow \Ht^j(D')$.

Thus we have that the map $g^* = f^* \co \Ht^j(D) \rightarrow \Ht^j(D')$ is a bifiltered isomorphism.

\subsection{The reduced invariant.}
\label{subsec:reduced_invariant}
The invariants of an involutive link $L$ that we have been describing so far have been the \emph{unreduced} invariants $\Ht^j(L)$ and $\Kht^{i,j,k}(L)$.  If $L$ has a component that intersects the axis of symmetry then choosing one of the points of intersection as a basepoint allows us to define \emph{reduced} invariants $\Htredj(L)$ and $\Khredt^{i,j,k}(L)$ in a similar fashion to the reduced version of Khovanov cohomology.  As a consequence of working over $\bF$, the unreduced invariant splits as the direct sum of two copies of the reduced invariant so that, up to isomorphism, the choice of component or of basepoint does not matter.  %We discuss this splitting in the remainder of this subsection, but nothing here will come as a surprise to the seasoned knot homologist.

Suppose that $L$ is an involutive link with involutive diagram $D$ with a basepoint as described above.  We consider the cochain complex $\CKht^j(D)$ with respect to the total differential $\partial = d + d_\tau$.  We may write
\[ \CKht^j(D) = {\rm cone}(\Phi) \]
where
\[ \Phi : \CKhtred^j(D)\{1\} \rightarrow \CKhtred^j(D)\{-1\} \]
is a cochain map.  Here we write $\CKhtred^j(D)\{-1\}$ (respectively $\CKhtred^j(D)\{1\}$) for the subcomplex (respectively quotient complex) of $\CKht^j(D)$ in which the basepointed component of each smoothing is decorated with a $v_-$ (respectively $v_+$).  The curly brackets denote a shift in the quantum degree.%Note that $\Phi$ is bifiltered of degree $(1,0)$.

\begin{definition}
	\label{def:reduced_invariant}
	$\Khredt^{i,j,k}(D)$ is the triply graded $\bF$ vector space given as the associated graded space to the bifiltered singly graded cohomology $\Htredj(D)$ of $\CKhtred^j(D)$.
\end{definition}

%\begin{proposition}
%	\label{prop:unreduced_splits}
%	The cochain map $\Phi$ defined above is homotopic via a bifiltered degree $(0,0)$ chain homotopy to the zero map.
%\end{proposition}

\begin{proof}[Proof of Proposition \ref{prop:unreduced_splits}.]
	The statement of the proposition follows from the observation that the cochain map $\Phi$ defined above is homotopic via a bifiltered degree $(0,0)$ chain homotopy to the zero map.
	
	Indeed, define the bilfiltered degree $(0,0)$ chain homotopy
	\[ h \co \CKhtred^j(D)\{1\} \rightarrow \CKhtred^j(D)\{-1\} \]
	on the standard basis for $\CKhtred^j(D)\{1\}$ in the following way.  Suppose that $v \in \CKhtred^j(D)\{1\}$ is a basis element thought of as a decorated smoothing of $D$ (in which the basepointed component is decorated with a $v_+$).  We define $h(v)$ to be the sum of all those decorations of the same smoothing as $v$, of the same quantum degree as $v$, and whose decorations differ on exactly two components of the smoothing (necessarily including the basepointed component).
	
	Then we observe that
	\[ \Phi = \partial h + h \partial \]
	so that $\Phi$ is chain homotopic to the zero map.  In fact, note that $h$ commutes with $d_\tau$ so that this observation amounts to $\Phi = dh + hd$, which is straightforward to verify.
\end{proof}

All results derived in this paper are for convenience stated and proved for the unreduced case.  But for every result there is a reduced version which may either be proved directly \emph{mutatis mutandis}, or derived from Proposition \ref{prop:unreduced_splits} and the unreduced statement.  In particular there is a spectral sequence from the reduced Khovanov cohomology of an involutive link $L$ (with a basepoint on the axis) converging to the associated graded space to the $\F$ filtration on $\Htredj(L)$.

\section{The $E^2_\G$ page.}
\label{sec:E2G}

Let $D$ be an involutive link diagram.  In this section we give an explicit basis for the $E^2_\G (D)$ page of the spectral sequence associated to the $\G$-filtration.  We further describe the differential on this page with respect to this basis.  This will be important to us both for our proofs of invariance and for making calculations for particular links and classes of link.

\subsection{Gauss elimination and spectral sequences.}
\label{subsec:Gauss_elim}
A useful way both to think about cohomology and to understand spectral sequences induced by filtrations is via the process of repeated Gauss elimination.  While this may be familiar to many, since we shall often refer to the process in the sequel we use this subsection to give a description for easy reference.

We state the following lemma, which is adapted to our context, without proof.

\begin{lemma}[\bf Gauss elimination]
	\label{lem:Gausselimination}
	Suppose that $A_i$ is a finite dimensional $\bF$ vector space for $1 \leq i \leq 3$, and further suppose that $f_{ij} \colon A_i \rightarrow A_j$ is a linear map for all $1 \leq i,j \leq 3$.
	We define the map
	\[ D \colon A_1 \oplus A_2 \oplus A_3 \rightarrow A_1 \oplus A_2 \oplus A_3 {\rm .} \]
	by the matrix $(f_{ij})$.
	Then if $D^2 = 0$ and $f_{23}$ is an isomorphism, there is a homotopy equivalence between the  complex $(A_1 \oplus A_2 \oplus A_3, D)$ and the  complex $(A_1, f_{11} + f_{21} \circ f_{23}^{-1} \circ f_{13})$.
	We say that the latter complex is obtained from the first by \emph{Gauss elimination along}~$f_{23}$.
\end{lemma}

Given a complex $(A, D)$ where $A$ is a finite dimensional $\bF$ vector space, there exists a sequence of Gauss eliminations from which one obtains a final complex for which the induced differential is trivial.  This final complex is then, of course, isomorphic to the cohomology of the original complex. Let us now bring a filtration into the picture to see how one arrives at a spectral sequence using Gauss elimination.

Suppose that $V = \oplus_{l \in \Z} V_l$ is a finite dimensional vector space graded by the index $l$.
%We write
%\[ i_l \colon V_l \rightarrow V, \,\, {\rm and} \,\,  \pi_l \colon V \rightarrow V_l \]
%for the inclusion and projection maps respectively.
We may consider the filtration on the vector space given by
\[ \LL^l V = \oplus_{n \leq l} V_n {\rm .} \]
Suppose further that we have a filtered differential on $V$:
\[ D \colon V \rightarrow V, \,\,\, D^2 = 0, \,\,\, D(\LL^l V) \subseteq \LL^{l + s} V {\rm .} \]

As a vector space, the first page $E^1_\LL$ of the spectral sequence associated with the filtered differential is simply the associated graded vector space to the filtration.  To understand the differential on this page, pick a basis for each summand $V_l$; this choice gives a homogeneous basis for $V$.
%Note that $\pi_{l+s} \circ D \circ i_l$ is a differential of degree $s$.
If $D$ has a non-trivial component of degree $s$, we choose a pair of basis elements $x \in V_l$, $y \in V_{l+s}$ with $x \not= y$ such that the matrix of $D$ has a non-zero entry in the place corresponding to $(x,y)$.  Then we perform Gauss elimination along the component of the differential $D$ corresponding to this matrix entry, arriving at a complex homotopy equivalent to $(V,D)$ but of dimension that has been lowered by $2$.

We repeat this process eliminating pairs of basis elements until the induced differential has no components of degree $s$.  The resulting vector space is the $E^2_\LL$ page of the spectral sequence.  In general, to obtain the $E^{n+1}_\LL$ page from the $E^n_\LL$ page we Gauss eliminate components of the differential of degree $s - n +1$.

\subsection{The $\G$ filtration and the second page of the spectral sequence.}
In this subsection we give an explicit description of the second page $E^2_\G$ of the spectral sequence corresponding to the $\G$ filtration.

To use the language of Gauss elimination, we first need a basis which is homogeneous with respect to the $k$-grading used to define the $\G$ filtration (see Definition \ref{defn:kasagrading}).  We use the standard basis elements of $\CKht^j(D)$.

\begin{definition}
	\label{defn:stdbasisofKhovanov}
	Recall that the standard basis elements of $\CKht^j(D)$ are given by decorations of smoothings of $D$ in which each component of the smoothing is decorated with a $v_+$ or a $v_-$.  We write $B$ for this basis.
	We write ${\rm Sym}B$ for the subset
	\[ {\rm Sym}B = \{ v \in B : \tau v = v \} \]
	and we write ${\rm ASym}B$ for a choice of subset
	\[ {\rm ASym}B \subseteq B \setminus {\rm Sym}B \]
	such that
	\[ B = {\rm Sym}B \cup {\rm ASym} B \cup \tau {\rm ASym} B \]
	and such that ${\rm ASym}B$ and $\tau {\rm ASym}B$ are disjoint. 
\end{definition}

Now we observe that $E^2_\G(D)$ is the cohomology of $E^1_\G = \CKht^j(D)$ under the components of the differential $\partial$ that preserve the $k$-grading (see Definition \ref{defn:kasagrading}).  Or, in other words, $E^2_\G(D)$ is the cohomology of $\CKht^j(D)$ under the differential $d_\tau = {\rm id} + \tau$.

	%If $s$ is a smoothing of $D$ we write $V(s)$ for the summand of the cochain group $\CKh(D)$ corresponding to $s$.  We further write $B(s)$ for the standard unordered basis of $V(s)$ (a standard basis element consists of a choice of $v_+$ or $v_-$ decorating each component of $s$).
	
	%If $s$ is a smoothing such that $\tau s = s$, we write ${\rm Sym} B(s) = \{ v \in B(s) : \tau(v) = v \}$.  Note that here, as in Section \ref{sec:hom_alg}, $\tau$ stands for both the involution on the set of smoothings and for the involution on $\CKh(D)$.

We perform successive Gauss elimination to cancel all pairs of basis elements $(x, \tau x)$ where $x \in {\rm ASym}B$.  After doing this we arrive at the second page of the spectral sequence $E^2_\G$ with a basis which we have identified with ${\rm Sym}B$.

\begin{remark}
	\label{rem:Ggradingis0mod1afterE2}
	Note that elements of ${\rm Sym}B$ are of homogeneous $k$-gradings and that all have the same half-integer $k$-grading modulo $1$.  We shall therefore be writing $E^3_\G$ for the result of taking cohomology with respect to the differentials of grading $1$ on the $E^2_\G$ page (rather than those of grading $\frac{1}{2}$ which are necessarily trivial).
\end{remark}

It remains to determine the differential on the $E^2_\G$ page.  In order to state this, we will require some further notation. We write $\tqft$ for the graded $1+1$ dimensional TQFT used in the definition of the Khovanov cochain complex.% To state this simply we first require a little notation.

\begin{definition}
	\label{def:KhovanovTQFT}
	Suppose that $D$ is an involutive link diagram and $s$ is a smoothing of $D$ such that $\tau s = s$.  We write $X_s$ for the set of all smoothings obtained by surgering a single on-axis $0$-smoothing of $s$ to a $1$-smoothing.  Furthermore we write $Y_s$ for the set of all smoothings obtained by surgering exactly two off-axis $0$-smoothings of $s$ to $1$-smoothings.  The trace of each surgery gives a $2$-manifold and applying $\tqft$ gives a map
	\[ \tqft_s \colon \tqft(s) \rightarrow \bigoplus_{r \in X_s} \tqft(r) \oplus \bigoplus_{t \in Y_s} \tqft(t)\]
\end{definition}

\begin{proposition}
	\label{prop:E2G_differential}
	Suppose that $v \in {\rm Sym} B \subset E^2_\G$ is given by a decoration of the smoothing $s$.  Let
	\[ \pi \colon \CKht^j(D) \rightarrow \langle {\rm Sym}B \rangle \]
	be projection. Then \[ \partial_\G^2 (v) = \pi\circ\tqft_s(v)\] where $\partial_\G^2$ denotes the differential on $E_\G^2$.
	%Then if we write $\partial_\G^2$ for the differential on $E_\G^2$ we have \[ \partial_\G^2 (v) = \pi\circ\tqft_s(v) {\rm .} \]
\end{proposition}

\begin{proof}
	We have performed Gauss elimination to remove all degree $0$ differentials on $\CKht^j(D)$, resulting in a complex with a basis identified with ${\rm Sym}B$.
	We need now to determine the degree $1$ part of the differential on $\langle {\rm Sym}B \rangle$ obtained by Gauss elimination.
	
	How the induced differential is obtained is explained in Lemma \ref{lem:Gausselimination}.  In the language of that lemma, the differential on the complex obtained via Gauss elimination is given by the formula $f_{11} + f_{21} \circ f_{23}^{-1} \circ f_{13}$.  Note that in our case, the original differential has the form $\partial = d + d_\tau$ where $d$ is the Khovanov differential and is of positive degree, while $d_\tau = {\rm id} + \tau$ is of zero degree.  As a consequence of our choice of pairs of elements to Gauss eliminate, in any Gauss elimination that we perform both $f_{21}$ and $f_{13}$ have positive degree.  Hence the degree 1 components of the differential following an elimination are those degree $1$ components of $f_{11}$ and the new components arising from pairs of degree $\frac{1}{2}$ components of $f_{21}$ and $f_{13}$.
	
	The degree $1$ components of the original differential which survive the Gauss eliminations give exactly the contribution to $\partial^2_\G$ corresponding to $\bigoplus_{r \in X_s} \tqft(r)$ in Definition \ref{def:KhovanovTQFT}.  The degree $1$ components arising from pairs of degree $\frac{1}{2}$ are those corresponding to $\bigoplus_{t \in Y_s} \tqft(t)$.
\end{proof}

\begin{remark}
	\label{rem:couture}
	Couture has given a bigraded invariant of a \emph{signed divide} \cite{Couture2009}.  A signed divide is another way of packaging the information of a strongly invertible link (an involutive link each of whose components intersect the fixed-point set of the involution).  Having set up a dictionary between signed divides and corresponding involutive diagrams, one may observe that Couture's invariant is exactly $E^3_\G$---which comes with the bigrading $(j,k)$---and is computed from a complex isomorphic to $E^2_\G$.  Understanding the relationship between Couture's invariant and the Khovanov cohomology of the underlying link provided some of the motivation for the current paper.
\end{remark}

\subsection{A concrete description of the second page of the spectral sequence.}
\label{subsec:concreteE2G}
The reader may appreciate a more concrete description of $E^2_\G$, formulated along the lines of the original description of the Khovanov cochain complex.  In this subsection we give such a description, extracted from Proposition \ref{prop:E2G_differential}.

Let $D$ be an involutive link diagram, with $m$ on-axis crossings and $2n$ off-axis crossings.  We are going to consider a $(m+n)$-dimensional cube $[0,1]^{m+n}$ of which each vertex is decorated by a $\tau$-equivariant smoothing of $D$ (notice that there are $2^{m+n}$ such smoothings).

Each on-axis crossing and each pair of off-axis crossings exchanged by $\tau$ corresponds to one of the coordinates of the cube.  All vertices of the cube with $i$th coordinate $0$ (respectively $1$) should be decorated with the smoothing of $D$ in which the $i$th crossing or pair of crossings are given the $0$-smoothing (respectively the $1$-smoothing).

The smoothings at two vertices that are connected by an edge of the cube are related either by a single surgery (if the coordinate changing along the edge corresponds to an on-axis crossing) or by a double surgery (if the coordinate changing corresponds to a pair of off-axis crossings).

Each vertex of the cube gives a cochain group summand with generators being all $\tau$-equivariant decorations of components of the smoothing at that vertex with either $v_+$ or $v_-$.  This can be thought of as a subspace of the vector space $\langle v_+, v_- \rangle^{\otimes \# {\rm components}}$.  In the familiar way, each cochain group is formed by taking the direct sum of all such summands whose coordinate-sum (a number between $0$ and $m+n$) is the same.  The $k$-grading of a cochain group is given by the coordinate sum plus a shift of $-(m_- + n_-/2 + m_0/2)$ where $m_-$ is the number of on-axis negative crossings with orientation reversed by $\tau$, $n_-$ is the number of off-axis negative crossings, and $m_0$ is the number of on-axis crossings with orientation preserved by $\tau$.

It remains to give the differential, which is formed as the direct sum of maps along each edge of the cube (in the direction of increasing $k$-grading).  Hence we can just give these maps along the edges.  Along each edge the maps are given by tensoring certain basic maps with the identity map on the tensor factors unaffected by the surgeries taking place along the edges of the cube.

We give these basic maps below.  Our notation is that we write $v_+$ or $v_-$ for the decorations of components of the smoothing that intersect the axis, while we write $v_{++}$ or $v_{--}$ for the decorations of pairs of components exchanged by $\tau$ (both decorated by $v_+$ or both by $v_-$).

%\begin{figure}
\[
\begin{tikzpicture}
  \matrix[matrix of math nodes, row sep=1em, column sep=1em,draw=gray]
    {
    \raisebox{-13pt}{\includegraphics[scale=0.75]{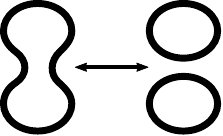}}
&
\begin{matrix*}[l]
v_+\mapsto v_+v_-+v_-v_+\\
v_-\mapsto v_-v_-
\end{matrix*}
&
\begin{matrix*}[l]
v_+\mapsfrom v_+v_+\\
v_-\mapsfrom v_+v_-,v_-v_+\\
0\phantom{_+}\mapsfrom v_-v_-
\end{matrix*}
\\
\raisebox{-6pt}{\includegraphics[scale=0.75]{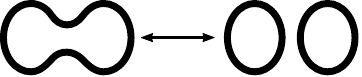}}
&
\begin{matrix*}[l]
v_+\mapsto 0\\
v_-\mapsto v_{--}
\end{matrix*}
&
\begin{matrix*}
v_+\mapsfrom v_{++}\\
0\phantom{_+}\mapsfrom v_{--}
\end{matrix*}
\\
\raisebox{-13pt}{\includegraphics[scale=0.75]{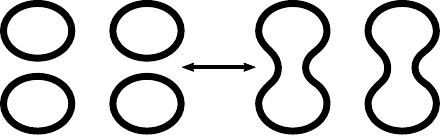}}
&
\begin{matrix*}[l]
v_{++}v_{++}\phantom{,v_{--}v_{++}}\mapsto v_{++}\\
v_{++}v_{--},v_{--}v_{++}\mapsto v_{--}\\
v_{--}v_{--}\phantom{,v_{--}v_{++}}\mapsto 0
\end{matrix*}
&
\begin{matrix*}[l]
v_{++}v_{--}+v_{--}v_{++}\mapsfrom v_{++}\\
v_{--}v_{--}\phantom{+v_{--}v_{++}\ }\mapsfrom v_{--}
\end{matrix*}
\\
\raisebox{-5pt}{\includegraphics[scale=0.75]{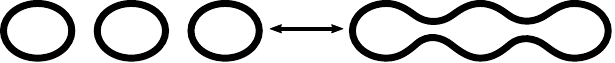}}
&
\begin{matrix*}[l]
v_{+}v_{++}\mapsto v_{+}\\
v_{-}v_{++}\mapsto v_{-}\\
v_{+}v_{--}\mapsto 0\\
v_{-}v_{--}\mapsto 0
\end{matrix*}
&
\begin{matrix*}[l]
v_{+}v_{--}\mapsfrom v_{+}\\
v_{-}v_{--}\mapsfrom v_{-}
\end{matrix*}
\\
\raisebox{-17pt}{\includegraphics[scale=0.75]{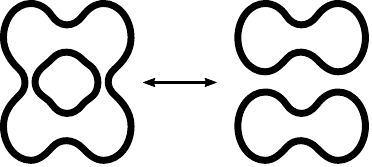}} 
& 
\begin{matrix*}[l]
v_{+}v_{+}\phantom{,v_{-}v_{+}}\mapsto v_{+}v_{-}+v_{-}v_{+}\\
v_{+}v_{-},v_{-}v_{+}\mapsto v_{-}v_- \\
v_-v_-\phantom{,v_{-}v_{+}}\mapsto 0
\end{matrix*}
& 
\text{(same)}\\
};
\end{tikzpicture}\]%\end{figure}

\section{Invariance}\label{sec:inv}
In this section we give the proof of Theorem \ref{thm:triply-graded-invariant}, namely invariance of our construction under the involutive Reidemeister moves of Figure \ref{fig:moves}.  The proof of this theorem is split into five parts---Propositions \ref{prop:offaxisRmrmoves} through %\ref{prop:onaxisRmrmoves}, \ref{prop:M1}, \ref{prop:M2}, and
\ref{prop:M3}---in each case following the strategy laid out in Subsection \ref{subsec:outline_invariance_proof}.

The moves of Figure \ref{fig:moves} are grouped in three types, of which the most complicated for us will be the third type consisting of the moves M1, M2, and M3.  These are the \emph{mixed moves} involving both off-axis and on-axis crossings.  Indeed, the invariance under the first type (IR1, IR2, and IR3) \emph{off axis moves}, and the second type (R1, R2) \emph{on axis moves} is formally very similar to showing the invariance of Khovanov cohomology under the usual Reidemeister moves.  Therefore we allow ourselves to be somewhat brief when dealing with the first two kinds of move (Subsections \ref{subsec:offaxismoves} and \ref{subsec:onaxismoves}), reserving special attention for the mixed moves (Subsection \ref{subsec:mixedmoves}).

\subsection{Off axis moves}
\label{subsec:offaxismoves}
We begin with the moves IR1, IR2, and IR3 of Figure \ref{fig:moves}.

%\begin{definition}
%	\label{def:offaxisRmrmoves}
%	We say that two involutive diagrams $D$ and $D'$ differ by the \emph{IR$i$ move} if they differ by two Reidemeister $i$ moves performed away from and either side of the axis of symmetry.  See Figure [] for an illustration of the IR$1$ case.
%\end{definition}

\begin{proposition}
	\label{prop:offaxisRmrmoves}
	Suppose for some $1 \leq i \leq 3$, that the involutive diagram $D$ differs from the involutive diagram $D'$ by ${\rm IR}i$.  Then the cohomologies $\Ht^j(D)$ and $\Ht^j(D')$ are bifiltered isomorphic.  Furthermore, there are isomorphisms of the induced spectral sequences starting with the $E^2_\F$ and with the $E^3_\G$ pages.
\end{proposition}

The heart of the proof of Proposition \ref{prop:offaxisRmrmoves} is the invariance proof for the corresponding Reidemeister move, coupled with the explicit description for $E^2_\G$ given in Subsection \ref{subsec:concreteE2G}.  So there is ultimately little real work that needs to be done.  We therefore give a general explanation of the proof, but only illustrate it in figures for the IR1 move.

\begin{proof}[Proof of Proposition \ref{prop:offaxisRmrmoves}]
Bar-Natan has given explicit cochain maps $f_i$, defined at the level of tangles, which give homotopy equivalences between the Khovanov complexes of diagrams differing by a usual $i^{\text{th}}$ Reidemeister move \cite{bncob}.  See the left of Figure \ref{fig:R1} for a schematic of the map $f_1$ for a Reidemeister $1$ move.  On the right of Figure \ref{fig:R1}, we show the same map after \emph{delooping}---namely, replacing the circle component with two summands corresponding to its possible decorations.

%\labellist
%\small
%\pinlabel {$\oplus$} at 307 23.5
%\tiny
%\pinlabel {$\cong$} at 279 11 \pinlabel {$+$} at 316 31 \pinlabel {$-$} at 316 1 \pinlabel {$\cong$} at 335 36
%\pinlabel {$\cong$} at 355 36
%\pinlabel {$\{+1\}$} at 328 40
%\pinlabel {$\{-1\}$} at 328 8
%	\endlabellist
%\begin{figure}[ht]
%\includegraphics[scale=0.75]{figures/R1-nbsp}
%\caption{The non-involutive type 1 Reidemeister move: A local picture of a diagram (left), the resulting chain complex (centre), and the result of {\it delooping} in the sense of Bar-Natan (right). Note that the homotopy equivalence provided by Bar-Natan is illustrated by the grey arrows.}\label{fig:R1}
%\end{figure}

\labellist
\small
\pinlabel {$f_1$} at 95 33
\pinlabel {$\oplus$} at 245 53.5
\tiny
\pinlabel {$\cong$} at 216 42.5 \pinlabel {$\cong$} at 268 32
% \pinlabel {$+$} at 316 31 \pinlabel {$-$} at 316 1 
%\pinlabel {$\cong$} at 355 36
%\pinlabel {$\{+1\}$} at 328 40
%\pinlabel {$\{-1\}$} at 328 8
	\endlabellist
\begin{figure}[ht]
\includegraphics[scale=0.75]{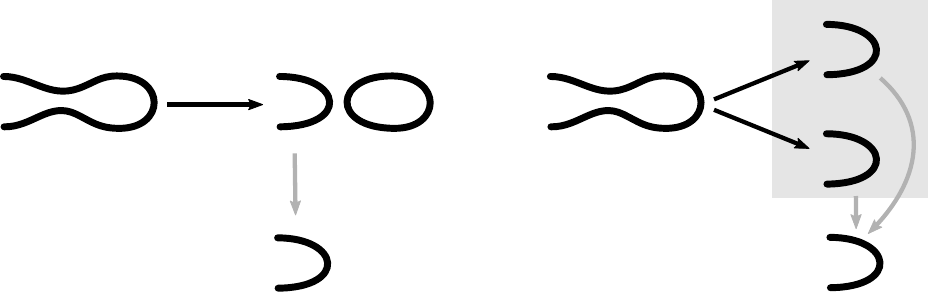}
\caption{The (non-involutive) Reidemeister 1 move.  This is the local picture of the cochain complex (left), and the result of {\it delooping} in the sense of Bar-Natan (right). Note that the homotopy equivalence provided by Bar-Natan is illustrated by the grey arrows.}\label{fig:R1}
\end{figure}

Suppose now that the involutive diagram $D$ differs from the involutive diagram $D'$ by IR$i$.  Working at the level of the local tangles, we take the tensor product of such a map $f_i$ together with the map $\overline{f_i}$ defined as the mirror image of $f_i$ to obtain a cochain map
\[ f_i \otimes \overline{f_i} \colon \CKh^j(D) \rightarrow \CKh^j(D') {\rm .}\]
By construction this map commutes with $d_\tau$ and so we obtain a cochain map
\[ f_i \otimes \overline{f_i} \colon \CKht^j(D) \rightarrow \CKht^j(D') {\rm .}\]
Furthermore, note that $f_i \otimes \overline{f_i}$ preserves both the $\F$ and $\G$ filtrations and induces the composition of two isomorphisms on $E_\F^2$ - the page isomorphic to the usual Khovanov cohomology.  Hence $f_i \otimes \overline{f_i}$ induces an $\F$-filtered isomorphism on cohomology $\Ht^j(D) \rightarrow \Ht^j(D')$.  To show that $f_i \otimes \overline{f_i}$ in fact induces a \emph{bilfiltered} isomorphism on cohomology, we shall see that the induced map
\[ (f_i \otimes \overline{f_i})^3_\G \colon E_\G^3(D) \rightarrow E_\G^3(D') \]
is an isomorphism.  This amounts to showing that the cone of the cochain map
\[(f_i \otimes \overline{f_i})^2_\G \colon E_\G^2(D) \rightarrow E_\G^2(D') \]
is acyclic.

The map $(f_i \otimes \overline{f_i})^2_\G$ is the associated graded map to the map induced by Gauss elimination along $d_\tau$.  This map is illustrated in Figure \ref{fig:IR1} for $i=1$.  The reader will notice the formal similarity to the map in Figure \ref{fig:R1}.  To push this formal similarity further, let us emphasize an important point, which may be gleaned from the concrete description of the $E_\G^2$ page given in Subsection~\ref{subsec:concreteE2G}.

\labellist
\small
\pinlabel {$f_1\otimes\overline{f_1}$} at 178 34
\pinlabel {$\oplus$} at 395 59
\tiny
\pinlabel {$\cong$} at 355 47 \pinlabel {$\cong$} at 424 25
%\pinlabel {$+$} at 386 31.5 \pinlabel {$+$} at 402 31.5
%\pinlabel {$-$} at 386 -1 \pinlabel {$-$} at 402 -1
	\endlabellist
\begin{figure}[ht]
\includegraphics[scale=0.75]{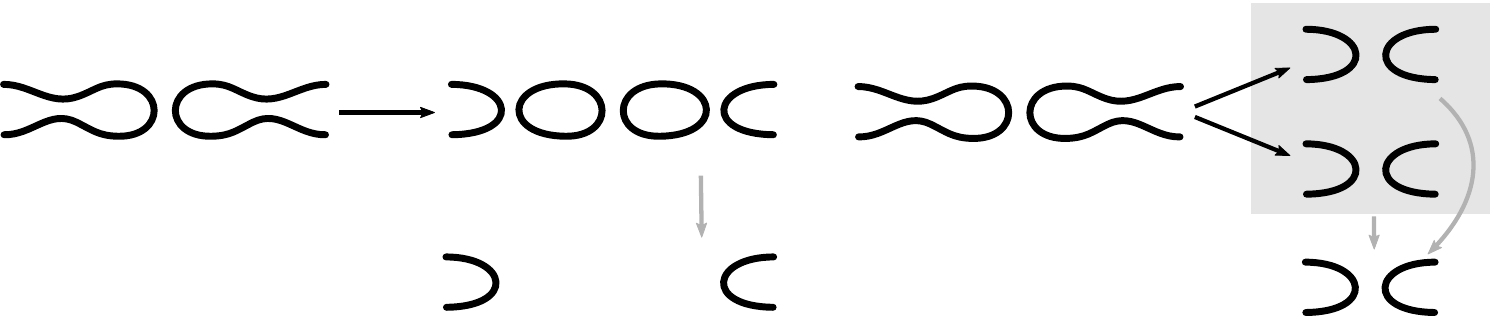}
\caption{Delooping the complex associated with an IR1 move.%, following the conventions from Figure \ref{fig:R1}
	}\label{fig:IR1}
\end{figure}

The reader will notice that we may \emph{deloop} the pair of equivariant circles.  This is just rewriting the relevant cochain group summand as a sum of two vector spaces, one in which the equivariant pair receives the decoration $v_{++}$ and one in which it receives $v_{--}$.  Following delooping, certain components of the differential and of cochain map are seen to be isomorphisms.  Specifically, the double saddle map producing an equivariant circle pair gives such a component to the $v_{--}$ summand; the double saddle map killing an equivariant circle pair gives such a component from the $v_{++}$ summand; a double cup producing an equivariant circle pair gives an isomorphism to the $v_{++}$ summand; and finally a double cap killing an equivariant circle pair gives an isomorphism from the $v_{--}$ summand.

This should be compared with the usual notion of delooping in the Khovanov complex in which one removes a single circle; see the right of Figures \ref{fig:R1} and \ref{fig:IR1}.  Note in particular that there is a correspondence between the two figures when considering the components of the differential and the components of the cochain maps that become isomorphisms after delooping.

By performing successive Gauss elimination along components that are isomorphisms, one then sees that the cone of the cochain map $(f_i \otimes \overline{f_i})^2_\G$ is acyclic, and hence $(f_i \otimes \overline{f_i})^3_\G$ is an isomorphism, as required.
\end{proof}

\subsection{On axis moves}
\label{subsec:onaxismoves}
Now consider the moves R1 and R2 from Figure \ref{fig:moves}.

\begin{proposition}
	\label{prop:onaxisRmrmoves}
	Suppose  that the involutive diagram $D$ differs from the involutive diagram $D'$ by an on-axis Reidemeister move.  Then the cohomologies $\Ht^j(D)$ and $\Ht^j(D')$ are bifiltered isomorphic.  Furthermore, there are isomorphisms of the induced spectral sequences starting with the $E^2_\F$ and with the $E^3_\G$ pages.
\end{proposition}

\begin{proof}
	There is a Bar-Natan cochain map
	\[ f \colon \CKh^j(D) \rightarrow \CKh^j(D') \]
	which induces an isomorphism on Khovanov cohomology.  Note that this map commutes with $\tau$ and hence gives a cochain map
	\[ f \colon \CKht^j(D) \rightarrow \CKht^j(D') {\rm .}\]
	Observe also that this map preserves both the $\F$ and $\G$ filtrations.  Now, by construction, $f$ induces a Khovanov isomorphism on $E^2_\F$, hence $f$ induces an $\F$-filtered isomorphism on cohomology $\Ht^j(D) \rightarrow \Ht^j(D')$.
	
	Similarly to the off axis case, we then wish to show that the cone of the map
	\[ f^2_\G \colon E_\G^2(D) \rightarrow E_\G^2(D') \]
	is acyclic.  Again this is achieved by Gauss elimination, this time following the delooping of the $\tau$-equivariant circle appearing in one of the cochain complexes.
\end{proof}

\subsection{Mixed moves.}
\label{subsec:mixedmoves}
The mixed moves M1, M2, and M3 present complications that are not present in the other moves.  For each move we shall follow carefully the strategy outlined in Subsection \ref{subsec:outline_invariance_proof} in which we construct two cochain maps $f$ and $g$ related by a chain homotopy $h$.  Before we begin, we should discuss how we represent these maps.  Each of the maps $f$, $g$, and $h$ is presented as a matrix of maps between vectors of tangles.  There is presumably a local involutive theory along the lines of Bar-Natan's work \cite{bncob} which makes rigorous sense of this, although we do not develop it in this paper.

For our purposes, each entry of the matrix will be the vector space map between summands of the complexes of the closed diagrams induced by a cobordism between the tangles (extended by the identity cobordism outside the tangle).  The one wrinkle is that we sometimes wish to precompose these maps with $\tau$ (see Figure \ref{fig:mixed-R3-with-f} for an example).

Most of the time, the tangle cobordisms in the non-zero matrix entries are uniquely determined by the requirement that they preserve the quantum degree---this explains, for example, why only one arrow in Figure \ref{fig:mixed-R3-with-g} is decorated with a cobordism.

\begin{proposition}
	\label{prop:M1}
	Suppose  that the involutive diagram $D'$ differs from the involutive diagram $D$ by the move M1.  Then the cohomologies $\Ht^j(D)$ and $\Ht^j(D')$ are bifiltered isomorphic.  Furthermore, there are isomorphisms of the induced spectral sequences starting with the $E^2_\F$ and with the $E^3_\G$ pages.
\end{proposition}

\begin{proof}
Recall that M1 is the move relating two involutive diagrams $D$, $D'$ which differ locally as shown towards the left of Figure \ref{fig:mixed-R3-with-f}.

Our strategy is to start by constructing a cochain map $f=f_0+f_1 \colon \CKht^j(D) \rightarrow \CKht^j(D')$.  The map $f_0$ is Bar-Natan's cochain map (so commuting with $d$), which exhibits invariance under Reidemeister 3 in Khovanov cohomology, and appears as the solid arrows preserving the $i$-degree in Figure \ref{fig:mixed-R3-with-f} (compare \cite[Figure 8]{bncob}).

\begin{figure}[ht]
	%\label{fig:mixed-R3-with-f}
	\labellist
	\small
	\pinlabel {$\tau$} at 336 162
	\pinlabel {$\tau$} at 379 162
	\pinlabel {$\tau$} at 405 162
	\endlabellist
	\includegraphics[scale=0.55]{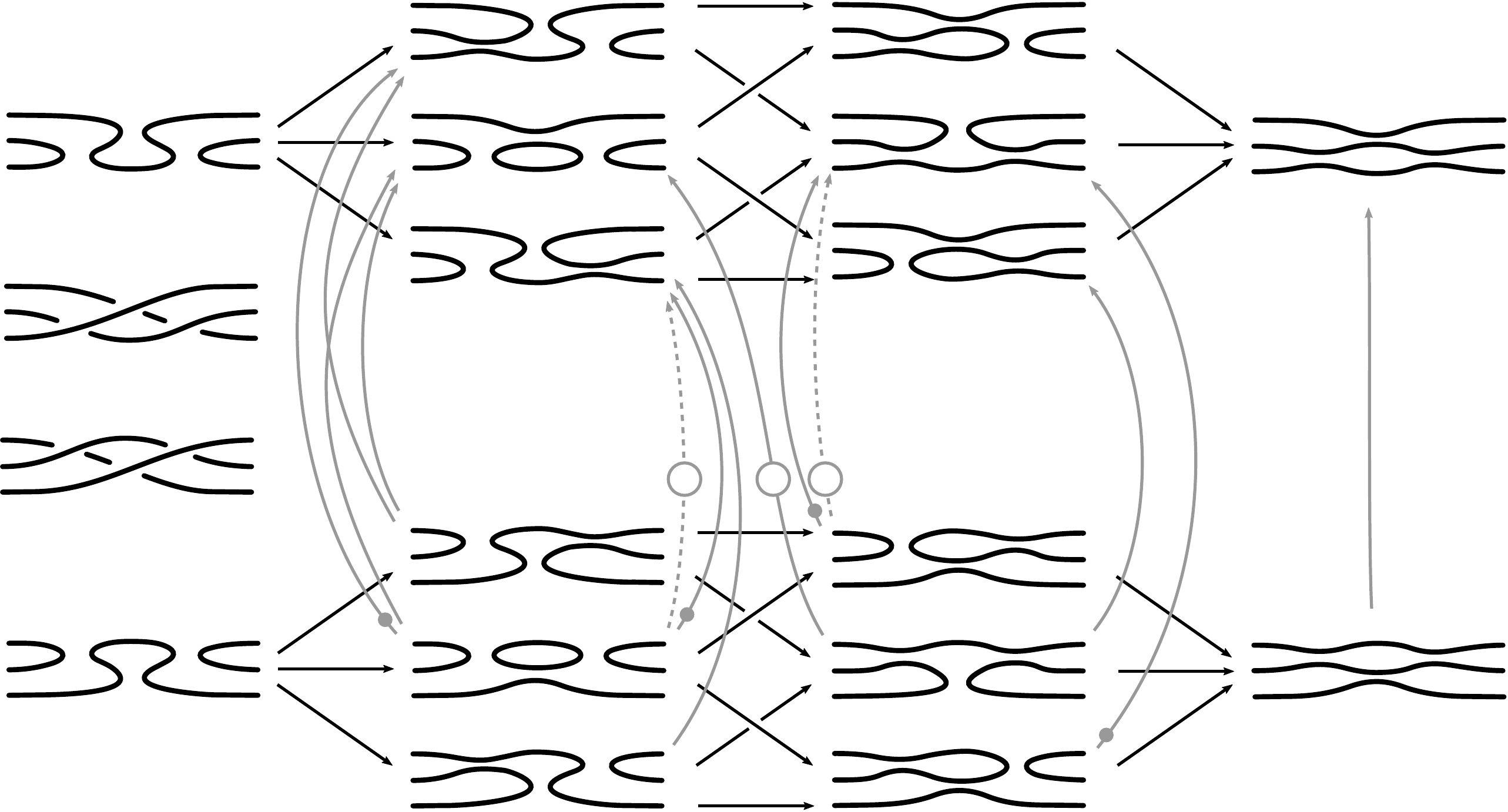}
	\caption{On the bottom is $\CKht^j(D)$ and on the top $\CKht^j(D')$.  We indicate the map $f=f_0+f_1$ by solid arrows and the homotopy $h$ by dashed arrows.  The presence of a circled $\tau$ on an arrow means that the cobordism has been precomposed with $\tau$. The map $f_0$ consists of those solid arrows preserving the $i$-degree, $f_1$ is the single solid arrow lowering the $i$ degree.  The failure of $f_0$ to commute with the total differential $\partial$ is illustrated in Figure \ref{fig:mixed-R3-failure}. Finally, the arrows labeled with a dot indicate those components of $f$ that lower the $\mathcal{G}$-filtration.  Hence $f$ is not a $\mathcal{G}$-filtered chain map.}\label{fig:mixed-R3-with-f}
\end{figure}

\begin{figure}[ht]
	%\label{fig:test_case_move}
	\labellist
	\small
	\pinlabel {$\tau$} at 152 134 
	\pinlabel {$\tau$} at 300 134
	\pinlabel {$\tau$} at 462 134
	\pinlabel {$\tau$} at 480 134
	\endlabellist
	\includegraphics[scale=0.68]{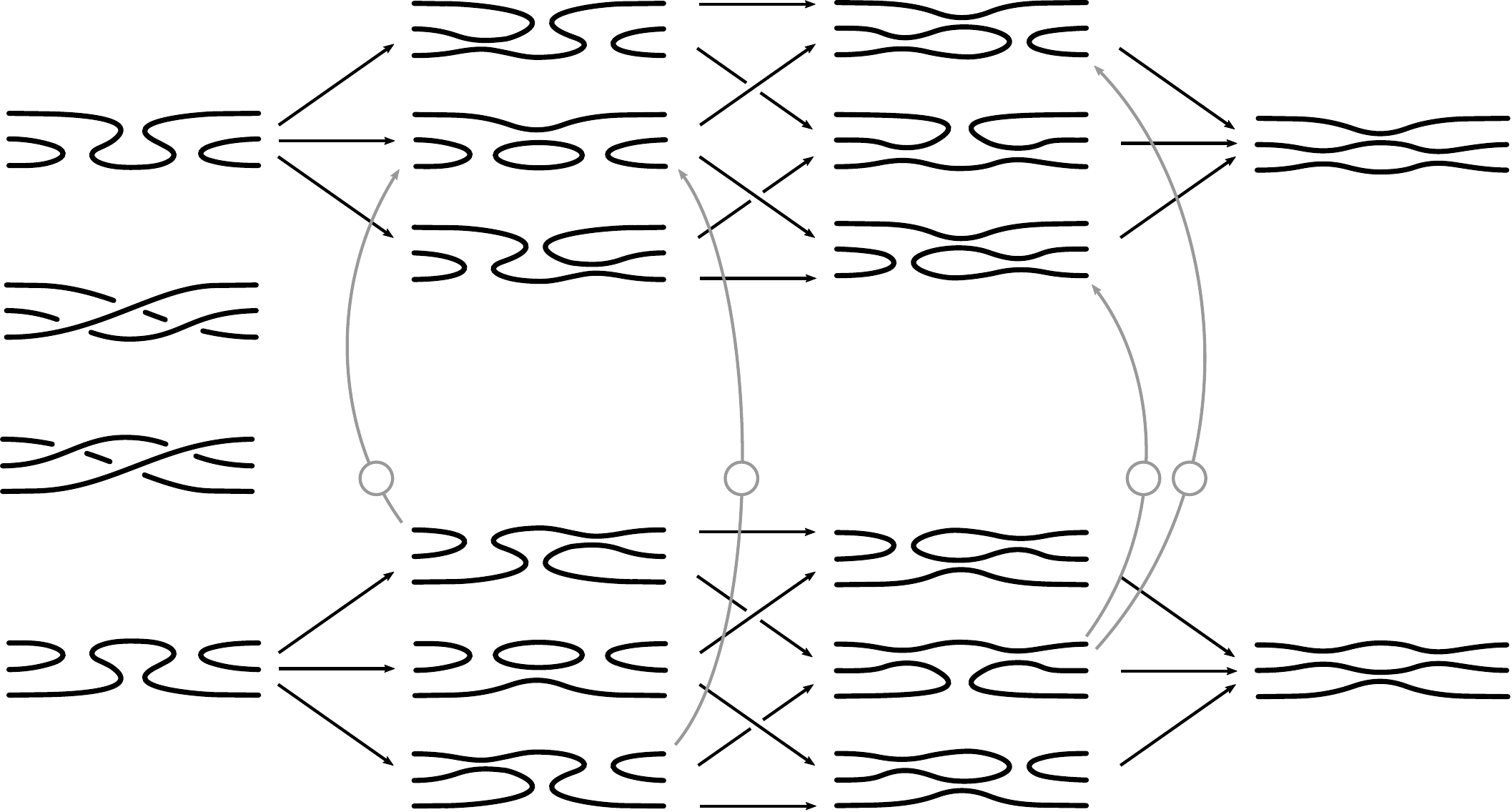}
	\caption{The failure of $f_0$ to commute with the total differential $\partial$ is measured by $f_0\tau+\tau f_0$, indicated here by the solid arrows. %This is corrected by the single component of $f_1$ so that $f=f_0+f_1$ is an  $\mathcal{F}$-filtered chain map.
		}\label{fig:mixed-R3-failure}
\end{figure}
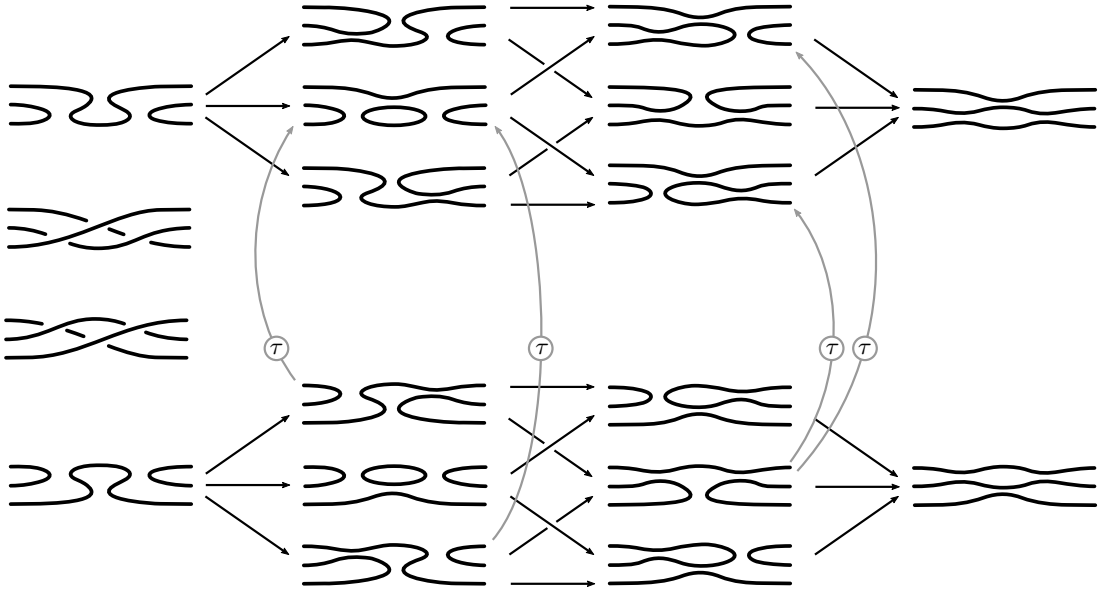

\begin{figure}[ht]
	%\label{fig:test_case_move}
	\labellist
	\small
	\pinlabel {$\tau$} at 150 162
	\pinlabel {$\tau$} at 322 162
	\pinlabel {$\tau$} at 343.5 162
	\pinlabel {$\tau$} at 363 162
	\pinlabel {$\tau$} at 382 162
\tiny
\pinlabel {$\frac{1}{2}$} at 295 100  \pinlabel {$\frac{3}{2}$} at 500 100
 	\pinlabel {$0$} at 100 47 \pinlabel {$1$} at 295 47  \pinlabel {$1$} at 500 47  \pinlabel {$2$} at 710 47
	\pinlabel {$\frac{1}{2}$} at 295 -9  \pinlabel {$\frac{3}{2}$} at 500 -9
	\pinlabel {$\frac{1}{2}$} at 295 360  \pinlabel {$\frac{3}{2}$} at 500 360
 	\pinlabel {$0$} at 100 305 \pinlabel {$1$} at 295 305 \pinlabel {$1$} at 500 305  \pinlabel {$2$} at 710 305
	\pinlabel {$\frac{1}{2}$} at 295 250  \pinlabel {$\frac{3}{2}$} at 500 250
	\endlabellist
	\includegraphics[scale=0.55]{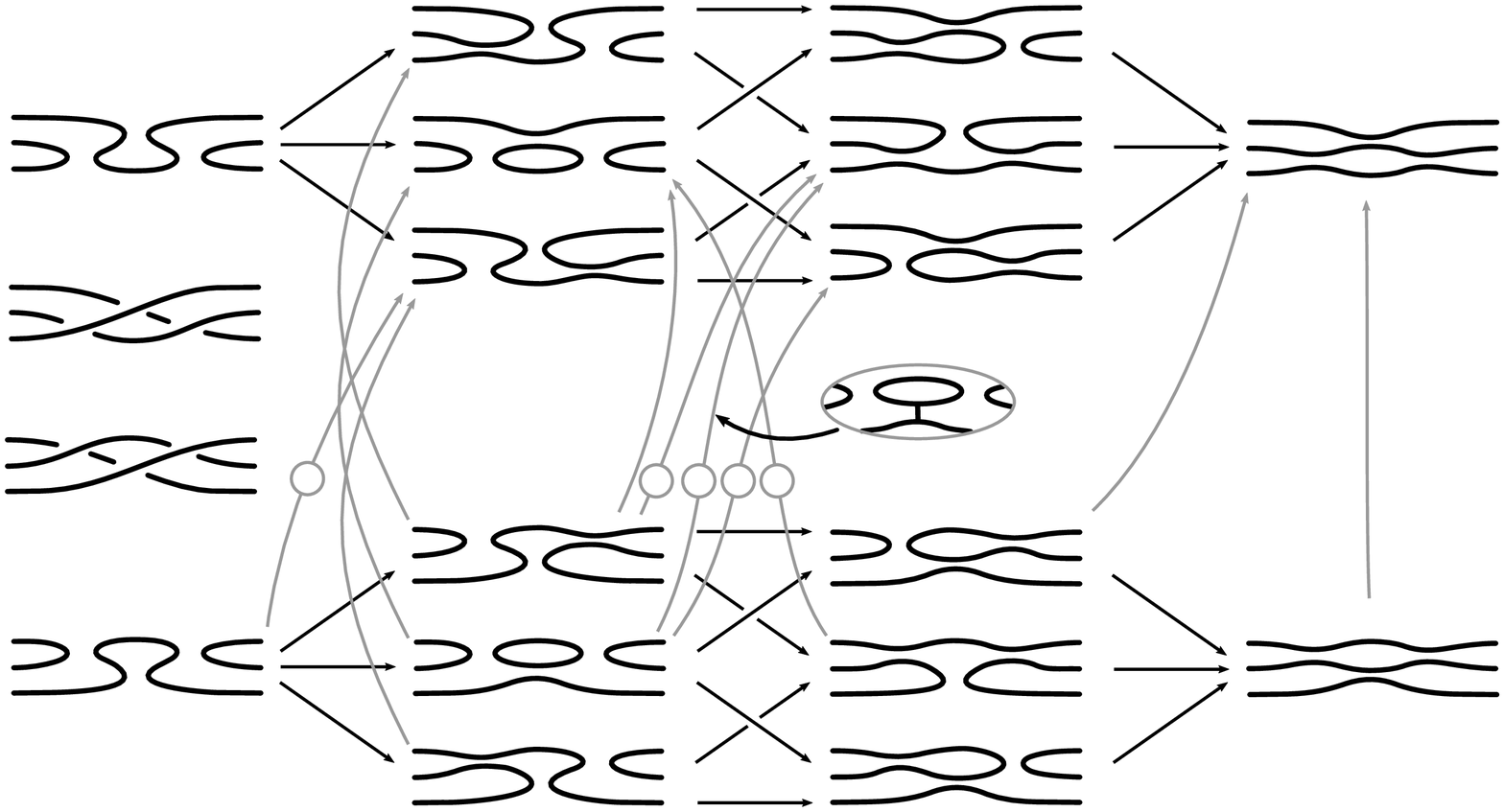}
	\caption{The map $g=f + \partial h + h \partial$ is $\mathcal{G}$-filtered (relative filtration levels are indicated).}\label{fig:mixed-R3-with-g}
\end{figure}

Although $f_0$ commutes with the Khovanov differential $d$, it does not commute with the total differential $\partial = d + d_\tau$.  This failure of commutativity is measured by
\[ f_0 \partial + \partial f_0 = f_0 d_\tau + d_\tau f_0 = f_0 \tau + \tau f_0 \colon \CKht^j(D) \rightarrow \CKht^j(D') {\rm ,}\]
and this is illustrated in Figure \ref{fig:mixed-R3-failure}.

The map $f_1 \colon \CKh^j(D) \rightarrow \CKh^j(D')$ is indicated in Figure \ref{fig:mixed-R3-with-f} by the single solid arrow that lowers the $i$-degree by $1$.  It can then be observed that
\[ f_1 \partial + \partial f_1 = f_1 d + d f_1 = f_0 \tau + \tau f_0 \colon \CKh^j(D) \rightarrow \CKh^j(D') {\rm ,}\]
or, in other words, that $f = f_0 + f_1$ commutes with $\partial$.

Hence we have constructed a cochain map $f$ which respects the $\F$ filtration.  We note that the $\F$-filtration preserving part of $f$ is exactly $f_0$, which is Bar-Natan's cochain map giving an isomorphism on the Khovanov cohomology.  Therefore $f$ induces an isomorphism $f^2_\F \co E^2_\F(D) \rightarrow E^2_\F(D')$ (namely, an isomorphism on the usual Khovanov cohomology).

%\begin{figure}[h]
 %\includegraphics[scale=0.6]{figures/mixed-R3-iso-BEFORE-DELOOPING}
%\caption{Before delooping on the $E^2_\mathcal{G}$ page.}\label{fig:mixed-R3-iso-BEFORE-DELOOPING}
%\end{figure}

After constructing $g=f+ \partial h + h \partial$ using the homotopy $h$ shown in Figure \ref{fig:mixed-R3-with-f} by dashed arrows, we obtain a $\mathcal{G}$-filtered chain map as shown in Figure  \ref{fig:mixed-R3-with-g}. It remains to verify that $g$ induces an isomorphism on the $E^3_{\mathcal{G}}$ page of the spectral sequence.  This is equivalent to showing that the cone of the cochain map $g^2_\G \co E^2_\G(D) \rightarrow E^2_\G(D')$ is acyclic.  We give a schematic of $g^2_\G$ in Figure \ref{fig:mixed-R3-iso}.

%\labellist
%	\tiny
%	\pinlabel {$+$} at 40 39 \pinlabel {$+$} at 40 92
%	\pinlabel {$\cong$} at 368 54
%	\pinlabel {$\cong$} at 127 54
%	\pinlabel {$\cong$} at 181 54
%	\endlabellist
 %\centering
% \includegraphics[scale=0.45]{figures/mixed-R3-iso}
% \captionof{figure}{Here we show the cochain map $g^2_\G \co E^2_\G(D_1) \rightarrow E^2_\G(D_2)$ for the move M1.  We use the description of the $E_\G^2$ page given in Subsection \ref{subsec:concreteE2G} and project to the coordinates (of the cube of all $\tau$-equivariant smoothings) corresponding to the tangle crossings.}
% \label{fig:mixed-R3-iso}

 \labellist
 \small
 \pinlabel {$\cong$} at 225 165 \pinlabel {$\cong$} at 314 165 \pinlabel {$\cong$} at 392 165 \pinlabel {$\cong$} at 135 261
 \pinlabel {$\cong$} at 135 70
 \tiny
 \pinlabel {$+$} at 232 285 \pinlabel {$-$} at 232 242
 \pinlabel {$+$} at 232 120 \pinlabel {$-$} at 232 78
 \endlabellist
 \begin{figure}[h]
 	\includegraphics[scale=0.70]{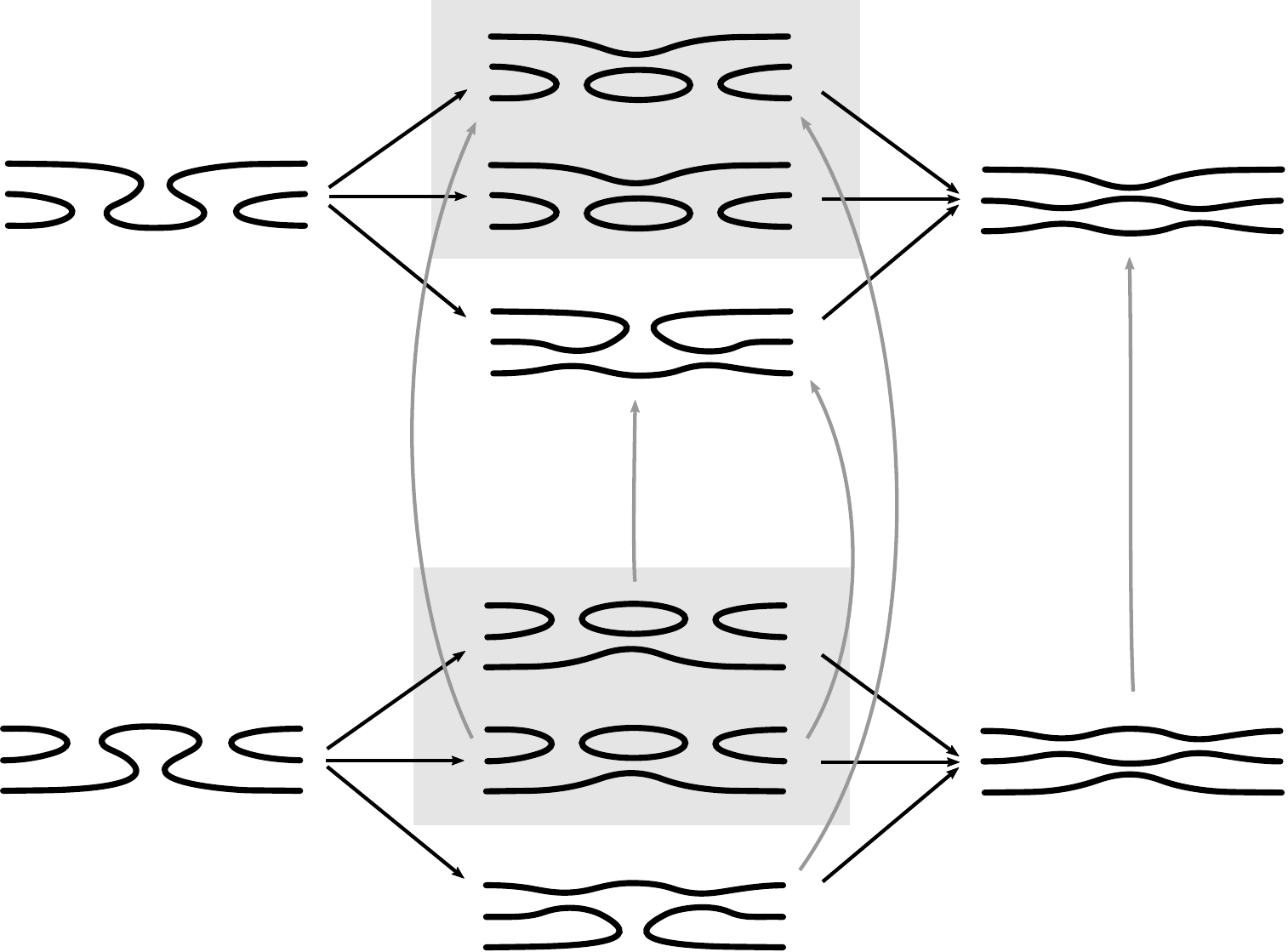}
 	\caption{Here we show the cochain map $g^2_\G \co E^2_\G(D) \rightarrow E^2_\G(D')$ for the move M1.  We use the description of the $E_\G^2$ page given in Subsection \ref{subsec:concreteE2G} and project to the coordinates (of the cube of all $\tau$-equivariant smoothings) corresponding to the tangle crossings.}\label{fig:mixed-R3-iso}
 \end{figure}

In this figure we are representing the complexes $E_\G^2(D)$ and $E_\G^2(D')$ using the description given in Subsection \ref{subsec:concreteE2G} (in which basis elements of the complex correspond to $\tau$-equivariant decorations of $\tau$-equivariant smoothings).  We have delooped the two $\tau$-equivariant circles (replacing each with the summands corresponding to the two possible decorations), and indicated which components of the differential and of the cochain map are necessarily isomorphisms.  The will readily verify that successive Gauss elimination along those components that are isomorphisms results in the zero complex.
\end{proof}

\begin{proposition}
	\label{prop:M2}
	Suppose  that the involutive diagram $D$ differs from the involutive diagram $D'$ by the move M2.  Then the cohomologies $\Ht^j(D)$ and $\Ht^j(D')$ are bifiltered isomorphic.  Furthermore, there are isomorphisms of the induced spectral sequences starting with the $E^2_\F$ and with the $E^3_\G$ pages.
\end{proposition}

\begin{proof}
For the move M2, we follow the same procedure as for M1.  Let us turn first to the cochain map
\[ f = f_0 + f_1 \co \CKht^j(D) \rightarrow \CKht^j(D')\]
that is shown in Figure \ref{fig:mixed-R2-with-f}.  The map $f_0$ commutes with $d$ and is a Khovanov chain homotopy equivalence for the usual Reidemeister 2 move.  The map $f = f_0 + f_1$ commutes with the total differential $\partial$ and preserves the $\F$-filtration.  Since $f_1$ strictly lowers the $i$-degree, $f$ induces a Khovanov isomorphism $f^2_\F \co E^2_\F(D) \rightarrow E^2_\F(D')$.

\begin{figure}[ht]
		\labellist
	\small
	\pinlabel {$\tau$} at 358 54
	\pinlabel {$\tau$} at 386 54
	\pinlabel {$\tau$} at 475 54
	\endlabellist
	\includegraphics[scale=0.55]{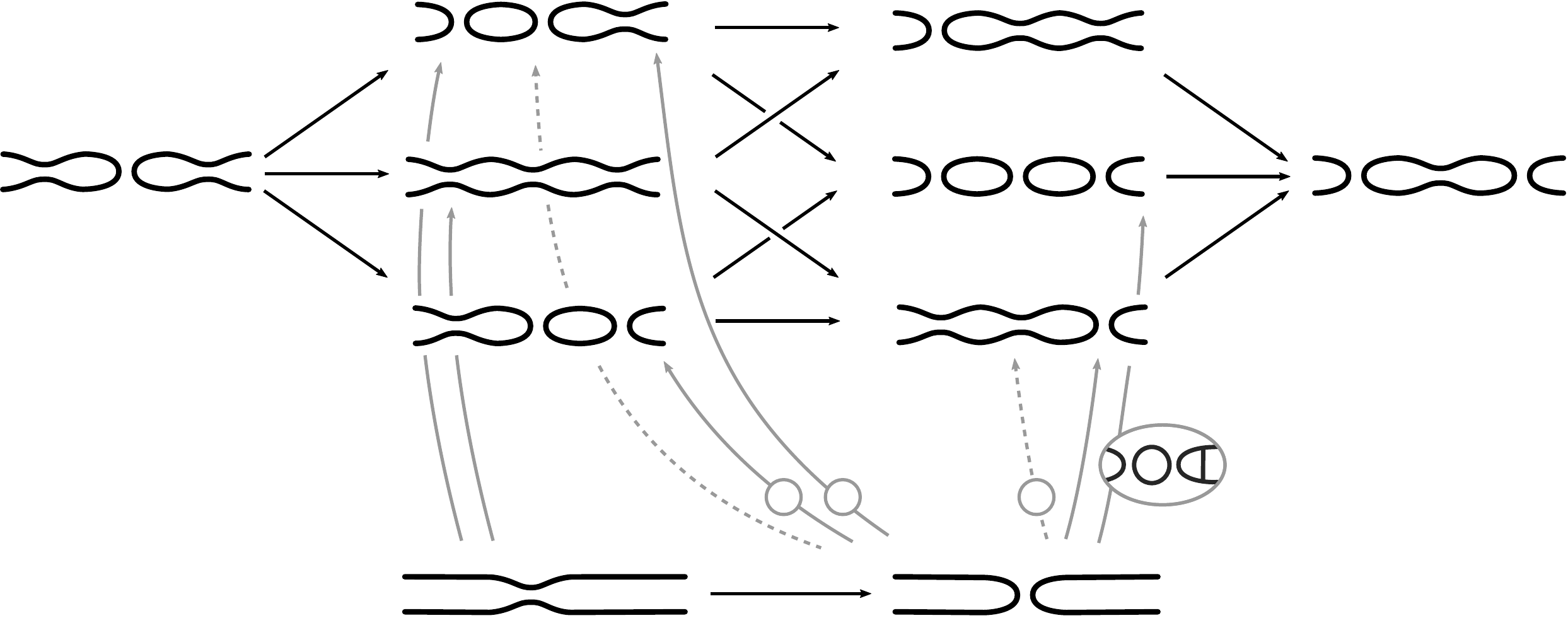}
	\caption{The map $f$ (solid lines) and a homotopy $h$ (dashed lines).  The components of $f$ that are vertical (so preserving $i$-degree) comprise the Khovanov chain homotopy equivalence $f_0$ for the Reidemeister 2 move.}\label{fig:mixed-R2-with-f}

\end{figure}

\begin{figure}[ht]
		\labellist
	\small
	\pinlabel {$\tau$} at 359 54
	\pinlabel {$\tau$} at 514 54
	\endlabellist
	\includegraphics[scale=0.55]{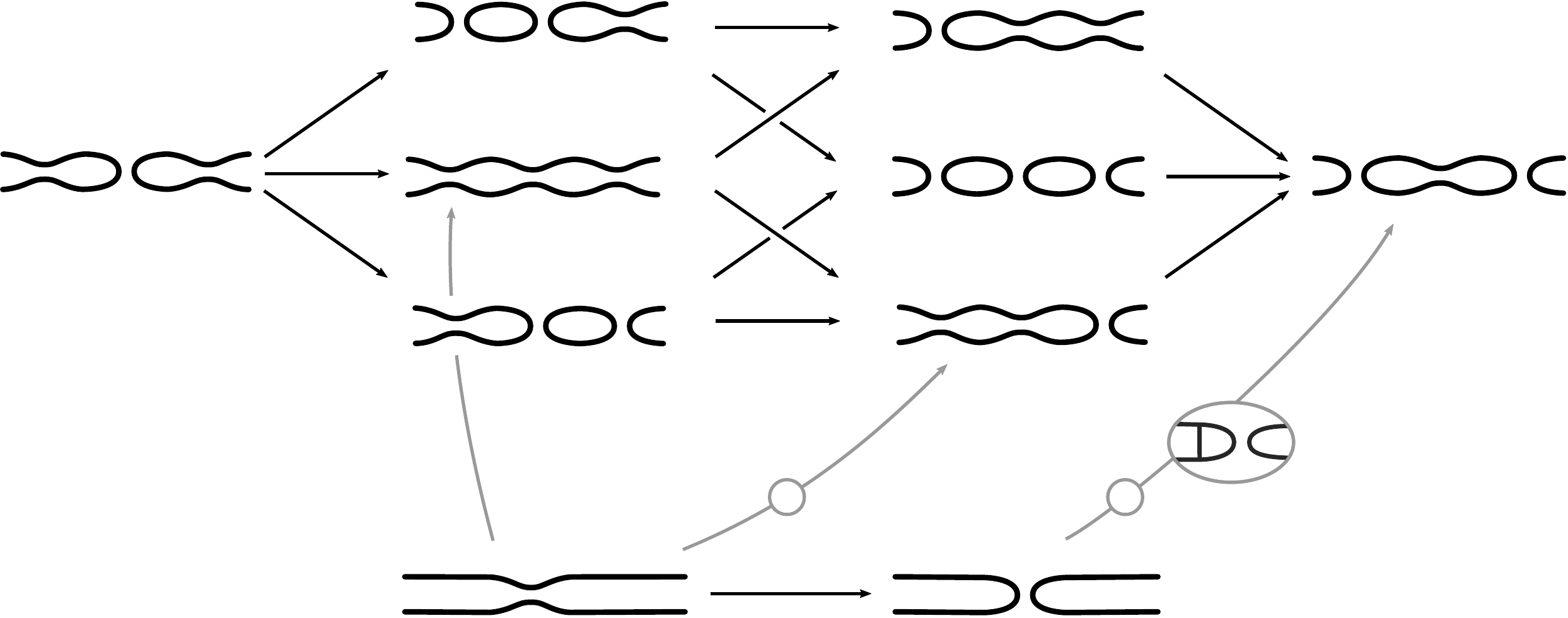}
	\caption{The map $g=f + \partial h + h \partial$.}\label{fig:mixed-R2-with-g}
\end{figure}

%\begin{figure}[h]
% \includegraphics[scale=0.6]{figures/mixed-R2-iso}
%\caption{Here we show the cochain map $g^2_\G \co E^2_\G(D) \rightarrow E^2_\G(D')$ for the move M2.}\label{fig:mixed-R2-iso}
%\end{figure}

 \labellist
 \small
 \pinlabel {$\cong$} at 135 170 \pinlabel {$\cong$} at 302 211 
 \pinlabel {$\cong$} at 202 70 \pinlabel {$\cong$} at 376 70
 \tiny
 \pinlabel {$+$} at 240 240 \pinlabel {$-$} at 240 181
  \pinlabel {$+$} at 190 240 \pinlabel {$-$} at 190 181
   \pinlabel {$+$} at 410 196 \pinlabel {$-$} at 410 136
 \endlabellist
 \begin{figure}[h]
 	\includegraphics[scale=0.70]{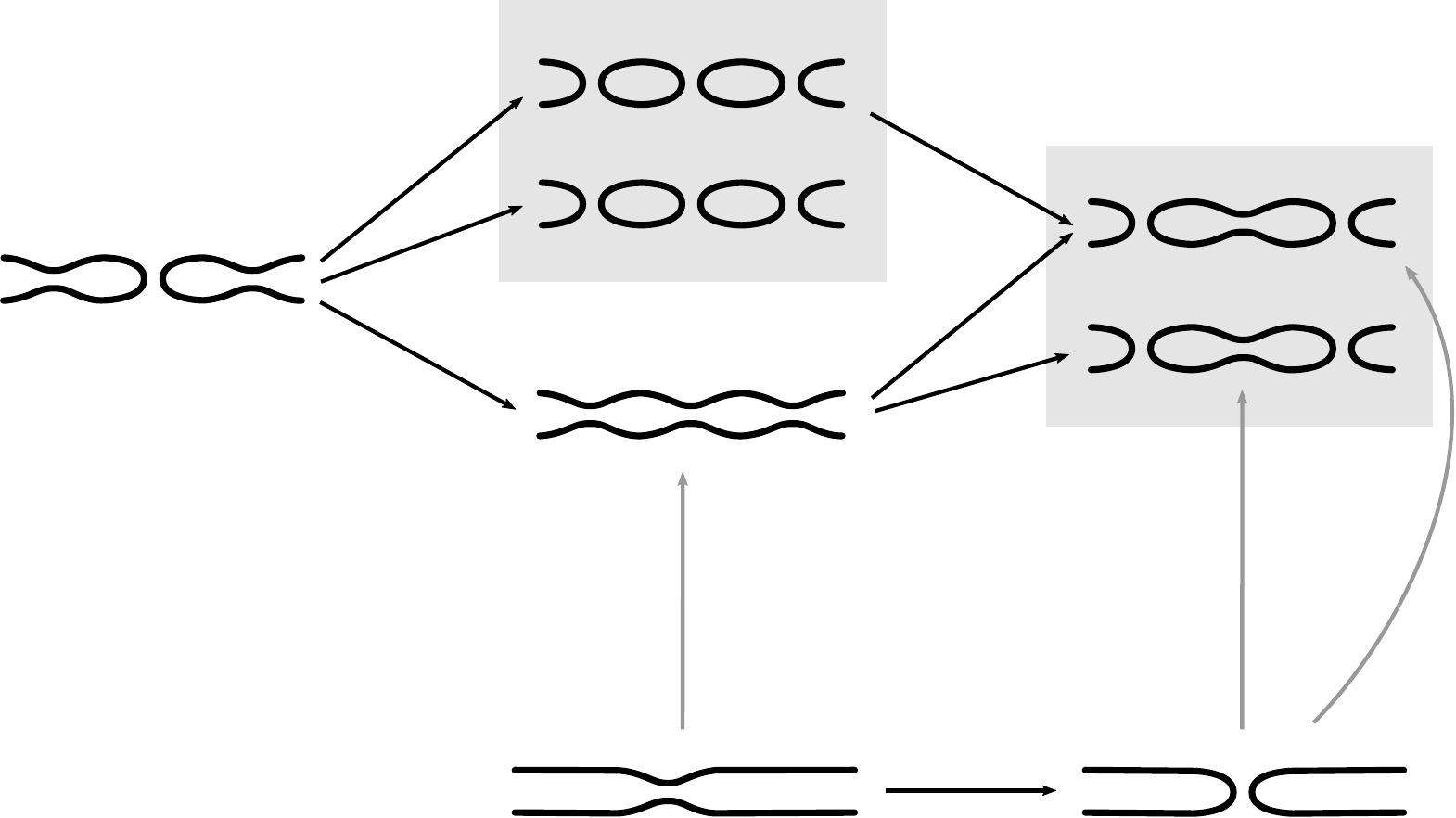}
 	\caption{Here we show the cochain map $g^2_\G \co E^2_\G(D) \rightarrow E^2_\G(D')$ for the move M2.  %We use the description of the $E_\G^2$ page given in Subsection \ref{subsec:concreteE2G} and project to the coordinates (of the cube of all $\tau$-equivariant smoothings) corresponding to the tangle crossings.
	}\label{fig:mixed-R2-iso}
 \end{figure}

Also shown in Figure \ref{fig:mixed-R2-with-f} is a homotopy $h$.  The cochain map $g=f + \partial h + h \partial$ is shown in Figure \ref{fig:mixed-R2-with-g}.  Note that $g$ preserves the $\G$ filtration---it remains to check that $g$ induces an isomorphism on the $E^3_\mathcal{G}$ page.  This is equivalent to checking that the cone of the cochain map $g^2_\G \co E^2_\G(D) \rightarrow E^2_\G(D')$ is acyclic.  The map $g^2_\G$ is shown in Figure \ref{fig:mixed-R2-iso} with those components that are necessarily isomorphisms indicated.  Again, the reader will observe that successive Gauss elimination results in the zero complex.
\end{proof}

%\subsection{The monster move}
%\label{sec:monster_move}

\begin{proposition}
	\label{prop:M3}
	Suppose  that the involutive diagram $D$ differs from the involutive diagram $D'$ by the move M3.  Then the cohomologies $\Ht^j(D)$ and $\Ht^j(D')$ are bifiltered isomorphic.  Furthermore, there are isomorphisms of the induced spectral sequences starting with the $E^2_\F$ and with the $E^3_\G$ pages.
\end{proposition}

\begin{proof}
The move M3 relates two tangles each of 6 crossings.  To make our proof a little more tractable, we consider an equivalent formulation of the move M3 which relates the $0$-crossing $4$-braid tangle to a $12$-crossing $4$-braid tangle.  These tangles appear as $D_1$ and $D_2$ in Figure \ref{fig:monster}.  %The equivalence of the formulations is left as an exercise to the students of the reader.

We now proceed as we have done before for the moves M1 and M2.  The first part of this procedure is to write down a cochain map
\[ f_0 \co \CKh(D_1) \rightarrow \CKh(D_2) \]
that induces an isomorphism on Khovanov cohomology.

Let us look more closely at Figure \ref{fig:monster}.  The tangle $D_2$ may be obtained from the identity braid tangle $D_1$ by a sequence of six Reidemeister 2 moves followed by four Reidemeister 3 moves.  The intermediate tangle $D_{3/2}$ obtained after the six Reidemeister 2 moves is given in Figure \ref{fig:monster}.  The crossings have been labelled for future reference.  Four Reidemeister 3 moves that may be used to get from $D_{3/2}$ to $D_2$ are two that slide the crossing labelled $1$ over the $(2,4)$-arc and then over the $(3,5)$-arc, and then two that slide the crossing labelled $7$ under the $(9,11)$-arc and then under the $(8,10)$-arc.

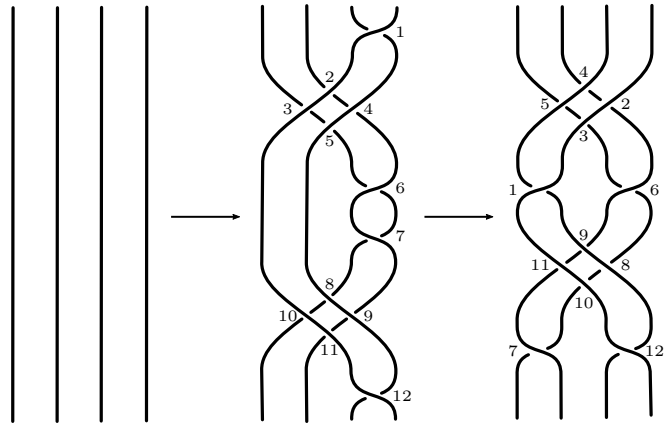
\begin{figure}[ht]
	\labellist
	\tiny
	%middle
	\pinlabel {$1$} at 242 245
	\pinlabel {$2$} at 198 217
	\pinlabel {$3$} at 172 197
	\pinlabel {$4$} at 222 197
	\pinlabel {$5$} at 198 177
	\pinlabel {$6$} at 242 147
	\pinlabel {$7$} at 242 117
	\pinlabel {$8$} at 198 87
	\pinlabel {$9$} at 222 67
	\pinlabel {$10$} at 172 67
	\pinlabel {$11$} at 198 46
	\pinlabel {$12$} at 243 17
	%next
	\pinlabel {$1$} at 312 146
	\pinlabel {$2$} at 382 200
	\pinlabel {$3$} at 356 180
	\pinlabel {$4$} at 356 220
	\pinlabel {$5$} at 332 200
	\pinlabel {$6$} at 400 146
	\pinlabel {$7$} at 312 45
	\pinlabel {$8$} at 382 99
	\pinlabel {$9$} at 356 118
	\pinlabel {$10$} at 356 75
	\pinlabel {$11$} at 329 99
	\pinlabel {$12$} at 400 45
	\endlabellist
	\includegraphics[scale=0.6]{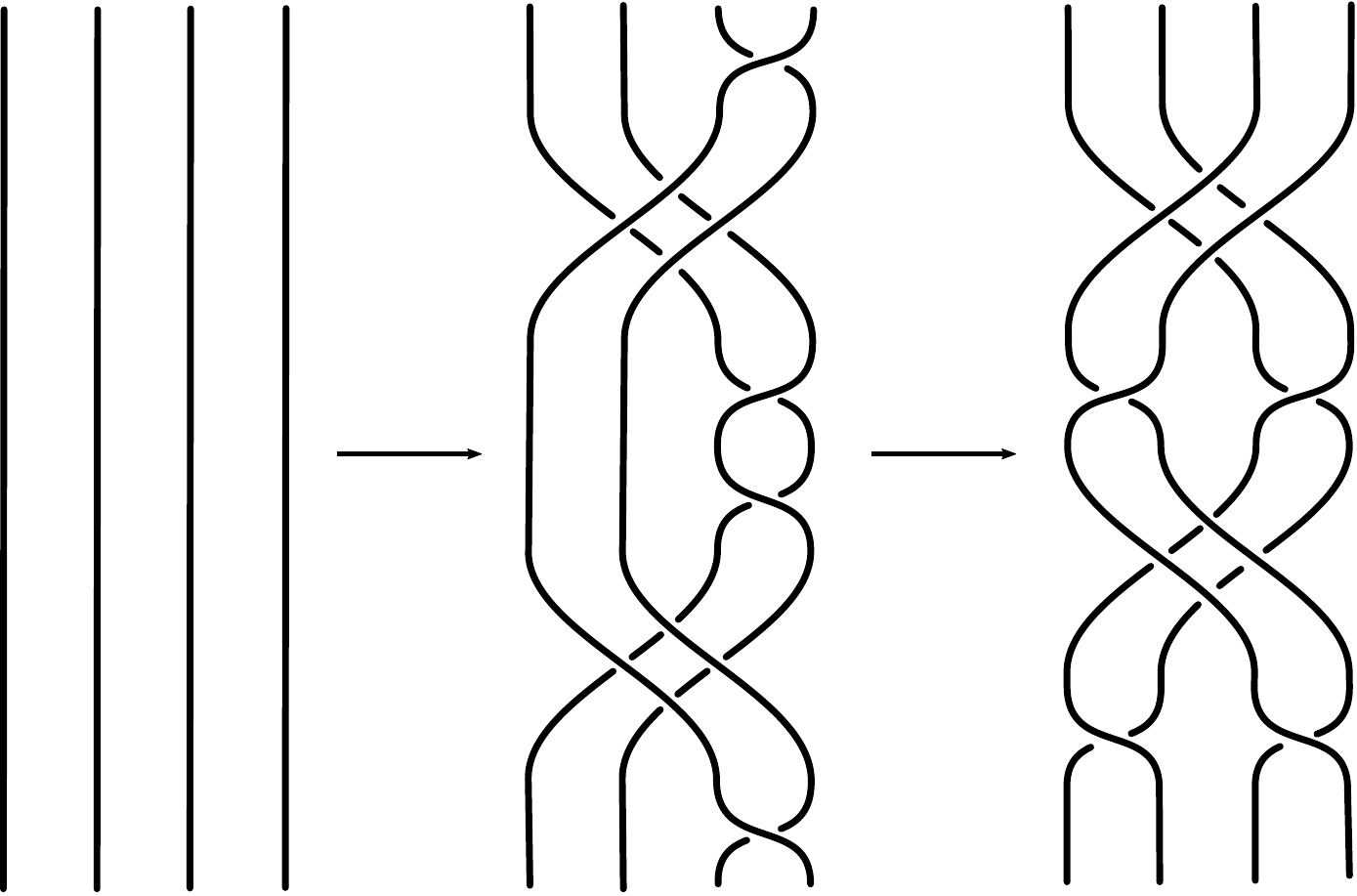}
	\caption{The tangle diagrams $D_1$, $D_{{3}/{2}}$, and $D_{2}$ (left to right).}
	\label{fig:monster}
\end{figure}

\begin{figure}[h]
	\labellist
	\tiny
	%middle
	\pinlabel {$1$} at 2 145 \pinlabel {$1$} at 30 145
	\pinlabel {$3$} at 42.5 179
	\pinlabel {$5$} at 18 198
	\pinlabel {$9$} at 42.5 116
	\pinlabel {$11$} at 16 97
	%next
	\pinlabel {$2$} at 350.5 199 %\pinlabel {$2$} at 363 200
	\pinlabel {$3$} at 339 187 %\pinlabel {$3$} at 338 179
	\pinlabel {$4$} at 339 208.5 %\pinlabel {$4$} at 338 220
	\pinlabel {$5$} at 327 199  %\pinlabel {$5$} at 314 200 
	\pinlabel {$8$} at 350 97
	\pinlabel {$9$} at 338 107
	\pinlabel {$10$} at 338 86
	\pinlabel {$11$} at 326 97
	\endlabellist
	\includegraphics[scale=0.64]{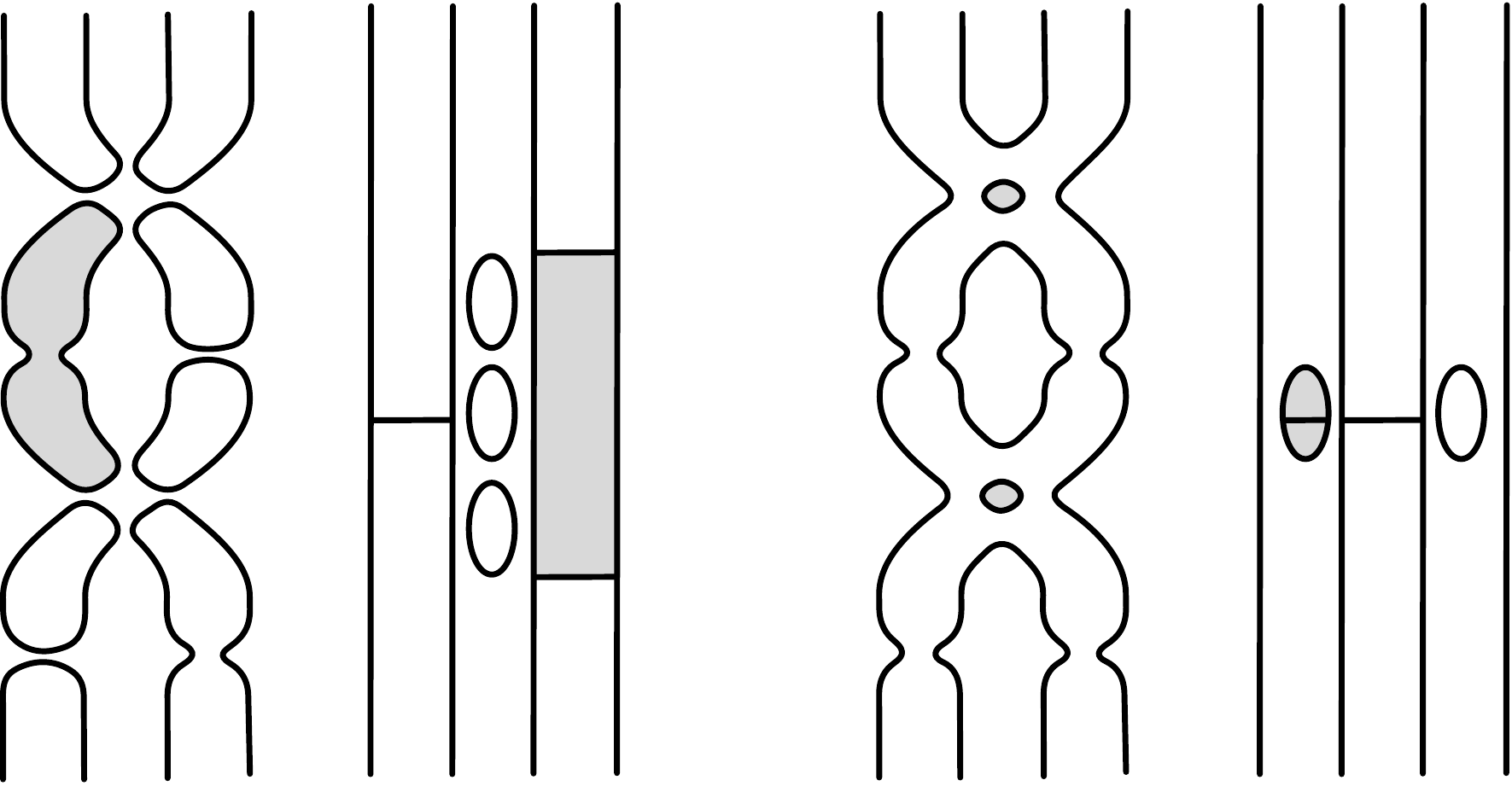}
	\caption{Two smoothings of $D_2$, together with the relevant cobordisms contributing to the map $f_0$.  On the left, the smoothing $010011001101$ contains four circles, one of which is the result of a pair of 1-handle attachments landing in the cycle specified by $(1,5,3,1,9,11)^*$. On the right, the smoothing $001100110011$ has two of three circles resulting from a birth followed by a 1-handle attachment; these are the target cycles $(2,3,5,4)$ and $(8,9,11,10)$.}
	\label{fig:sample-monster-cobordisms}
\end{figure}

Work of Khovanov \cite{kh1} and of Bar-Natan \cite{bncob} has given us explicit homotopy equivalences between the Khovanov cochain complexes of diagrams differing by Reidemeister moves.  Composing these gives rise to a chain homotopy equivalence
\[ f_0 : \CKh(D_1) \rightarrow \CKh(D_2) {\rm ,} \]
which may be represented, following Bar-Natan, by a matrix (in this case, a vector) of linear combinations of cobordisms.  Each such cobordism goes from the crossingless tangle $D_1$ to some smoothing of $D_2$ that lies in cohomological degree $0$ of the Bar-Natan complex associated with $D_2$.

We have computed this map $f_0$ and describe the result it as follows:  Each non-zero entry of the matrix contains the target smoothing of $D_2$.  In some cases, just knowing the target smoothing and that the matrix entry is non-zero is enough information to determine the matrix entry.  This is because our coefficient ring is $\bF$ and because each such cobordism has the same number of even-index handles as it has $1$-handles.  This last constraint comes from the fact that $f_0$ preserves the quantum degree. In those cases where it is not clear what the matrix entry should be we give an explicit picture of the cobordism.  Lines indicate the cores of attaching $1$-handles while $0$-handle additions are represented by circles.

We shall denote a smoothing of $D_2$ by a $12$-digit binary number.  The $i$th digit of the number indicates whether the $i$th crossing of $D_2$ is replaced by a $0$- or by a $1$-smoothing.  In those cases in which we draw the cobordism schematically, to determine the cobordism up to diffeomorphism rel boundary we also need to be specific about circles of the smoothing.  A circle component of a smoothing of $D_2$ may be identified by a cyclic sequence of numbers corresponding to crossings of $D_2$.  In the cases that arise, it is enough for us to indicate which circles do not bound discs in the cobordism and, of these, which circles are in the same component of some point of the initial (identity braid) smoothing of the cobordism.  We have listed such circles and put an asterisk next to those that are in the same component of some point of the initial smoothing.  See Figure \ref{fig:sample-monster-cobordisms} for a graphical explanation of the notation in some specific examples that arise. 

The following is a table in which each binary number corresponds to a non-zero entry of the matrix describing $f_0$.

\begin{tcolorbox}[title={In this table we give the components of $f_0$.}]
	010101010101-100110100110;
	010101010011-100110110010;
	001011001011-111000111000;
	001011001101-111000101100;
	001011010011-111000110010;
	001011010101-111000100110;
	010011010011-110010110010;
	010011010101-110010100110;
	010011001011-110010111000;
	001101001101-101100101100;
	001101001011-101100111000;
	001101010101-101100100110;
	001101010011-101100110010;
	010101001101-100110101100;
	010101001011-100110111000;
	%NEXTSECTION
	000010111101-010000101111;
	011100100011-001110110001;
	001010110101-011000100111;
	000110111001-010100101011;
	010101110001-100110100011;
	010111100001-110110100001;
	000001101111-100000111101;
	000011101101-110000101101;
	001001101011-101000111001;
	001001110011-101000110011;
	010001100111-100010110101;
	010001110011-100010110011;
	010001101011-100010111001;
	010011100101-110010100101;
	000101101011-100100111001;
	000111101001-110100111001;
	001101101001-101100101001;
	001101110001-101100100011;
	010101100011-100110110001;
	010101101001-100110101001;
	110101100001-100111100001.
\tcblower
		010011001101 $\cobOIOOIIOOIIOI$ (1,5,3,1,9,11)* - 110010101100 $\cobIIOOIOIOIIOO$ (2,3,6,9,8,6)*.
		
\DrawLine

		100001101101; 
		101001101001; 
		100101101001; 
		101101100001; 
		000000111111; 
		011110100001; 
		001000110111; 
		010010101101; 
		011010100101; 
		000100111011; 
		010110101001;
		101101010010.
		
\DrawLine

		001100110011 $\cobOOIIOOIIOOII$ (2,3,5,4), (8,9,11,10).
		
\DrawLine

100111010010;
101101010100; 
100111010100; 
111001010010; 
111001010100; 
101101001010; 
100111001010; 
111001001010; 
101101100100; 
100111100100; 
111001100100; 
111001101000; 
101101101000; 
100111101000; 
%NEXTBIT.
110001100101; 
100011101001; 
100011001011; 
100011010011; 
100011010101; 
110001010011; 
110001010101; 
110001001011.

\DrawLine

110011010010 $\cobIIOOIIOIOOIO$ (2,3,6)*; 
110011010100 $\cobIIOOIIOIOIOO$ (2,3,6)*; \\
101101001100 $\cobIOIIOIOOIIOO$ (1,9,11), (8,10,12); 
100111001100 $\cobIOOIIIOOIIOO$ (1,9,11), (8,10,12); 
110011001010 $\cobIIOOIIOOIOIO$ (2,3,6)*; 
111001001100 $\cobIIIOOIOOIIOO$ (1,9,11), (8,10,12); \\
110011001100 $\cobIIOOIIOOIIOO$ (1,9,11), (8,10,12), (2,3,6)*; 
110011100100 $\cobIIOOIIIOOIOO$ (2,3,6)*; 
110011101000 $\cobIIOOIIIOIOOO$ (2,3,6)*; 
100011001101 $\cobIOOOIIOOIIOI$ (1,9,11)*;\\
110001001101 $\cobIIOOOIOOIIOI$ (1,9,11)*.

\end{tcolorbox}

Although $f_0$ commutes with the differential $d$, it does not commute with the differential $\partial$, and hence does not give a cochain map $\CKht^j(D_1) \rightarrow \CKht^j(D_2)$.  The failure to commute with $\partial$ is measured by the map
\[ \dtot f_0 + f_0 \dtot = \tau f_0 + f_0 \tau \co \CKht^j(D_1) \rightarrow \CKht^j(D_2){\rm .} \]

The first two rows of the table describing $f_0$ contain pairs of components of $f_0$ which are conjugate under $\tau$.  The third and fourth rows contain components of $f_0$ which commute with $\tau$.  The fifth and sixth rows contain all remaining components of $f_0$ - it is these components that contribute to the failure of $f_0$ to commute with $\partial$.

We are now going to give a map $f_1 \co \CKht^j(D_1) \rightarrow \CKht^j(D_2)$ which strictly lowers the $\F$-filtration such that we have both
\[ d f_1 + f_1 d = d_\tau f_0 + f_0 d_\tau \,\,\, {\rm and} \,\,\, d_\tau f_1 + f_1 d_\tau = 0 {\rm .} \]

This therefore gives a cochain map
\[ f  := f_0 + f_1 \co \CKht^j(D_1) \rightarrow \CKht^j(D_2) {\rm } \]
which induces a composition of Khovanov isomorphisms
\[ f^2_\F \co E^2_\F(D_1) \rightarrow E^2_\F(D_2) {\rm .} \]

We begin by giving two tables, defining the maps
\[ f'_1, f''_1 \co \CKht^j(D_1) \rightarrow \CKht^j(D_2) {\rm .}\]
We follow the same conventions as when describing $f_0$.  The reader will note that $f'_1$ and $f''_1$ map into $i$-degree $-1$.  In order that both preserve the quantum degree, each component of $f'_1$ and $f''_1$ shall have one more even handle than it has 1-handles.

\begin{tcolorbox}[title={The map ${f_1}'$.}]
	
	110001100100;
	110001101000;
	110001010010;
	110001010100;
	110001001010;
	110101001000;
	110101000100.
	\tcblower
	
	110001001100 $\cobIIOOOIOOIIOO$ (1, 9, 11), (8,10,12).
\end{tcolorbox}

\begin{tcolorbox}[title={The map ${f_1}''$.}]
	
	101101001000;
	101101000100.
	
\tcblower

	110011001000 $\cobIIOOOIOOIIOO$ (1,3,5), (2,3,6); 
	110011000100 $\cobIIOOOIOOIIOO$ (1,3,5), (2,3,6).

\end{tcolorbox}

We define
\[ f_1 = (\tau f'_1 + f'_1 \tau) + \tau f''_1 \]
(note that $\tau f''_1 = f''_1 \tau$).  It can then be checked that $f_1$ has the properties that we required of it.

Now, proceeding as for the moves M1 and M2, we wish to find a cochain map $g$, homotopic to $f$, that preserves the $\G$ filtration and that induces an isomorphism $g^3_\G \co E^3_\G(D_1) \rightarrow E^3_\G(D_2)$.  We begin by giving two maps $h_0, h_{\pm 1} \co \CKht^j(D_1) \rightarrow \CKht^j(D_2)$.  Note that $h_0$ maps into the degree $i=0$ summand of $\CKht^j(D_2)$ while $h_{\pm 1}$ maps into the degree $i = \pm 1$ summands.

\begin{tcolorbox}[title={The map $h_{\pm 1}$.}]
	
	110001100100;
	110001101000;
	110001010010;
	110001010100;
	110001001010;
	110101001000;
	110101000100;
	110001110011.
	
	\tcblower
	
	110001001100 $\cobIIOOOIOOIIOO$ (1, 9, 11), (8,10,12);
	110011001000 $\cobIIOOOIOOIIOO$ (2,3,6), (7,10,12,8,10,11);
	110011000100 $\cobIIOOOIOOIIOO$ (1,3,5), (1,9,6,8,9,11);
	110011010011 $\cobD$ (2,3,6)*.
	
\end{tcolorbox}

\begin{tcolorbox}[title={The map $h_0$.}]
	
		010101010011;
		001011010011;
		010011010011;
		010011010101;
		010011001011;
		010101110001;
		010111100001;
		001001110011;
		010001100111;
		010001110011;
		010001101011;
		010011100101;
		010101100011;
		110101100001;
		100111010010;
		100011101001;
		100011001011;
		100011010011;
		100011010101.
	\tcblower
		110011100100 $\cobA$ (2,3,6), (1,9,6,8,9,11);
		110011101000 $\cobB$ (1,3,5)*;\\
		110011010010 $\cobC$ (1,3,5), (1,9,6,8,12,10,7,11);\\
		110011010100 $\cobA$ (2,3,6), (1,9,6,8,12,10,8,9,11);\\
		110011001010 $\cobA$ (1,3,5), (1,9,11,10,8,12,10,7,11).
	
\end{tcolorbox}

We set $h = \tau h_0 + h_{\pm 1}$.  It may then be checked that $g := f + \dtot h + h \dtot\co \CKht^j(D_1) \rightarrow \CKht^j(D_2)$ preserves the $\G$ filtration.

It remains to show that $g^3_\G \co E^3_\G(D_1) \rightarrow E^3_\G(D_2)$ is an isomorphism.  When working with the M1 and M2 moves we relied on an understanding of the $E^2_\G$ page when we reached the corresponding stage of the argument.  In the case of the M3 move we shall make use of a shortcut based on the fact that the tangle $D_1$ is crossingless.

We start by defining a cochain map $\overline{g} \co \CKht^j(D_2) \rightarrow \CKht^j(D_1)$ as in Figure \ref{fig:double-monster}, and explained in detail in the caption.
%The map $\overline{g}$ is a composition of $g$ together with a sequence of cochain maps previously constructed which preserve the $\G$ filtration for on-axis and off-axis involutive Reidemeister moves.
\begin{figure}[ht]
	\labellist
	\small
	\pinlabel {$g'$} at 125 111
	\pinlabel {$p$} at 280 109
	%\pinlabel {$\tau$} at 475 54
	\endlabellist
		\includegraphics[scale=0.5]{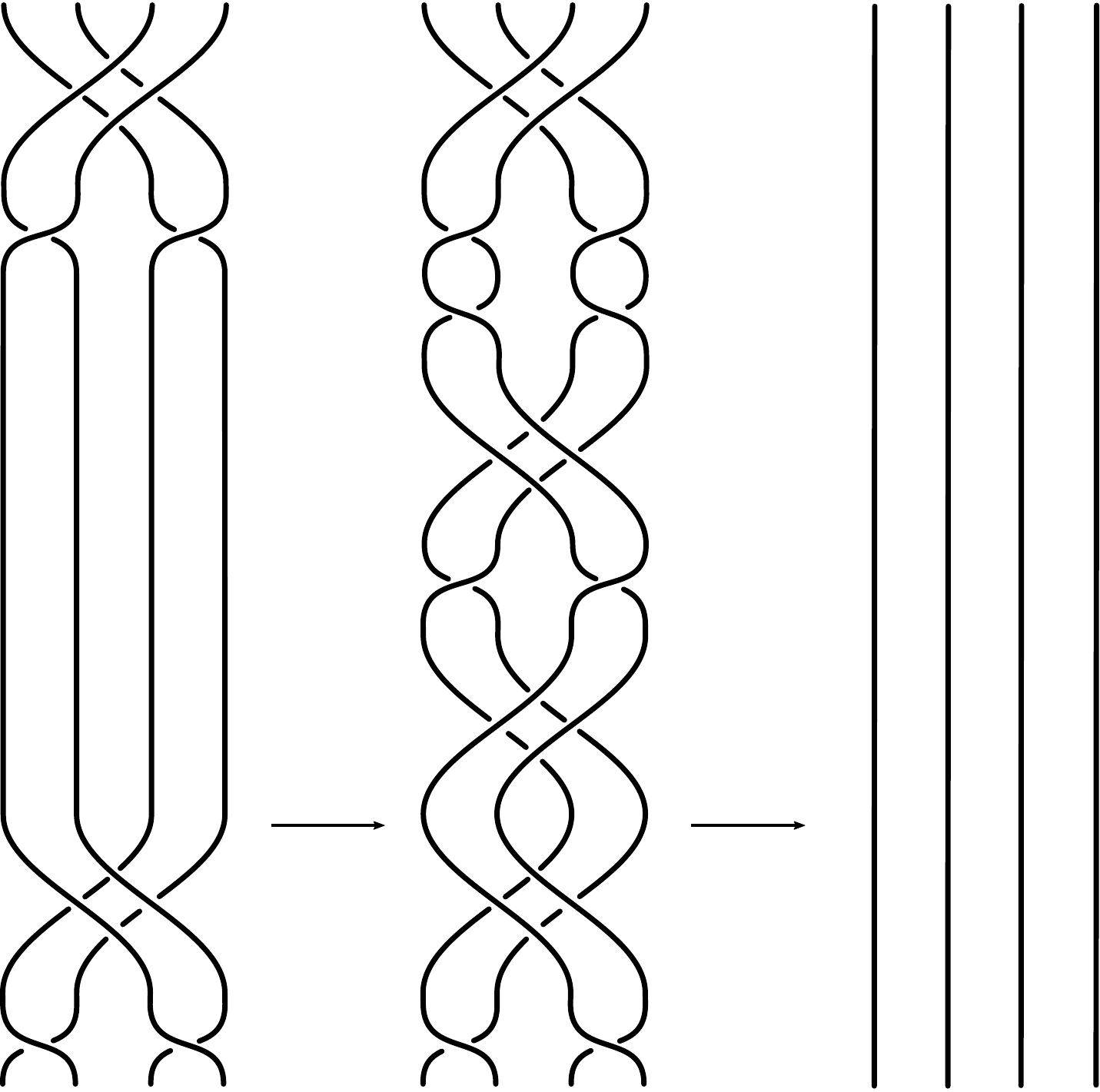}
	\caption{We construct the map $\overline{g} \co \CKht^j(D_2) \rightarrow \CKht^j(D_1)$.  Writing $D'_{3/2}$ for the middle diagram, notice (by turning the page upside down) that $D'_{3/2}$ differs from $D_2$ by an M3 move.  Therefore, by the construction above, we have a cochain map $g' \co \CKht^j(D_2) \rightarrow \CKht^j(D'_{3/2})$.  Furthermore, by composing a number of the cochain maps associated to the off-axis and on-axis involutive Reidemeister moves, there is a cochain map $p \co \CKht^j(D_{3/2}) \rightarrow \CKht^j(D'_{1})$.  We define~$\overline{g} = p \circ g'$.}
	\label{fig:double-monster}
\end{figure}
Then we consider the composition
\[\overline{g} \circ g \co \CKht^j(D_1) \rightarrow \CKht^j(D_1) {\rm .} \]
We note that the map $\overline{g} \circ g$ is given as a composition of cochain maps.  Each component of these cochain maps is described by a sum of tangle cobordisms between some smoothings of a $4$-braid.  On the other hand, the map $\overline{g} \circ g$ is described by a sum of cobordisms from the trivial $4$-braid to itself.

By referring to the constituent cochain maps, one may check that there is a contribution to $\overline{g} \circ g$ given by the composition of identity cobordisms on the trivial $4$-braid.  Furthermore, this has to be the only contribution, since any other contribution necessarily factors through a non-braid-like smoothing, and such a contribution must contain closed components if it is to preserve the quantum degree.

Hence we have determined that the map $\overline{g} \circ g$ is the identity map on complexes.  So we can conclude in particular that $g^3_\G$ is an injection, and ${\overline{g}}^3_\G$ is a surjection.
%$g$ induces an injection on the $E^3_\G$ page, and that $\overline{g}$ induces a surjection on the $E^3_\G$ page.
Furthermore, since $\overline{g}$ was defined as a composition of maps each of which (we now know) is injective on the $E^3_\G$ page, ${\overline{g}}^3_\G$ is also injective and hence an isomorphism.  Therefore $g^3_\G$ is an isomorphism as well.
\end{proof}

\section{Properties and examples}\label{sec:examples}%examples.tex

We conclude with some calculations in order to establish Theorem \ref{thm:mutation} and illustrate some of the properties of our refinement. We would like to be able to apply the invariant without being required to calculate it in full. To this end, before turning to examples, it is worth considering what we can determine about the support of the invariant $\Khredt(K)$ given a choice of strongly invertible diagram $D$ representing the strongly invertible knot $K$.  It is immediate from the definition of the $\G$-filtration on the underlying cochain complex $\CKht^j(D)$ that  \[k_{\min}=\textstyle-\frac{1}{2}n_-^E(D)-n_-^A(D) \quad\text{and}\quad k_{\max}=\frac{1}{2}n_+^E(D)+n_+^A(D)\]
provide absolute lower and upper bounds, respectively, on the possible $k$-gradings supporting $\Khredt(K)$. Here we denote the number of positive crossings in $D$ by $n_+(D)=n_+^A(D)+n_+^E(D)$ where $n_+^A$ counts positive on-axis crossings and $n_+^E$ counts off-axis positive crossings. Similarly, $n_-(D)=n_-^A(D)+n_-^E(D)$ where $n_-^A$ counts negative on-axis crossings while $n_-^E$ counts off-axis negative crossings. (This is a direct analogue of the fact that $i_{\min}=n_-(D)$ and $i_{\max}=n_+(D)$ provide {\em a priori} bounds on the $i$-grading of $\Khred(K)$.) This observation gives rise to the following useful trick:

\begin{lemma}[\bf The support lemma]\label{support_lemma}
Fix an involutive diagram $D$ with $m$ negative crossings and $n$ positive crossings for some strongly inversion of a knot $K$. If $\Khred(K)$ is non-trivial in grading $i=-m$ then $\Khredt(K)$ is non-trivial in bigrading $(i,k)=(-m,k_{\min})$ and, similarly, if $\Khred(K)$ is non-trivial in grading $i=n$ then $\Khredt(K)$ in non-trivial in bigrading $(i,k)=(n,k_{\max})$.
\end{lemma}
\begin{proof}
We prove the negative case and leave the positive case, which is similar (or derivable via Proposition \ref{prp:mirror}), to the reader.

Let $D$ be an involutive diagram for $K$ containing $m$ negative crossings and $n$ positive crossings.  Then the cochain group $\CKhred{}^{-m,*}(D)$
has a basis consisting of all decorations of the all-zero resolution $S_0$ of $D$, in which the basepointed circle receives the decoration $v_-$.  Consider the cochain $\zeta=v_-^{\otimes|S_0|} \in \CKhred_\tau(D)$---that is, the cochain obtained by decorating each of the $|S_0|$ circles of the all-zero resolution with a $v_-$.  If $\Khred(K)$ is non-trivial in grading $i=-m$, this means that the Khovanov differential $d \co {\CKhred{}^{-m,*}(D)} \rightarrow {\CKhred{}^{-m+1,*}(D)}$ is not injective.  As a result, every resolution of $D$ which contains exactly one $1$-resolution has one fewer circle than $S_0$.

It follows that $\zeta$ is the unique non-zero cochain of quantum grading $n-2m-|S_0|+1$ (note in fact that every other non-zero cochain must have a strictly higher quantum grading).  Therefore 
\[ \Khred_\tau^{-m, n-2m-|S_0|+1, k_{\min}}(K) = \bF\]
since $\zeta$ is of homogeneous bigrading $(i,k)=(-m,k_{\min})$.
\end{proof}
%is a cocycle for $d$ representing a non-zero cohomology class.
%Since $S_0$ is invariant under $\tau$, we see that $\zeta$ is also a cocycle for the total differential $\partial$.  Since it is clearly not cohomologous to any other element of $\CKhred_\tau(D)$ and is of homogeneous bigrading $(i,k)=(-m,k_{\min})$, we have proved the result.
%is generated by the all-zero resolution $S_0$ of $D$. Consider the cochain $\zeta=v_-^{\otimes|S_0|}$, that is, the enhanced state obtained by decorating each of the $|S_0|$ circles of the all-zero resolution with a $-$. Since $\Khred{}^{-m,*}(K)$ is non-trivial by assumption, we have that $\zeta$ is in the kernel of $d_{\Kh}$ so that $[\zeta]\in\Khred{}^{-m,*}(K)$. In order to show that $\partial(\zeta)=0$, it is enough to consider $d_\tau(\zeta)$; since $S_0$ is itself an involutive diagram, $\tau(\zeta)=\zeta$ and we have $d_\tau(\zeta)=0$. As a result, $m=n_-^E(D)+n_-^A(D)$ and $[\zeta]\ne0$ is supported in $k$-grading $k_{\min}$.

\subsection{Detecting mutation.} We now give a detailed account of the calculation proving Theorem~\ref{thm:mutation}, revisiting the pair from Figure \ref{fig:mutants-small}. 
\begin{figure}[ht]
\includegraphics[scale=0.75]{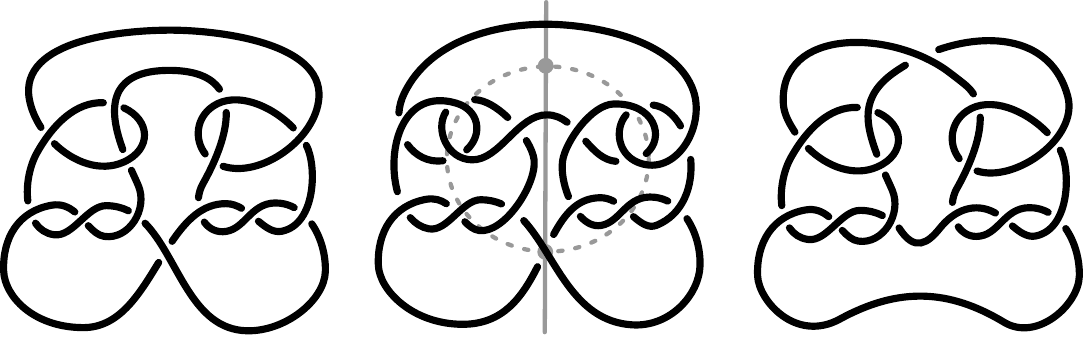}
\caption{The knot $K$ admits two distinct strong inversions (shown on the left and right), which generate the periodic symmetry of the centre diagram having fixed point set the dashed curve.}\label{fig:mutants-symmetries}
\end{figure}
We will denote by $K$ the knot before mutation, shown in Figure \ref{fig:mutants-symmetries}. Since this is a non-torus alternating knot, it must be hyperbolic and so admits at most two strong inversions \cite{Sakuma1986}. The symmetries are illustrated in Figure \ref{fig:mutants-symmetries}. Each of these diagrams has four negative crossings and, being alternating, it follows that $\Khred(K)$ must be non-trivial in $i=-4$ . Now apply the support lemma to the left and right diagrams of Figure \ref{fig:mutants-symmetries}:  in each diagram all the negative crossings are off-axis, so that $n_-(D)=n_-^E(D)=4$ and $\Khredt(K)$ must be non-trivial in $k=-2$. 

\begin{figure}[ht]
\includegraphics[scale=0.75]{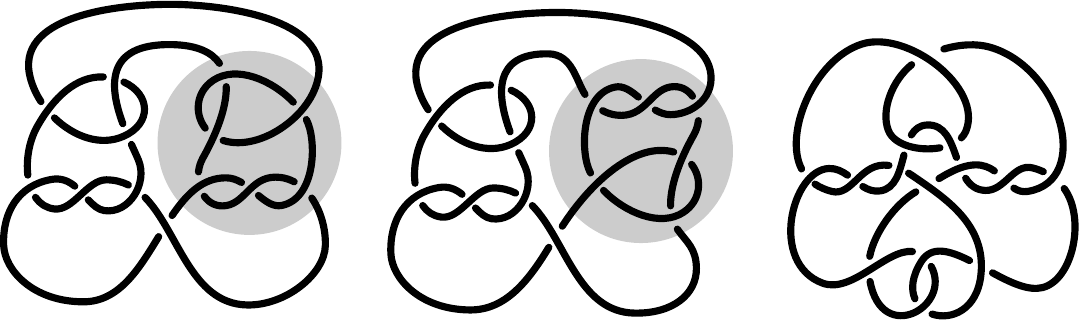}
\caption{A diagram for $K$ (far left), its mutant $K^\mu$ (centre), and a symmetric diagram $D^\mu$ for $K^\mu$ exhibiting a strong inversion (right).}\label{fig:mutants-detail}
\end{figure}

We now consider the knot $K^\mu$, obtained from $K$ by mutation as illustrated in Figure \ref{fig:mutants-detail}. Note that the strong inversion on $K^\mu$ is not immediately obvious and the minimality of the diagram is lost. Let $D^\mu$ be the strongly invertible diagram of the mutant $K^\mu$ shown, and note that there are once again four negative crossings, although this time two are on-axis and two off-axis. Since the unrefined Khovanov cohomology of $K^\mu$ agrees with that of $K$ \cite{Bloom2010,Wehrli2010}, we conclude that $\Khred(K^\mu)$ must be supported in $i=-4$. However, we verify that we have that $\Khredt(K^\mu)$ is non-trivial in $k=-3$,  %$4=n_-(D^\mu)=n_-^A(D^\mu)+n_-^A(D^\mu)=2+2$ and hence $k_{\min}=-1-2=-3$
by applying the support lemma.  So we see that while $\Khredt(K)$ is supported in $k\ge-2$ (for either of the two possible strong inversions) we have that $\Khredt(K^\mu)$ is non-trivial in $k=-3$. 

This completes the proof of Theorem \ref{thm:mutation}. We leave it to the reader to identify various locations that twists can be added to $K$ to produce infinitely many strongly invertible alternating knots, with strongly invertible mutants detected by the $k$-grading following an identical argument to the one given above. 

\subsection{Separating other pairs with identical Khovanov cohomology.} There are many other examples of pairs of knots with identical Khovanov cohomology, and we consider some sporadic examples that are known to be non-mutant and admit strong inversions. 

The first pair are examples of Kanenobu knots: the knots $8_8$ and $10_{129}$ have identical Khovanov cohomology (and moreover identical Khovanov-Rozansky cohomologies \cite{lobb4}). See \cite{Watson2007} for this pair in context, and for a proof that they cannot be related by mutation. Appealing to symmetries, we have:

\begin{proposition}\label{prp8_8-vs-10_129} The $k$-grading separates the knots $8_{8}$ and $10_{129}$.\end{proposition}

\begin{proof} Consulting Figure \ref{fig:8_8-vs-10_129} we can calculate that the invariant of the first diagram of $8_8$ is supported in $k$-gradings \[-2\le k \le 4\] while the invariant of the second diagram of $8_8$ is supported in \[-3\le k \le 4 {\rm .}\]
To see that these two strong inversions are distinct, observe that the latter diagram has non-trivial invariant in $k_{\min}=-3$ by appealing to the support lemma (the reader can find a non-involutive alternating diagram with $n_-=5$ and supporting ordinary Khovanov cohomology in $i=-5$). We can also check that the invariant (and the whole of the $\G$ spectral sequence) of the diagram shown for $10_{129}$ is supported in $k$-gradings
	\[-4\le k \le 3 {\rm .}\]
To complete the proof it is enough to check that the second diagram of $8_8$ has $E^3_\G$ page which is non-trivial in grading $k=4$ and that the invariant for $10_{129}$ is non-trivial in grading $k=-3$. Cocycles establishing these facts are given in Figure \ref{fig:8_8-vs-10_129}. \end{proof}

\labellist
\small
\pinlabel {$\boldsymbol +$} at 140 40
	\tiny
	\pinlabel {$+$} at 66 29 \pinlabel {$+$} at 119 31  \pinlabel {$+$} at 123 13 \pinlabel {$+$} at 100 106
	\pinlabel {$+$} at 215 35 \pinlabel {$+$} at 200 18   \pinlabel {$+$} at 194 106
	\pinlabel {$-$} at 273 80 \pinlabel {$-$} at 348 80  
	\pinlabel {$-$} at 360 15 \pinlabel {$-$} at 300 23
	\endlabellist
\begin{figure}[h!]
\includegraphics[scale=0.75]{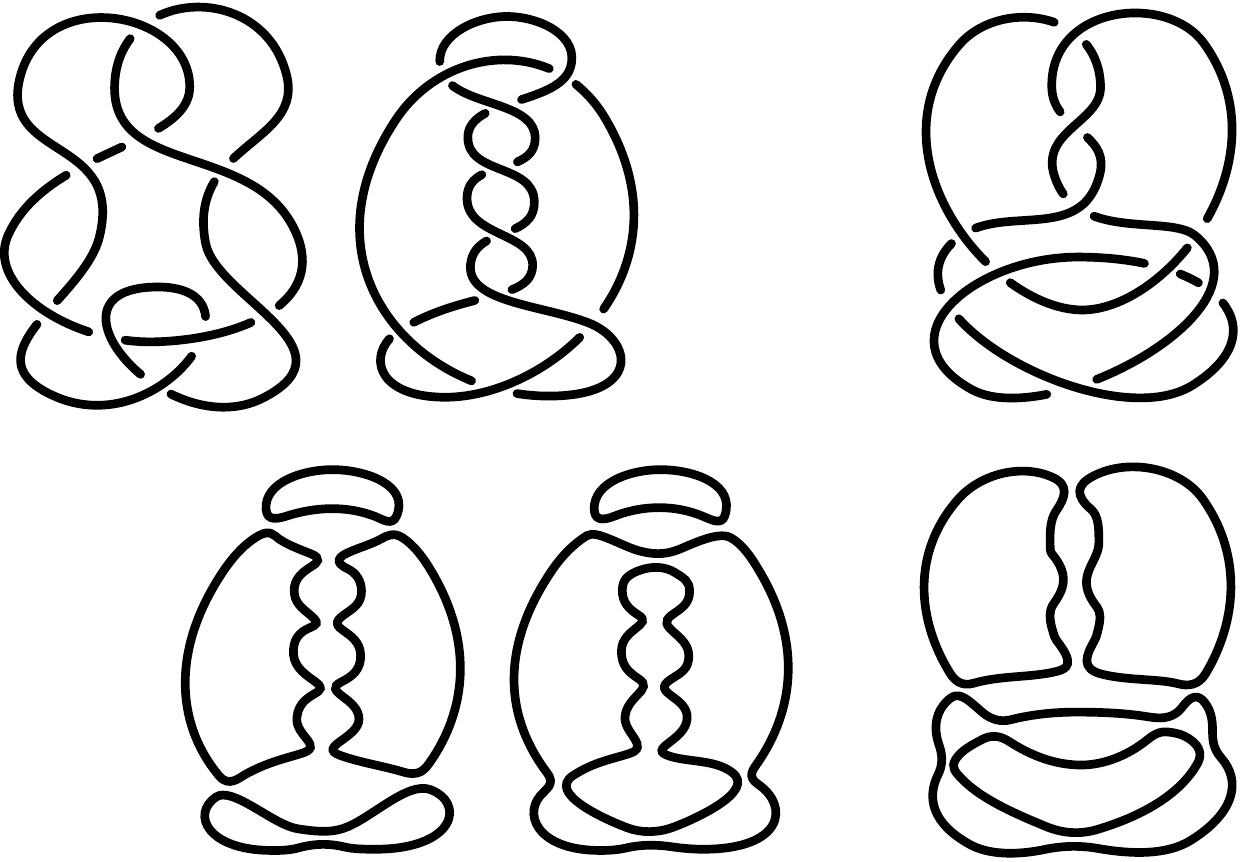}
\caption{The knots $8_8$ (top left and top centre) and $10_{129}$ (top right) share the same Khovanov cohomology, a fact that is explained in \cite{Watson2007}. The relevant enhanced states establishing the $k$-support are shown on the second line: call these cochains $\zeta_{8_8}$ (lower left) and $\zeta_{10_{129}}$ (lower right). Note that $k(\zeta_{8_8})=4=k_{\max}-1$ and $k(\zeta_{10_{129}})=-3=k_{\min}+1$, and check that $\zeta_{8_8}$ is in the kernel of $\partial^2_\G \co E^2_\G\to E^2_\G$ while $\partial(\zeta_{10_{129}})=0$.}\label{fig:8_8-vs-10_129}
\end{figure}

\labellist
\small
%\pinlabel {$+$} at 140 40
	\tiny
%	\pinlabel {$+$} at 66 29 \pinlabel {$+$} at 119 31  \pinlabel {$+$} at 123 13 \pinlabel {$+$} at 100 106
%	\pinlabel {$+$} at 215 35 \pinlabel {$+$} at 200 18   \pinlabel {$+$} at 194 106
%	\pinlabel {$+$} at 273 80 \pinlabel {$+$} at 348 80  
%	\pinlabel {$+$} at 360 15 \pinlabel {$+$} at 300 23
	\endlabellist
\begin{figure}[ht]
\includegraphics[scale=1]{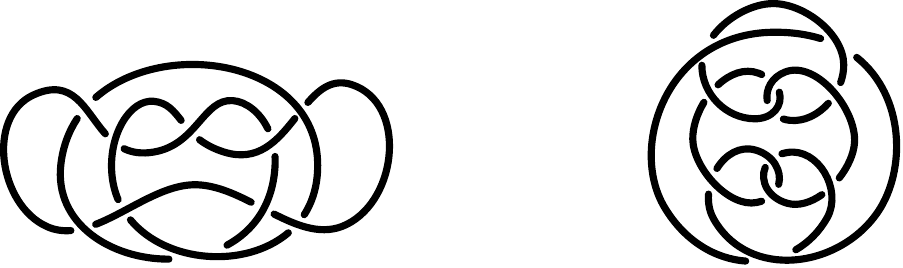}
\caption{The knots $10_{104}$ (left) and $10_{71}$ (right) share the same Khovanov cohomology.}\label{fig:10_104-vs-10_71}
\end{figure}

The knots $10_{71}$ and $10_{104}$ are alternating knots that share the same Jones polynomial and signature, so it follows that this pair cannot be separated by Khovanov cohomology. Note that this pair {\em is} separated by the Alexander polynomial (and hence they are not related by mutation) and that their common determinant is not a square (and hence they cannot be Kanenobu knots). On the other hand, they each admit a unique strong inversion \cite{Watson2017} so the refinement may be applied:

\begin{proposition}\label{prp:10_104-vs-10_71} The $k$-grading separates the knots $10_{71}$ and $10_{104}$. \end{proposition}

\begin{proof}Consult Figure \ref{fig:10_104-vs-10_71} and apply the support lemma: In the case of $10_{104}$, we have $n_\pm=n_\pm^A+n_\pm^E=1+4=5$ that $\Khredt(10_{104})$ is non-trivial in bigradings $(i,k)=\pm(5,3)$ and supported in $k$-gradings $-3\le k\le 3$. On the other hand, in the case of the non-minimal diagram of $10_{71}$ shown, we have that $n_+=n_+^A+n_+^E=3+2=5$ so the support lemma ensures that $\Khredt(10_{71})$ is non-trivial in bigrading $(i,k)=(5,4)$.  \end{proof}

\subsection{A remark on amphicheirality.}\label{sub:amph} Following Sakuma \cite{Sakuma1986} (see also \cite[Section 5]{Watson2017}), it is possible to use invariants of strong inversions to rule out amphicheirality. In particular:

\begin{proposition}[compare Sakuma {\cite[Proposition 3.4 (1)]{Sakuma1986}}] \label{prp:amph}
Let $K$ be an amphicheiral knot, and suppose that $K$ admits a unique strong inversion up to conjugacy. Then $\Kht^{i,j,k}(K^*) \cong \Kht^{-i,-j,-k}(K)$ where $K^*$ denotes the strongly invertible mirror image of $K$ obtained by reversing orientation on $(S^3, \tau)$ by reflecting in a sphere containing the axis $\boldsymbol a$. \qed
\end{proposition}

Thus, the absence of such a symmetry in the refinement establishes non-amphicheirality. This follows from the behaviour of the refined invariant under mirrors, which we now establish. 

\begin{proof}[Proof of Proposition \ref{prp:mirror}]
This follows along similar lines to the usual duality satisfied by Khovanov cohomology; see for instance \cite[Proposition 31]{kh1} or \cite[Proposition 3.10]{ras3}.
%As a first step, as in the usual setting, one observes from the definition that the claimed symmetry is enjoyed by the $j$ grading. Therefore, it is enough to show that, in any given $j$-grading, $\Kht^{i,k}(K^*) \cong \Kht^{-i,-k}(K)$ as claimed. To see this,

Let $D$ be an involutive diagram and $D^*$ the involutive mirror obtained by changing all the crossings of $D$ from under- to over-crossings.
Consider the vector space $\CKht^j(D)$ and, with respect to some ordering of the standard basis of decorated smoothings, express the differential $\partial \co \CKht^j(D) \rightarrow \CKht^j(D)$ as a square
%$n\times n$
matrix $M_\partial$.
%where $n$ is the dimension of $\CKht(K)$ in quantum grading $j$.
This matrix has entries in $\bF$, but we can refine it by replacing the non-trivial entries with a pair in $\Z\times\frac{1}{2}\Z$ representing the $(i,k)$ grading of the relevant coefficient map.

Now choose the dual basis for the vector space $\CKht^{-j}(D^*)$ (the smoothings of $D^*$ are the same as the smoothings of $D$, and we call two decorations of a smoothing \emph{dual} if the decorations differ on every component of the smoothing).  Then one observes that the transpose of $M_\partial$ agrees with the differential on $\CKht^{-j}(D^*)$. 
\end{proof}

Proposition \ref{prp:amph} can be useful even in settings where $K$ appears to be amphicheiral, namely, when the knot (ignoring the strong inversion) satisfies $\Kh(K)\cong\Kh(K^*)$ as bigraded vector spaces. Note that, as observed in \cite[Section 5]{Watson2017}, there are exactly three knots of 10 or fewer crossings for which non-amphicheirality is {\em not} certified by Khovanov cohomology: $\{10_{48}, 10_{71},10_{104}\}$. Khovanov cohomology may be used, by way of a different construction appealing to the unique symmetry in each case, to certify that these three knots are indeed not amphicheiral \cite{Watson2017}. Similarly, consulting the proof of Proposition \ref{prp:10_104-vs-10_71}, we see that 
\[\Kht^{i,j,k}(10_{71})\ncong \Kht^{-i,-j,-k}(10_{71})\]
despite the fact that $\Kh^{i,j}(10_{71})\cong \Kht^{-i,-j}(10_{71})$ for all pairs $(i,j)$. In particular, we can calculate that the invariant is supported on $-3\le k \le 4$ (using the symmetric diagram in Figure \ref{fig:10_104-vs-10_71}) and we checked that the invariant is non-trivial in $k=4$ using the support lemma. 

\subsection{Small torus knots} To this point, we have applied our refinement of $\Khred(K)$ to strongly invertible knots $K$ without calculating the associated triply graded invariant $\Khredt(K)$ in full. While the ability to apply a cohomology theory by only appealing to part of the invariant is a desirable property, we would like to supply the reader with some basic calculations of   $\Khredt(K)$ for small knots $K$.  As a first step, consider the case of torus knots; recall that torus knots admit unique strong inversions \cite{Schreier1924}.  Fix the notation $T_{p,q}$ for $p,q$ relatively prime and positive -- these are the positive torus knots. For instance, $T_{2,3}$ is the right-hand trefoil knot or $3_1$ in Rolfsen's table. Indeed, the torus knot $T_{2,n}$ can be identified (up to mirrors) with the knot $n_1$, explaining the omission of these knots in Figure \ref{fig:table}. Recalling Khovanov's initial calculation of $\Kh(T_{2,n})$ in which the skein exact sequence makes an implicit appearance \cite{kh1}, we make a brief digression on the invariant $E_\mathcal{G}^3(K)$ for positive involutive links. 

Suppose that $D$ is a positive diagram and fix a distinguished equivariant pair of off-axis crossings. If $D_0$ is the involutive diagram resulting from the simultaneous $0$-resolution of this pair and, similarly,  $D_1$ is the involutive diagram resulting from their simultaneous $1$-resolution, then there is a subcomplex $E_\mathcal{G}^2(D_1)\hookrightarrow E_\mathcal{G}^2(D)$ and a short exact sequence of cochain complexes \[0\to E_\mathcal{G}^2(D_1)\to E_\mathcal{G}^2(D)\to E_\mathcal{G}^2(D_0)\to 0\]
There are some grading shifts that need to be worked out in order to make use of this decomposition in practice. For the quantum grading $j$, the relevant shift is the same as for the usual skein exact sequence in Khovanov cohomology (applied twice). On the other hand, the shift in the cohomological grading $k$ on $E^2_\mathcal{G}$ differs slightly.  Using square brackets to denote a shift in the $k$-grading,
%then if $D_0$ is the oriented (equivariant) resolution of the involutive diagram $D$,
we have 
\[0\to E_\mathcal{G}^2(D_1)[\ell]\to E_\mathcal{G}^2(D)\to E_\mathcal{G}^2(D_0)\to 0\]
where $\ell=1+\frac{1}{2}\big(n_-^E(D_1)+n_-^{A_\mathbf{o}}(D_1)\big)+n_-^{A_\mathbf{r}}(D_1)$. As before, a distinction between on-axis and off-axis crossings must be made: Having made a choice of orientation for the strands affected by the resolution giving rise to $D_1$, $n_-^E(D_1)$ is the number of off-axis crossings carrying a negative sign, $n_-^{A_\mathbf{r}}(D_1)$ is the number of on-axis crossings carrying a negative sign for which the involution reverses orientation, and $n_-^{A_\mathbf{o}}(D_1)$ is the number of on-axis crossing carrying a negative sign for which the involution preserves orientation. Notice that, when $D_1$ is a diagram for a strong inversion, this latter value necessarily vanishes. We reiterate that this statement has been simplified by appealing to a positive diagram $D$; we leave the minor adjustments required for arbitrary diagrams to the reader. This culminates in a skein exact sequence of the form
\[
\begin{tikzcd}[column sep=0em]
& E_\mathcal{G}^3(D) \arrow{dr}{} \\
E_\mathcal{G}^3(D_1)[\ell] \arrow{ur}{}  && E_\mathcal{G}^3(D_0) \arrow{ll}{\delta}
\end{tikzcd}
\]
%\[\cdots\to E_\mathcal{G}^3(D_1)[\ell]\to E_\mathcal{G}^3(D)\to E_\mathcal{G}^3(D_0)\to \cdots\]
where the connecting homomorphism $\delta\co E_\mathcal{G}^3(D_0) \to E_\mathcal{G}^3(D_1)[\ell]$ increases the $k$-grading by 1.  Furthermore, since we are working over a field, this implies that there is a $(j,k)$-graded isomorphism $E_\mathcal{G}^3(D)\cong H^*(\operatorname{cone}(\delta))$. With minor adjustments to the discussion above, there is a similar skein exact sequence when an on-axis crossing is resolved.

\labellist
	\tiny
	\pinlabel {$k$} at 88 62
	\pinlabel {$i$} at 169 0
	\pinlabel {0} at 105 -5 \pinlabel {1} at 122 -5 \pinlabel {2} at 136 -5 \pinlabel {3} at 152 -5
	\pinlabel {0} at 87 13.5 \pinlabel {1} at 87 29 \pinlabel {2} at 87 45
	\endlabellist
\parpic[r]{
 \begin{minipage}{60mm}
 \centering
 \includegraphics[scale=0.65]{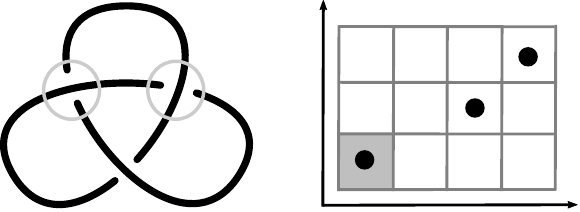}
 \captionof{figure}{$E^3_\mathcal{G}(T_{2,3})\cong \Khredt(T_{2,3})$.}
 \label{fig:trefoil-skein}
  \end{minipage}%
  }
This construction is best-illustrated in an explicit example. Consider the the trefoil as shown in Figure \ref{fig:trefoil-skein}; this diagram is $D$, and the unique equivariant pair of off-axis crossings is indicated. In this case, $E^3_\mathcal{G}(D_0)$ is a 1-dimensional vector subspace highlighted in grading $(i,k)=(0,0)$ and the two dimensional vector subspace that is not highlighted is  $E^3_\mathcal{G}(D_1)[\frac{3}{2}]$ since $n_-^{A_\mathbf{o}}(D_1)=1$ and $n_-^E(D_1)=n_-^{A_\mathbf{r}}(D_1)=0$. Perhaps the important thing to note here is that $D_1$ is an involutive link diagram, but it is not a diagram for a strongly invertible link in this case. For grading reasons, all the higher differentials in the $\mathcal{G}$-spectral sequence necessarily vanish, so we have in fact computed $E^3_\mathcal{G}(T_{2,3})\cong \Khredt(T_{2,3})$. The table uses $\bullet$ to indicate a copy of $\bF$. We will adhere to this convention for all of the calculations that follow, and we will also continue to record the $(i,k)$ bigrading in the plane and drop the labelled axis. For this particular example, there is no significant information lost in making this choice: recall that the invariant $\Khred(T_{2,3})$ is thin, that is, it is supported in a single diagonal in the $(i,j)$-plane. This diagonal, and hence the $j$-grading, is completely determined once we have the $s$ invariant in hand. In this case, $s(T_{2,3})=2$ so that the highlighted generator is in trigrading $(0,2,0)$. More generally, the relationship $s=2i-j$ holds for any thin knot. 

We observe that our refinement for the $(2,n)$-torus knots (relative to this unique strong inversion) have a particularly simple form. This is essentially an adaptation of Khovanov's original observation/calculation for this family; see Figure \ref{fig:small-torus-generic}. The reader is now equipped to induct in the number of equivariant pairs of crossings appealing to the skein exact sequence, with the trefoil discussed in detail above providing a base case. As before, for grading reasons, the $\mathcal{G}$-spectral sequence converges at the $E^3_\mathcal{G}$-page, so we are able to record $\Khredt(T_{2,n})$. 

\labellist
\small
%\pinlabel {$+$} at 140 40
	\tiny
	\pinlabel {$\frac{n+1}{2}$} at -14 122 
	%\pinlabel {$\cdot$} at -7 97, \pinlabel {$\cdot$} at -7 102, \pinlabel {$\cdot$} at -7 107
	%\pinlabel {$4$} at -7 85
	\pinlabel {$\vdots$} at -7 97
	\pinlabel {$3$} at -7 67
	\pinlabel {$2$} at -7 48
	\pinlabel {$1$} at -7 30
	\pinlabel {$0$} at -7 10
	\pinlabel {$0$} at 10 -7 
	\pinlabel {$1$} at 30 -7 
	\pinlabel {$2$} at 48 -7
	\pinlabel {$3$} at 67 -7
	\pinlabel {$4$} at 85 -7
	\pinlabel {$5$} at 105 -7
	\pinlabel {$\cdots$} at 134 -7
	\pinlabel {$n$} at 160 -7
	\endlabellist
\begin{figure}[ht]
\includegraphics[scale=0.5]{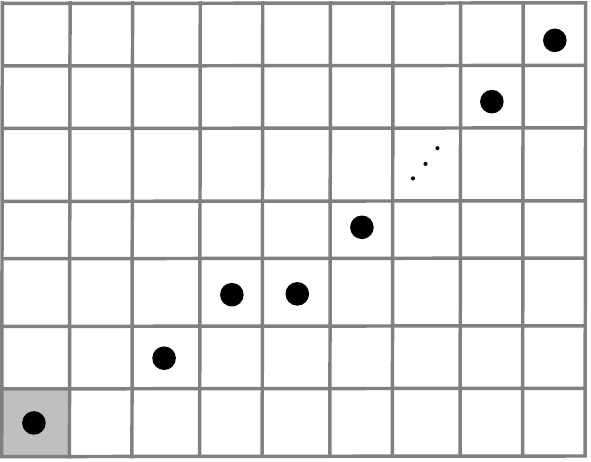}
\qquad
\includegraphics[scale=0.75]{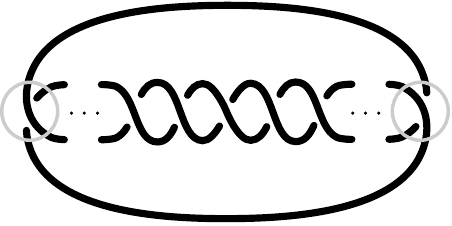}
\caption{The invariant $\Khredt(T_{2,n})$ for $n>0$ an odd integer (shown at left), that is, $T_{2,n}$ is a positive $(2,n)$-torus knot (shown at right, with the $\big(\frac{n-1}{2}\big)^{\!\text{th}}$ equivariant crossing-pair circled). 
%We have used $\bullet$ to indicate a copy of the ground field $\bF$. The $i$-grading is horizontal while the $k$-grading is vertical. 
The $j$-grading can be reconstructed from $j=2i+n-1$; recall that $s(T_{2,n})=n-1$ and the invariant $\Khred(T_{2,n})$ is supported in a single diagonal in the $(i,j)$-plane. }\label{fig:small-torus-generic}
\end{figure}

\labellist
	\tiny
	\pinlabel {0} at 145 -3 \pinlabel {1} at 162 -3 \pinlabel {2} at 176 -3 \pinlabel {3} at 192 -3 \pinlabel {4} at 207 -3 \pinlabel {5} at 223 -3
	\pinlabel {0} at 130 13.5 \pinlabel {1} at 130 29 \pinlabel {2} at 130 45 \pinlabel {3} at 130 61
	\endlabellist
\parpic[r]{
 \begin{minipage}{60mm}
 \centering
 \includegraphics[scale=0.65]{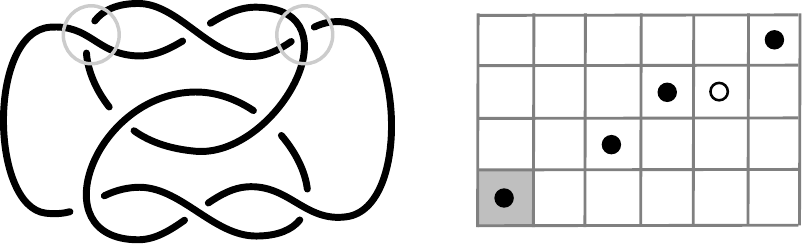}
 \captionof{figure}{$E^3_\mathcal{G}(T_{3,4})\cong \Khredt(T_{3,4})$.}
 \label{fig:8_19-skein}
  \end{minipage}%
  }
We will see further examples of our refinement for thin knots below; see Figure \ref{fig:table}. Before turning to these calculations, which also rely on the exact sequence described above, we should consider the first non-thin knots in the table: The knots $8_{19}$ and $10_{124}$ correspond to the torus knots $T_{3,4}$ and $T_{3,5}$, respectively. These enjoy the property of fitting into a larger (generically hyperbolic) family of positive pretzel knots. Consider the knots $P_{-2,3,q}$ for $q$ being an odd positive integer. The torus knots $T_{2,5}$, $T_{3,4}$, and $T_{3,5}$ correspond to the cases $q=1,q=3$, and $q=5$, respectively.  A diagram for the case $q=3$ is shown in Figure \ref{fig:8_19-skein}. Proceeding as before by resolving the pair of off-axis crossings shown, one checks that $E^3_\mathcal{G}(D_1)[\frac{5}{2}]\cong\bF_{(12,2)}\oplus \bF_{(14,3)}$ since $n_-^{A_\mathbf{o}}(D_1)=1$ and $n_-^E(D_1)=2$ (the shorthand $\bF_{(j,k)}$ indicates a copy of $\bF$ supported in bigrading $(j,k)$). On the other hand, $E^3_\mathcal{G}(D_0)\cong E^3_\mathcal{G}(T_{2,5)}$ has dimension 5, while $5=\dim\Khred(T_{3,4})\ge\dim\Khredt(T_{3,4})$. Consequently, there must be a non-trivial connecting homomorphism $\delta\co\bF_{(14,2)}\to \bF_{(14,3)}$   so that $E^3_\mathcal{G}(T_{3,4})\cong E^\infty_\mathcal{G}(T_{3,4})$ agrees with the dimension of $\Khred(T_{3,4})$. The latter invariant is supported in two diagonals, these being $j=2i+6$ (the $\bullet$ generators) and $j=2i+4$ (the $\circ$ generators). Note that $s(T_{3,4})=6$. Now for general $q>3$, including $T_{3,5} \simeq P_{-2,3,5}$, we proceed by induction as summarized in Figure \ref{fig:pretzels}; the reader will observe that the skein exact sequence splits in the inductive step.
\labellist
\small
%\pinlabel {$+$} at 140 40
	\tiny
	\pinlabel {$\frac{q+3}{2}$} at -10 125
	%\pinlabel {$\cdot$} at -7 97, \pinlabel {$\cdot$} at -7 102, \pinlabel {$\cdot$} at -7 107
	%\pinlabel {$4$} at -7 85
	\pinlabel {$\vdots$} at -7 106
	\pinlabel {$5$} at -7 83
	\pinlabel {$4$} at -7 69
	\pinlabel {$3$} at -7 53
	\pinlabel {$2$} at -7 39
	\pinlabel {$1$} at -7 24
	\pinlabel {$0$} at -7 9
	\pinlabel {$0$} at 9 -7 
	\pinlabel {$1$} at 25 -7 
	\pinlabel {$2$} at 40 -7
	\pinlabel {$3$} at 55 -7
	\pinlabel {$4$} at 70 -7
	\pinlabel {$5$} at 85 -7
	\pinlabel {$6$} at 100 -7
	\pinlabel {$7$} at 115 -7
	\pinlabel {$\cdots$} at 165 -7
	\pinlabel {$q\!+\!2$} at 198 -7
	\endlabellist
\begin{figure}[ht]
\includegraphics[scale=0.75]{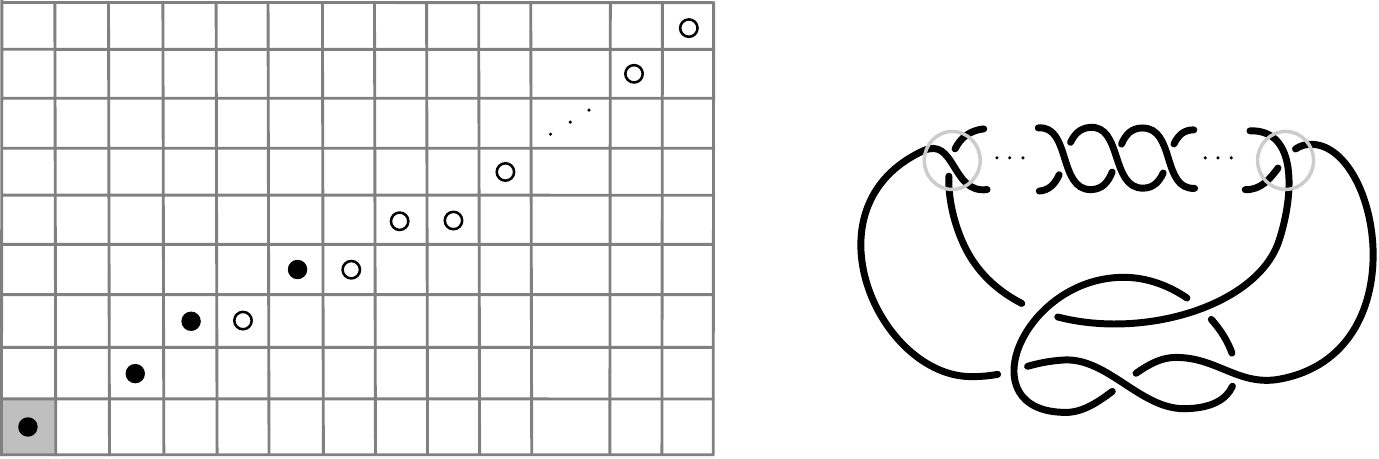}
\caption{The invariant $\Khredt(P_{-2,3,q})$ for odd integers $q>3$ (shown at left), that is, $P_{-2,3,q}$ is the (positive) $(-2,3,q)$-pretzel knot knot (shown at right). There are two supporting $(i,j)$-diagonals in general: We use $\bullet$ to indicate a copy of $\bF$ in $j=2i+s$ and $\circ$ to indicate a copy of $\bF$ in $j=2(i-1)+s$. Recall that $s(P_{-2,3,q})=q+3$. }\label{fig:pretzels}
\end{figure}

Apart from the appearance of the second diagonal, as expected, in the $(i,j)$-plane, this new family still enjoys the property that the $E^3_\mathcal{G}$ page agrees with the triply-graded invariant. This leads us to ask: 

\begin{question}Does the $\mathcal{G}$ spectral sequence always converge at the $E^3_\mathcal{G}$ page? \end{question}

\labellist
\small
\pinlabel {$s(4_1)=0$} at -80 985
\pinlabel {$s(5_2)=2$} at -80 910
\pinlabel {$s(6_1)=0$} at -80 820
\pinlabel {$s(6_2)=2$} at -80 730
\pinlabel {$s(6_3)=0$} at -80 655
\pinlabel {$s(7_2)=2$} at -80 550
\pinlabel {$s(7_3)=4$} at -80 445
\pinlabel {$s(7_4)=3$} at -80 345
\pinlabel {$s(7_5)=4$} at -80 240
\pinlabel {$s(7_6)=-2$} at -80 140
	\pinlabel {$s(7_7)=0$} at -80 35
	\endlabellist
\begin{figure}[h!]
\hspace{2cm}\includegraphics[scale=0.55]{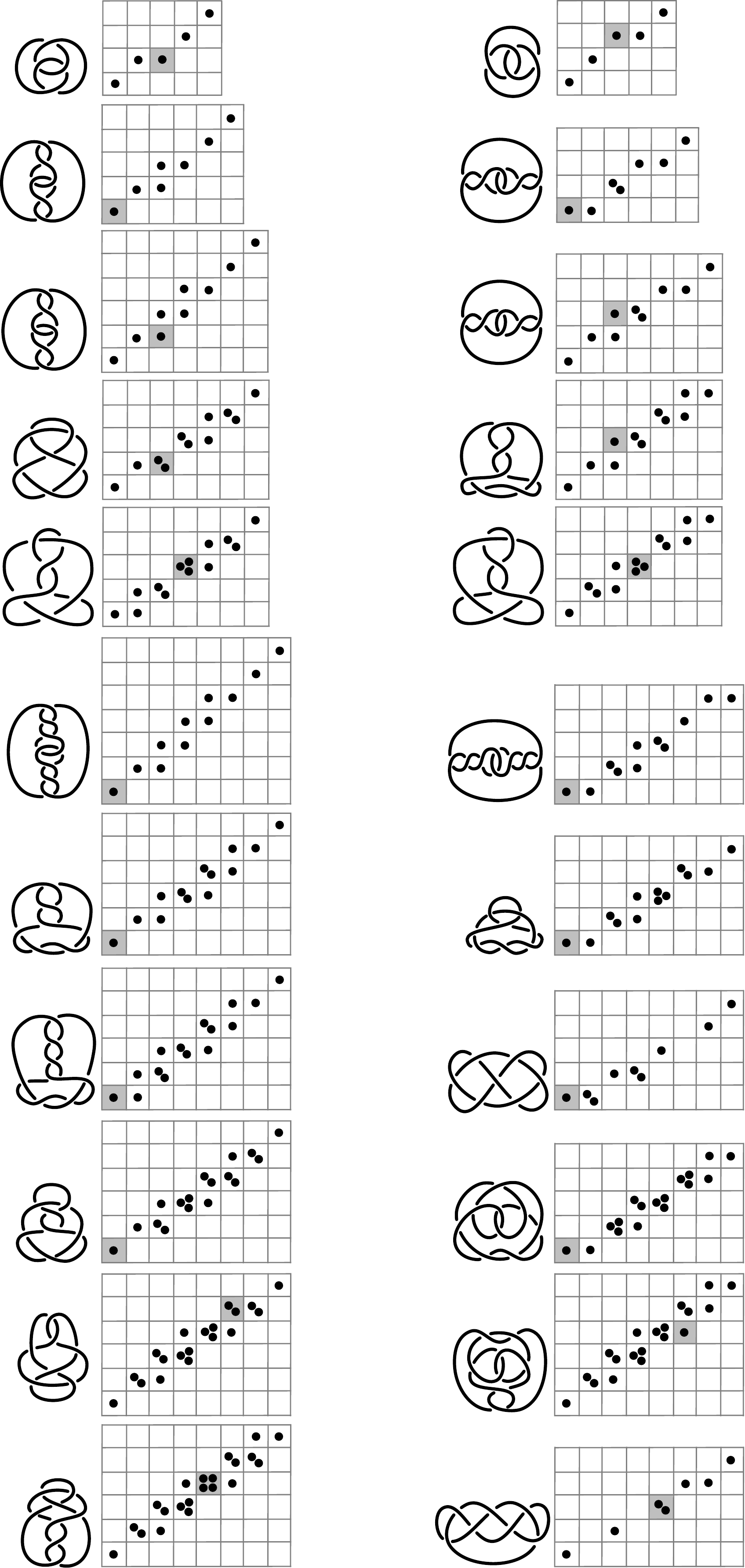}
\caption{Hyperbolic knots through 7 crossings: there are 11 examples, each of which admits a pair of strong inversions (this follows, for example, from the fact that each is a 2-bridge knot). The shaded box locates $(i,k)=(0,0)$ and the $j$ grading can be reconstructed from the $s$-invariant.}\label{fig:table}
\end{figure}

\subsection{A table of invariants of low-crossing hyperbolic knots}
We conclude with a table of invariants associated with prime knots of at most seven crossings (Figure \ref{fig:table}), omitting torus knots since these have been dealt with above. There are some observations worth making here. First, note that the knots $4_1$ and $6_3$ are amphicheiral, explaining the symmetry in the invariants in both cases; compare  \cite[Proposition 3.4 (2)]{Sakuma1986} and Proposition  \ref{prp:amph}. In both cases the symmetries are exchanged when taking the mirror. Second, the pair of symmetries on each knot is separated by the refinement. At present, we do not have an example of a prime knot admitting a pair of strong inversions that are not distinguished by this new invariant. Notice that the refinement can be only very mildly different on some pairs of strong inversions; see, in particular, $7_6$ in the table.

Finally, perhaps the most surprising behaviour is exhibited by the knots $7_4$ and $7_7$, where in each case one strong inversion has $d_\tau^\ast=0$ while the other has  $d_\tau^\ast\ne0$. This runs against the general principal that Khovanov invariants are relatively uninteresting for alternating knots.

%\newpage
\clearpage
\bibliographystyle{amsplain}

%\input{lobb_watson_involutive.bbl}
%\bibliography{works-cited}
%.
%..
\end{document}